\documentclass[pdflatex,sn-mathphys-ay]{sn-jnl}

\usepackage[utf8]{inputenc}
\usepackage{quiver}
\usepackage{mathrsfs}
\usepackage{hyperref}
\usepackage{amsfonts}
\usepackage{amsthm}
\usepackage{amssymb}
\usepackage{amscd}
\usepackage{amstext}
\usepackage{blkarray}
\usepackage{amsmath}
\usepackage{mathtools}
\usepackage{graphicx}
\usepackage[all]{xy}
\xyoption{2cell}
\UseTwocells
\usepackage{stackengine}
\usepackage{calrsfs}
\usepackage{framed}
\usepackage{url}
\usepackage{float}
\usepackage{tikz}
\usetikzlibrary{positioning,intersections,cd,matrix, arrows, backgrounds}
\usetikzlibrary{shapes,shapes.geometric,shapes.misc}
\usetikzlibrary{decorations}
\usetikzlibrary{decorations.pathreplacing}
\newcommand\Ar[3]{\ar[from={#1}, to={#2}, #3]}
\newcommand\Nname[1]{|[alias=#1]|}
\usepackage[OT2,T1]{fontenc}
\usepackage{enumitem}
\newcommand{\refitem}[1]{\hyperref[#1]{#1}}
\usepackage{anyfontsize}
\usepackage{aliascnt}
\usepackage[noabbrev,capitalise]{cleveref}
\let\emptyset\varnothing
\usepackage{booktabs}
\newtheorem{thm}{Theorem}[section]
\crefname{thm}{Theorem}{Theorems}

\newaliascnt{prp}{thm}
\newtheorem{prp}[prp]{Proposition}
\aliascntresetthe{prp}
\crefname{prp}{Proposition}{Propositions}

\newaliascnt{lem}{thm}
\newtheorem{lem}[lem]{Lemma}
\aliascntresetthe{lem}
\crefname{lem}{Lemma}{Lemmas}

\newaliascnt{cor}{thm}
\newtheorem{cor}[cor]{Corollary}
\aliascntresetthe{cor}
\crefname{cor}{Corollary}{Corollaries}

\theoremstyle{definition}
\newaliascnt{dfn}{thm}
\newtheorem{dfn}[dfn]{Definition}
\aliascntresetthe{dfn}
\crefname{dfn}{Definition}{Definitions}

\newaliascnt{exm}{thm}
\newtheorem{exm}[exm]{Example}
\aliascntresetthe{exm}
\crefname{exm}{Example}{Examples}

\newaliascnt{rmk}{thm}
\newtheorem{rmk}[rmk]{Remark}
\aliascntresetthe{rmk}
\crefname{rmk}{Remark}{Remarks}

\newaliascnt{ntn}{thm}
\newtheorem{ntn}[ntn]{Notation}
\aliascntresetthe{ntn}
\crefname{ntn}{Notation}{Notations}

\newaliascnt{cvn}{thm}

\aliascntresetthe{cvn}
\crefname{cvn}{Convention}{Conventions}

\crefname{equation}{equation}{equations}
\Crefname{equation}{Equation}{Equations}

\crefname{figure}{Fig.}{Figs.}

\newcommand\al{\alpha}
\newcommand\be{\beta}
\newcommand\ga{\gamma}
\newcommand\de{\delta}
\newcommand\ep{\varepsilon}
\newcommand\ze{\zeta}
\newcommand\et{\eta}
\renewcommand\th{\theta}
\newcommand\io{\iota}

\newcommand\la{\lambda}
\newcommand\ro{\rho}
\newcommand\si{\sigma}

\newcommand\ta{\tau}
\newcommand\up{\upsilon}

\newcommand\ps{\psi}

\newcommand\om{\omega}
\newcommand\Om{\Omega}

\newcommand\Ps{\Psi}

\renewcommand\top{\operatorname{top}}
\newcommand\Ab{\operatorname{Ab}}
\newcommand\soc{\operatorname{soc}}
\newcommand\Ker{\operatorname{Ker}}
\newcommand\Cok{\operatorname{Coker}}
\newcommand\cok{\operatorname{coker}}
\renewcommand\Im{\operatorname{Im}}

\newcommand\Hom{\operatorname{Hom}}

\newcommand\rad{\operatorname{rad}}

\newcommand\End{\operatorname{End}}

\renewcommand\mod{\operatorname{mod}}

\newcommand\supp{\operatorname{supp}}

\newcommand{\udim}{\operatorname{\underline{dim}}\nolimits}

\newcommand\prj{\operatorname{prj}}

\newcommand\p{\operatorname{\mathbf{p}}\nolimits}

\newcommand\rank{\operatorname{rank}}

\newcommand\Mat{\operatorname{Mat}}

\newcommand\calA{{\mathcal A}}

\newcommand\calC{{\mathcal C}}

\newcommand\calF{{\mathcal F}}

\newcommand\calK{{\mathcal K}}
\newcommand\calL{{\mathcal L}}

\newcommand\bbA{\mathbb{A}}
\newcommand\bbI{\mathbb{I}}

\newcommand\bbZ{\mathbb{Z}}

\newcommand\sfJ{{\mathsf J}}
\newcommand\sfP{{\mathsf P}}
\newcommand\sfQ{{\mathsf Q}}
\newcommand\sfM{{\mathsf M}}

\newcommand{\bfg}{\boldsymbol{g}}
\newcommand{\bsa}{\boldsymbol{a}}

\newcommand\op{^{\mathrm{op}}}

\newcommand\inv{^{-1}}

\renewcommand\implies{\text{$\Rightarrow$}\ }

\newcommand\iso{\cong}
\newcommand\ds{\oplus}

\newcommand\udl{\underline}
\newcommand\ovl{\overline}
\newcommand\Ds{\bigoplus}
\newcommand\DS{\bigoplus\limits}

\def\dsm#1,#2..#3{\bigoplus_{{#1}={#2}}^{#3}}
\def\sm#1,#2..#3{\sum_{{#1}={#2}}^{#3}}

\newcommand\id{1\kern-.25em{\text{{\rm l}}}}

\newcommand\isoto{\ \raise.8ex\hbox{$^{\sim}$}\kern-.7em\hbox{$\to$}\ }

\newcommand\Cdot{\raisebox{1pt}{\scalebox{0.4}{$\bullet$}}}
\newcommand\down{_{\Cdot}}
\renewcommand\up{^{\Cdot}}

\newcommand\ya[1]{\xrightarrow{#1}}
\newcommand\blank{\operatorname{-}}

\def\repr[#1;#2;#3;#4;#5]{
\left(
\begin{matrix}#1\\#2\end{matrix}
#3
\begin{matrix}#4\\#5\end{matrix}
\right)}

\newcommand\mat[1]{\begin{matrix} #1 \end{matrix}}
\newcommand\pmat[1]{\begin{pmatrix} #1 \end{pmatrix}}
\newcommand\bmat[1]{\begin{bmatrix} #1 \end{bmatrix}}
\newcommand\smat[1]{\begin{smallmatrix} #1 \end{smallmatrix}}
\newcommand\sbmat[1]{\left[\begin{smallmatrix} #1 \end{smallmatrix}\right]}

\renewcommand{\theequation}{\thesection.\arabic{equation}}

\newcommand\bfP{\mathbf{P}}

\renewcommand\k{\Bbbk}

\newcommand{\bfa}{\mathbf{a}}
\newcommand{\bfb}{\mathbf{b}}
\newcommand{\bfc}{\mathbf{c}}
\newcommand{\bfd}{\mathbf{d}}
\newcommand{\bfzero}{\mathbf{0}}

\newcommand{\src}{\operatorname{sc}}
\newcommand{\snk}{\operatorname{sk}}

\newcommand{\lex}{\mathrm{lex}}

\newcommand{\plex}{\preceq_{\mathrm{lex}}}

\newcommand{\Tr}{\operatorname{Tr}}

\newcommand{\Ac}{\operatorname{Ac}}
\newcommand{\Pzero}{\bmat{\sfP_{b_1,a_1} & \mathbf{0}\\\mathbf{0}&\mathbf{0}}}
\newcommand{\pr}{\operatorname{pr}}

\renewcommand{\vec}[1]{\smash{\ensurestackMath{\stackengine{1pt}{#1}{\scriptscriptstyle\sim}{U}{c}{F}{F}{S}}}
  \vphantom{#1}
}

\newcommand{\sub}{\mathrm{C}}

\newcommand{\ddset}{\mathord{\Downarrow}}
\newcommand{\uuset}{\mathord{\Uparrow}}
\newcommand{\dset}{\mathord{\downarrow}}
\newcommand{\uset}{\mathord{\uparrow}}

\newcommand{\bbIu}{\bbI_{\mathrm{u}}}
\newcommand{\bbId}{\bbI_{\mathrm{d}}}
\newcommand{\bbIr}{\bbI_{\mathrm{r}}}
\newcommand{\bbIl}{\bbI_{l}}
\newcommand{\crt}{\operatorname{crt}}
\newcommand{\zp}{\mathrm{zp}}

\DeclarePairedDelimiterX\setc[2]{\{}{\}}{\,#1 \;\delimsize\vert\; #2\,}

\newcommand{\thistheoremname}{}
\theoremstyle{plain}
\newtheorem*{genericthm*}{\thistheoremname}
\newenvironment{namedthm*}[1]
  {\renewcommand{\thistheoremname}{#1}
   \begin{genericthm*}}
  {\end{genericthm*}}

\begin{document}

\title{Interval multiplicities of persistence modules}

\author[1,2,3]{\fnm{Hideto} \sur{Asashiba}}\email{asashiba.hideto@shizuoka.ac.jp}

\author*[2]{\fnm{Enhao} \sur{Liu}}\email{liu.enhao.b93@kyoto-u.jp}

\affil[1]{\orgdiv{Department of Mathematics, Faculty of Science}, \orgname{Shizuoka University}, \orgaddress{\street{836 Ohya, Suruga-ku}, \city{Shizuoka}, \postcode{4228529}, \country{Japan}}}

\affil[2]{\orgdiv{Kyoto University Institute for Advanced Study}, \orgname{Kyoto University}, \orgaddress{\street{Yoshida Ushinomiya-cho, Sakyo-ku}, \city{Kyoto}, \postcode{6068501}, \country{Japan}}}

\affil[3]{\orgdiv{Osaka Central Advanced Mathematical Institute}, \orgaddress{\street{3-3-138 Sugimoto, Sumiyoshi-ku}, \city{Osaka}, \postcode{5588585}, \country{Japan}}}

\abstract{
For any persistence module $M$ over a finite poset $\bfP$, and any interval $I$ of $\bfP$, we give a formula for the multiplicity $d_M(V_I)$ of the interval module $V_I$
in the indecomposable decomposition of $M$ in terms of the ranks of matrices consisting of structure linear maps of $M$.
This generalizes the corresponding formula for 1-dimensional persistence modules.
As applications, the formula enables us to compute the maximal interval-decomposable direct summand of $M$, to decide whether $M$ is interval-decomposable, and to detect properties determined by prescribed interval summands without decomposing $M$.
We also give criteria, in terms of top and socle supports along minimal projective resolutions and injective coresolutions of $M$, restricting the intervals that can occur as direct summands of $M$ and thereby reduce the number of intervals to be computed in practice.

Moreover, the formula tells us which morphisms of $\bfP$ are essential to compute $d_M(V_I)$.
This leads to the notion of an order-preserving map $\ze \colon Z \to \bfP$ essentially covering $I$, for which the multiplicity is preserved under the induced restriction functor $R \colon \mod \bfP \to \mod Z$.
When $Z$ is of Dynkin type $\bbA$, also known as a zigzag poset, this allows the multiplicity to be computed more efficiently from the filtration level of topological spaces, without computing all structure linear maps of $M$.

Finally, we give a formula for $d_M(V_I)$ in terms of a projective (or injective) (co)presentation of $M$.
In the 2D-grid case, this is more practical since such resolutions can be computed from the filtration level of topological spaces.
}



\maketitle

\section{Introduction}
Persistent homology analysis has been invented and well-developed in recent decades, regarded as one of the main tools in topological data analysis (TDA for short) \cite{MR1949898}. The standard workflow in the persistent homology analysis is that we first construct a filtration of complexes (or more generally, topological spaces) from the input data, and then record the homological cycles (holes) appearing in the filtration, not only the number but also the time of cycles when they get born and vanish.

In more detail, a \emph{filtration of topological spaces} can be defined as a (covariant) functor $\mathcal{F}\colon \bfP\to \mathrm{Top}$ from a poset $\bfP$ as a category to the category of topological spaces $\mathrm{Top}$ (morphisms are continuous maps). If we let $G$ be an abelian group and $H_{q}(\mbox{-}; G)\colon \mathrm{Top}\to \Ab$ be the $q$-th homology functor from $\mathrm{Top}$ to the category of abelian groups, then $H_{q}(\mbox{-}; G)\circ \mathcal{F}\colon \bfP\to \Ab$ is called the $q$-th \emph{persistent homology}. In most situations, we would let $G$ be a field $\k$ and require the tameness on the filtration, meaning that the poset $\bfP$ is finite and each topological space $\mathcal{F}(x)$ has finite-dimensional $q$-th homology (see~\cite{cohen-steinerStabilityPersistenceDiagrams2007}). Hence the $q$-th persistent homology $H_{q}(\mbox{-}; \k)\circ \mathcal{F}\colon \bfP\to \mod \k$ becomes a functor from the poset as a category to the category $\mod \k$ of finite-dimensional $\k$-vector spaces.

As a mathematical generalization, it is natural to forget about the process of taking homology and directly consider the (covariant) functor $M\colon \bfP\to \mod \k$, known as the definition of persistence module (over $\bfP$) in TDA community nowadays. This point of view makes it possible to study persistence modules in the context of representation theory. By convention, we call $H_{\star}(\mbox{-}; \k)\circ~\mathcal{F}$ (resp. $M$) the \emph{one-parameter persistent homology} (resp. \emph{1-dimensional persistence module}) if $\bfP$ is a totally ordered set. In this case, the construction of the filtration depends only on a single parameter. With the study of {\em zigzag persistence}, representations of a zigzag partially ordered set (\emph{zigzag poset} for short, including totally ordered set as a special case)
are proposed and called {\em zigzag modules}~\cite{carlssonZigzagPersistence2010,botnanAlgebraicStabilityZigzag2018}.
This type of persistence modules can be also regarded as representations of a Dynkin quiver of type $\bbA$, and they are uniquely determined by all intervals of underlying ordered sets~\cite{crawley-boeveyDecompositionPointwiseFinitedimensional2015,carlssonZigzagPersistence2010,botnanIntervalDecompositionInfinite2017}, hence we can visualize the persistent homology by drawing the persistence diagram or persistence barcodes, presenting the multiset of intervals~\cite{MR2121296}.

Nevertheless, in some applications, the construction of a filtration depends on multiple parameters, causing the underlying poset $\bfP$ to no longer be of Dynkin type $\bbA$. Hence, it is also necessary to study the persistence module for different types of poset.
In the context of multi-parameter persistent homology, for example,
the product poset of $d$ posets of Dynkin type $\bbA$ for some $d > 1$ is commonly considered and called the $d$D-grid (see Definition~\ref{dfn:2D-grid} for the 2D-grid as an example) \cite{deyComputingBottleneckDistance2018}. Except for only a few cases,
the category of $d$-dimensional persistence modules has infinitely many indecomposables up to isomorphism if $d > 1$ \cite{leszczynskiRepresentationTypeTensor1994,bauerCotorsionTorsionTriples2020}. In these cases, dealing with all indecomposable persistence modules is very difficult and is usually inefficient.

For the practical analysis, analog to the one-parameter persistence case, one can also restrict to the well-defined interval modules in the general poset setting as intervals encode lifetimes of topological features emerging from data and admit simple characterizations. On the other hand, the multiplicity of interval modules plays a key role in relating other invariants. For example, the interval rank invariants defined in~\cite{asashiba2024interval} can be interpreted as the multiplicity of some interval module after the restriction.
Therefore, computing the multiplicity of each interval summand\footnote{A direct summand is sometimes called just a summand for short in this paper.} of a given persistence module over $\bfP$ becomes a central task.

\subsection{Notation conventions}

Throughout this paper, we fix a field $\k$, and
all vector spaces are assumed to be over $\k$,
and the word ``linear'' always means ``$\k$-linear''.
The category of finite-dimensional vector spaces is denoted by $\mod \k$. We always assume the tameness on the filtration. For each positive integer $n$, we set $[n]\coloneqq \{1, 2, \dots, n\}$.

We also fix a finite poset $\bfP$ and regard it as a category in an obvious way, and for any $x, \, y \in \bfP$ with $x \le y$, a unique morphism $x \to y$ is denoted by $p_{y,x}$.
Then the incidence category $\k[\bfP]$ of $\bfP$ is defined
as a linearization of the category $\bfP$ (see Definition \ref{dfn:inc-cat}).
Each functor $F \colon \bfP \to \mod \k$ is uniquely extended to a linear functor
$\bar{F}\colon \k[\bfP] \to \mod \k$.
Therefore, we identify $F$ with $\bar{F}$, and denote it simply by $F$.

Let $\calC$ be a linear category with only a finite number of objects.
Then covariant (resp.\ contravariant) functors $\calC \to \mod \k$
are called finite-dimensional left (resp.\ right) \emph{modules} over $\calC$
or shortly left (resp.\ right) $\calC$-modules, the category of which
is denoted by $\mod \calC$ (resp.\ $\mod \calC\op$).
We usually consider finite-dimensional left modules and call them simply modules unless otherwise stated. In this paper, modules over the incidence category of a finite poset will be called persistence modules. By the Krull-Schmidt Theorem, every persistence module $M$ is uniquely decomposed into indecomposables up to isomorphism, which gives the multiplicity of each indecomposable $L$, denoted by $d_M(L)$, in the decomposition of $M$ (see~\cref{thm:KS}).

A full subposet $I$ of $\bfP$ is called an \emph{interval} if
it is convex in $\bfP$ and connected (see Definition \ref{dfn:ss-intervals}).
The set of all intervals of $\bfP$ is denoted by $\bbI$.
Each $I \in \bbI$ defines an indecomposable $\k[\bfP]$-module $V_I$ with support $I$,
which is called an interval module (see Definition \ref{dfn:intv}).
A persistence module is said to be \emph{interval-decomposable}
if it is isomorphic to the direct sum of a finite number of interval modules. In what follows, we call $d_{M}(L)$ the interval multiplicity of $L$
in $M$ if $L$ is an interval module.

The following is necessary to state our main results.

\begin{ntn}
\label{ntn: notations in Intro}
\begin{enumerate}
\item Let $x\in \bfP$, and $I$ an interval of $\bfP$.
We set $\uset x\coloneqq \{y \in \bfP \mid x \le y\}$ (resp.\ $\dset x\coloneqq \{y \in \bfP \mid y \le x\}$), and call it the \emph{up-set} (resp.\ \emph{down-set}) of $x$.
In turn, we set
\[
\begin{aligned}
\uset I\coloneqq \bigcup_{x \in I} \uset x\ \ (\text{resp.}\
\dset I\coloneqq \bigcup_{x \in I} \dset x),\ \text{and}\ \uuset  I\coloneqq \uset I\setminus I\ \ (\text{resp.}\ \ddset  I\coloneqq \dset I\setminus I),
\end{aligned}
\]
and call them the \emph{up-set} (resp.\ \emph{down-set}) of $I$, and
the \emph{proper up-set} (resp.\ \emph{proper down-set}) of $I$, respectively.

\item For any totally ordered set $T = (T, \preceq)$,
we denote by $\sub_2 T$ the set of all subsets of $T$ consisting of exactly two elements. For any $\bfa = \{i, j\} \in \sub_2 T$ with $i \prec j$ in $T$,
we set $\udl{\bfa}\coloneqq i$ (resp. $\ovl{\bfa}\coloneqq j$).
Then we can write $\bfa = \{\udl{\bfa},\, \ovl{\bfa}\}$.

Now, after giving total orders on $\src(I)$ and $\snk(I)$,
for any $\bfa \in \sub_2\src(I)$ (resp.\ $\bfb \in \sub_2\snk(I)$),
we set $\vee'\bfa \coloneqq \src (\uset \udl{\bfa} \cap \uset \ovl{\bfa})$ (\text{resp.\ } $\wedge'\, \bfb \coloneqq \snk (\dset \udl{\bfb} \cap \dset \ovl{\bfb}$),
and call it the \emph{pre-join} (resp.\ \emph{pre-meet}) of $\bfa$ (resp.\ $\bfb$).
We then set $\src_{1}(I)$ to be the disjoint union of all pre-joins of the two-element subsets of $\src(I)$. Namely,
\[
\begin{aligned}
\src_{1}(I)&\coloneqq\bigsqcup\limits_{\bfa \in \sub_2\src(I)}\kern -1em \vee'\bfa
= \setc*{\bfa_c\coloneqq(\bfa, c)}{\bfa \in \sub_2\src(I),\, c \in \vee'\bfa},\\
\text{and similarly, }\ \snk_{1}(I)&\coloneqq\bigsqcup\limits_{\bfb \in \sub_2\snk(I)}\kern -1.2em \wedge'\bfb = \setc*{\bfb_d\coloneqq (\bfb, d)}{\bfb \in \sub_2\snk(I),\, d \in \wedge'\bfb}.
\end{aligned}
\]
For example in \cref{fig:explain-of-labels}, for $\bfa \coloneqq \{a_2, a_3\}\in \sub_2\src(I)$ (with the additional total order $\preceq$) we have $\udl{\bfa} = a_2$ and $\ovl{\bfa} = a_3$. Moreover, the element $\{a_2, a_3\}_{x}$ is minimal in $\uset \udl{\bfa} \cap \uset \ovl{\bfa}$.

\item If $\src(\uuset I)\neq \varnothing$, then for each $a'\in \src(\uuset I)$, we have $\src(I)\cap \dset a'\neq \varnothing$ because $a' \in \uset I$. Fixing one element in $\src(I)\cap \dset a'$ for each $a'\in \src(\uuset I)$ yields a map $\bfc\colon \src(\uuset I)\to \src(I)$. We call such $\bfc$ a \emph{choice map}. Dually, if $\snk(\ddset I)\neq \varnothing$, then for each $b'\in \snk(\ddset I)$, we have $\snk(I)\cap \uset b'\neq \varnothing$ because $b' \in \dset I$. Fixing one element $b\in \snk(I)\cap \uset b'$ for each $b'\in \snk(\ddset I)$ yields another choice map $\bfd\colon \snk(\ddset I)\to \snk(I)$ that sends $b'$ to $b$. See \cref{fig:explain-of-labels} for an illustration of such an $a'\in \src(\uuset I)$ and a choice of $a_1\in \src(I)$ such that $a_1\leq a'$.
\begin{figure}[ht]
  \centering
  \includegraphics[width=0.45\textwidth]{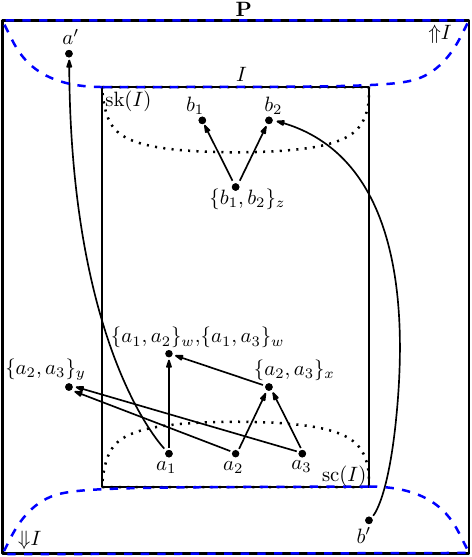}
  \caption{Illustration of notations}
  \label{fig:explain-of-labels}
\end{figure}

\end{enumerate}
\end{ntn}

As a remark in~\cref{fig:explain-of-labels}, there are two elements $x$ and $y$ in $\vee' \{a_2, a_3\}$, labeled as $\{a_2, a_3\}_{x}$ and $\{a_2, a_3\}_{y}$ in $\src_1(I)$, respectively. We address that $\src_1(I)$ may contain some elements which are not in $I$. For example, $y$ is not an element of $I$ in this illustration. We also note that the element $w\notin \vee' \{a_2, a_3\}$ since $w$ is not minimal in $\uset\, a_2 \cap \uset\, a_3$, but $w$ is in both $\vee' \{a_1, a_2\}$ and $\vee' \{a_1, a_3\}$. By this we write $\{a_1, a_2\}_w$ and $\{a_1, a_3\}_w$ in $\src_1(I)$, standing for $w\in I$.

\subsection{Purposes}
In the standard one-parameter persistent homology, the multiplicity of an interval can be computed by taking ranks (a.k.a., persistent Betti numbers) along some larger intervals and then operating an alternating sum of the ranks by the inclusion-exclusion principle. This computation is well-known as the formula of the persistent Betti numbers and the multiplicity in one-parameter persistent homology (see \cite[Chapter VII]{edelsbrunner2010computational}).

More precisely, if we let the poset $\bfP$ be the set $[n]$ together with the natural number ordering, $\calF$ a $\bfP$-filtration, and $M \coloneqq H_{q}(\mbox{-}; \k)\circ~\mathcal{F}$ the $q$-th persistent homology, then the multiplicity of each interval $I \coloneqq \{x\in \bfP \mid s\leq x\leq t\}\subseteq \bfP$ ($[s, t]$ for short) appearing in the $q$-th persistence barcodes of $M$ is given by
\begin{equation}
\label{eq:1-para-filt}
\mu_M(I) = \rank M(p_{t,s}) - \rank M(p_{t,s-1}) - \rank M(p_{t+1,s}) + \rank M(p_{t+1,s-1})
\end{equation}
where $\mu_M(I) \coloneqq d_M(V_I)$ denotes the multiplicity of $I$, and $\rank M(p_{t,s})$ denotes the rank of the linear map $M(p_{t,s}) \colon M(s) \to M(t)$. As a demonstration, in \cref{fig:intro} we consider the multiplicity of interval $[3, 4]$. The intervals that appear in the right-hand side of \eqref{eq:1-para-filt} are illustrated in the violet color. It is straightforward to see that $\rank M(p_{4,3}) = 2$, $\rank M(p_{4,2}) = \rank M(p_{5,3}) = 1$, and $\rank M(p_{5,2}) = 0$, thus $\mu_M(I) = 0$ follows by \eqref{eq:1-para-filt}. The reader can similarly check by \eqref{eq:1-para-filt} that multiplicities of intervals $[2,4]$ and $[3,5]$ are both $1$, and other intervals have zero multiplicities.

\begin{figure}[ht]
  \centering
  \includegraphics[width=\textwidth]{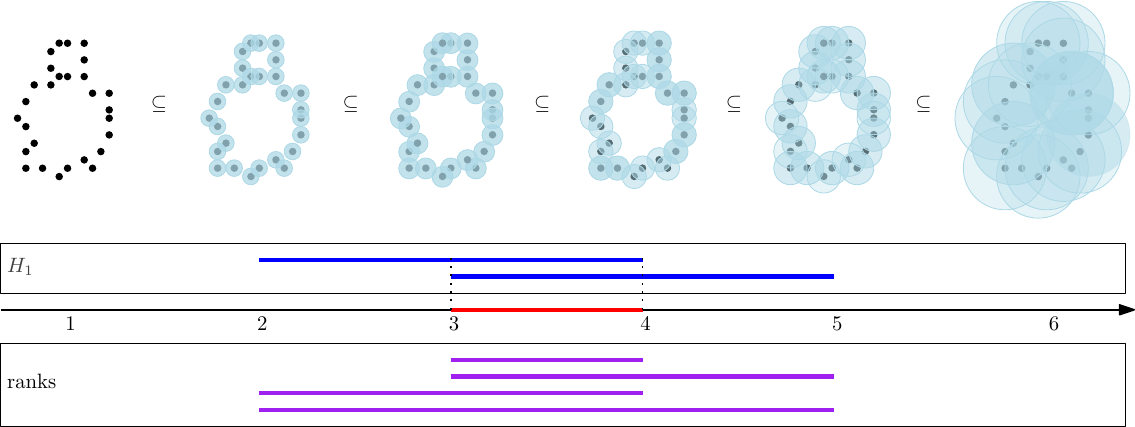}
  \caption{An illustration of the formula of persistent Betti numbers and multiplicities in one-parameter persistent homology. The barcodes of the $1$-st persistent homology are shown in blue. The red bar $[3,4]$ is the interval that we would like to demonstrate the computation of its multiplicity. The violet bars are intervals along which we take the ranks to recover the multiplicity.}
  \label{fig:intro}
\end{figure}

We recall that meanings of the multiplicity and  the rank are different. The multiplicity of an interval $[s, t]$ in persistence barcodes indicates the number of generators of homology that are newly born at $s$ (do not exist before $s$) and die at $t$ (do not exist after $t$), persisting from $s$ to $t$, while the rank along $[s, t]$ only counts the number of generators of homology that persist from $s$ to $t$ without caring those generators whether are newly born or eventually die at endpoints. In other words, the rank (or persistent Betti number) along an interval $[s, t]$ only needs the information of $[s,t]$, but the multiplicity needs extra information of $[s,t]$ that is recorded in a larger interval $[s-1, t+1]$ which contains $[s,t]$.

In the multi-parameter persistence or the persistence over general poset context, the persistent homology obtained is generally not interval-decomposable.
Therefore, we need a formula for the multiplicity $d_M(L)$ of $L$ in $M$ for any indecomposable direct summand $L$ of a given persistence module $M$ in general cases. This formula is given in \cite[Theorem 3]{Asashiba2017} (see Theorem \ref{thm:dMV} for detail)
in terms of the dimensions of Hom-spaces $\Hom_{\k[\bfP]}(X, M)$ with
$X \in \{L, E_L, \ta\inv L\}$, where there exists a minimal left almost split morphism from $L$ to $E_L$ (see also \cite[Corollary 2.3]{dowborMultiplicityProblemIndecomposable2007}).
Thus, to apply this formula, we have to know a minimal left almost split morphism $f \colon L \to E_L$,
$\ta\inv L = \Cok f$ and the dimensions of Hom-spaces above, which are in general hard to do if
the computation of the Auslander-Reiten quiver (AR-quiver for short) is not easy.
For example, in the finite 2D-grid case, namely when $\bfP$ is the product poset $G_{m,n}\coloneqq [m] \times [n]$ for some $m,n \ge 2$, let $L$ be an interval module and $M$ a $\k[\bfP]$-module.
Then by \cite[Propositions 42 and 43]{Asashiba2017},
the time complexity of the computation of $\{f, E_L, \ta\inv L\}$
and that of multiplicity $d_{L}(M)$ was given as
$O(mnz^\om)$ and $O(((\dim M)^\om + mn)z^\om)$, respectively,
where $z\coloneqq \min\{m,n\}$ and $\om < 2.373$ is the matrix multiplication exponent.

However, when the AR-quiver is known, the general formula can
easily be applied.  For example, the AR-quiver of $\k[\bfP]$ for $\bfP = [n]$ with $n \ge 1$
is well-known, and the formula \eqref{eq:1-para-filt} follows from Theorem \ref{thm:dMV}
by computing the dimensions $\dim\Hom_{\k[\bfP]}(X, M)$ in terms of the structure linear maps of $M$
as was shown in \cite[Formula (9) in Example 3]{Asashiba2017}.
Hence Theorem \ref{thm:dMV} can be seen as a generalization of \eqref{eq:1-para-filt}.

Concerning the formula \eqref{eq:1-para-filt},
the following general question naturally arises:
For a given interval $I$,
which structure of $M$ determines the multiplicity $d_M(V_I)$?
If we find a formula of multiplicities of intervals in terms of the structure linear maps of $M$ analogous to \eqref{eq:1-para-filt}, then it will give an answer to the question.
However, this kind of formula is not yet given explicitly in the literature.
The main purpose of this paper is to give such an explicit formula for the interval multiplicity in the finite poset case.
This makes it clear which structure of $M$ is essential
to determine the multiplicity $d_M(V_I)$, and leads to an idea of essential cover explained in  the next subsection \ref{ssec:contrib} (3).
On the other hand, in \cite[Lemma~4.8]{asashiba2024interval}, we developed a way (Lemma \ref{lem:dim-Hom-coker} in the present paper) to compute $\dim\Hom_{\k[\bfP]}(X, M)$ in terms of the structure linear maps of $M$ by using a projective presentation of $X$.
Thus our purpose can be achieved by computing $E_{V_I}$ and $\ta\inv V_I$ as mentioned above,
and to compute projective presentations of $V_I, E_{V_I}$ and $\ta\inv V_I$
for any interval $I$ in general.
A priori, our proposed formula is different from the M{\"o}bius inversion formula of signed interval multiplicities (resp. generalized persistence diagrams) and interval rank invariants given in~\cite{asashiba2024interval} (resp. generalized rank invariants given in~\cite{kim2021generalized}) because the interval multiplicities and the signed interval multiplicities do not coincide in general when the persistence module $M$ is not interval-decomposable.

\subsection{Our contributions}
\label{ssec:contrib}

(1) We provide an explicit formula for computing the multiplicities of interval summands of a given persistence module over arbitrary finite poset, at the algebraic level
(Theorem \ref{thm:gen-formula-unified}). The formula only depends on some of the structure linear maps inside the persistence module. It turns out that the task of computing the interval multiplicity converts to the task of computing the rank of some matrices.

\begin{namedthm*}{Main result A (\cref{thm:gen-formula-unified})}
\label{thm:main-result-formula-int-multiplicity}
Let $M \in \mod \k[\bfP]$, and $I$ an interval of $\bfP$. Then
\begin{align}
\label{eq:int-multiplicity}
\scalebox{1}{$
d_M(V_I) = \rank
\begin{blockarray}{cccc}
& \scalebox{0.8}{$\src(I)$} & \scalebox{0.8}{$\snk(\ddset I)$} & \scalebox{0.8}{$\snk_1(I)$} \\
\begin{block}{c[ccc]}
  \scalebox{0.8}{$\src_1(I)$} & \sfM_1 & \mathbf{0} & \mathbf{0}  \\
  \scalebox{0.8}{$\src(\uuset I)$} & \sfM_2 & \mathbf{0} & \mathbf{0}  \\
  \scalebox{0.8}{$\snk(I)$} & \sfM_3 & \sfM_4 & \sfM_5 \\
\end{block}
\end{blockarray}
- \rank
\bmat{
  \sfM_1 & \mathbf{0} & \mathbf{0}  \\
  \sfM_2 & \mathbf{0} & \mathbf{0}  \\
  \mathbf{0} & \sfM_4 & \sfM_5 \\
}
$}
\end{align}
holds. Here the $(\bfa_c,a)$-entry of $\sfM_1$ (\textnormal{resp.} the $(b, \bfb_d)$-entry of $\sfM_5$) is given by
\begin{align*}
\scalebox{1}{$
(\sfM_1)_{\bfa_c,\, a}\coloneqq
\begin{cases}
M_{c,a} & (a = \udl{\bfa}),\\
-M_{c,a} & (a = \ovl{\bfa}),\\
\mathbf{0} & (a \not\in \bfa),
\end{cases}
$}
\ \left(\textnormal{resp. }
\scalebox{1}{$
(\sfM_5)_{b,\, \bfb_d}\coloneqq
    \begin{cases}
    M_{b,d} & (b=\udl{\bfb}),\\
    -M_{b,d} & (b=\ovl{\bfb}),\\
    \mathbf{0},\,& (b \not\in \bfb),
    \end{cases}
    $}
    \right)
\end{align*}
and
\begin{align*}
\sfM_2\coloneqq \bmat{\de_{a,\, \bfc(a')}M_{a',\, \bfc(a')}}_{(a',a) \in \src(\uuset I) \times \src(I)}\\
\bigg(\textnormal{resp. } \sfM_4\coloneqq \bmat{\de_{b,\, \bfd(b')}M_{\bfd(b'),\, b'}}_{(b,b') \in \snk(I) \times \snk(\ddset I)}\bigg).
\end{align*}
$\sfM_3$ is given by a choice of pair $(b_j, a_i)\in \snk(I)\times \src(I)$ with $b_j\ge a_i$. Namely,
the $(b_j, a_i)$-entry is the only non-zero entry of $\sfM_3$ and it equals to $M_{b_j, a_i}$. Index sets of matrices in formula~\eqref{eq:int-multiplicity} are allowed to be empty. In this case, we remove rows (columns) of matrices corresponding to empty index sets.
\end{namedthm*}

We remark that the proposed formula~\eqref{eq:int-multiplicity} is a generalization of the formula~\eqref{eq:1-para-filt} in the one-parameter persistence case
(see Remark \ref{rmk:1par-multpar}).

(2) Because the 2D-grid poset setting is well motivated
and has more interest for TDA researchers,
especially in the multi-parameter persistent homology, we also give the explicit formula in this poset setting (Theorem \ref{thm:gen-formula}). In this case, we set
$$
\begin{aligned}
\src^\circ_1(I)= \{a_{i,i+1}\coloneqq a_i \vee a_{i+1} \mid i \in [k-1]\},\ \snk^\circ_1(I)= \{b_{i,i+1}\coloneqq b_i \wedge b_{i+1} \mid i \in [l-1]\}.
\end{aligned}
$$
We note that $\src^\circ_1(I)\subseteq \src_1(I)$ and $\snk^\circ_1(I)\subseteq \snk_1(I)$. Then we have the following.

\begin{namedthm*}{Main result B (\cref{thm:gen-formula})}
\label{thm:formula-int-multiplicity}
Let $M \in \mod \k[\bfP]$, and $I$ an interval of $\bfP$. Then
\begin{align}
\label{eq:main-result-int-multiplicity-2D}
\scalebox{0.9}{$
d_M(V_I) = \rank
\begin{blockarray}{cccc}
& \scalebox{0.8}{$\src(I)$} & \scalebox{0.8}{$\snk(\ddset I)$} & \scalebox{0.8}{$\snk^\circ_1(I)$} \\
\begin{block}{c[ccc]}
  \scalebox{0.8}{$\src^\circ_1(I)$} & \sfM_1 & \mathbf{0} & \mathbf{0}  \\
  \scalebox{0.8}{$\src(\uuset I)$} & \sfM_2 & \mathbf{0} & \mathbf{0}  \\
  \scalebox{0.8}{$\snk(I)$} & \sfM_3 & \sfM_4 & \sfM_5 \\
\end{block}
\end{blockarray}
- \rank
\bmat{
  \sfM_1 & \mathbf{0} & \mathbf{0}  \\
  \sfM_2 & \mathbf{0} & \mathbf{0}  \\
  \mathbf{0} & \sfM_4 & \sfM_5 \\
}
$}
\end{align}
holds. Here $\sfM_1$ and $\sfM_5$ are given by
\[
\sfM_1\coloneqq
\bmat{
M_{a_{12},a_1} & -M_{a_{12},a_{2}}\\
&M_{a_{23},a_{2}} & -M_{a_{2,3},a_{3}}\\
&& \ddots & \ddots\\
&&& M_{a_{k-1,k},a_{k-1}} & -M_{a_{k-1,k},a_{k}}
}
\]
and
\[
\sfM_5\coloneqq \bmat{
M_{b_1,b_{12}} \\
-M_{b_2,b_{12}}&M_{b_2,b_{23}} \\
& -M_{b_3,b_{23}}& \ddots \\
&&\ddots& M_{b_{l-1},b_{l-1,l}}\\
&&&-M_{b_{l},b_{l-1,l}}
}.
\]
$\sfM_2$, $\sfM_4$ are given by
\begin{align*}
\sfM_2\coloneqq \bmat{\de_{a,\, \bfc(a')}M_{a',\, \bfc(a')}}_{(a',a) \in \src(\uuset I) \times \src(I)},\ \sfM_4\coloneqq \bmat{\de_{b,\, \bfd(b')}M_{\bfd(b'),\, b'}}_{(b,b') \in \snk(I) \times \snk(\ddset I)},
\end{align*}
respectively. $\sfM_3$ is given by a choice of pair $(b_j, a_i)\in \snk(I)\times \src(I)$ with $b_j\ge a_i$. Namely,
the $(b_j, a_i)$-entry is the only non-zero entry of $\sfM_3$ and it equals to $M_{b_j, a_i}$. Index sets of matrices in formula~\eqref{eq:main-result-int-multiplicity-2D} are allowed to be empty. In this case, we remove rows (columns) of matrices corresponding to empty index sets.
\end{namedthm*}

We provide~\cref{exm:2Dgrid} of computing the interval multiplicity in the 2D-grid case by utilizing the formula~\eqref{eq:main-result-int-multiplicity-2D}.

(3) As we saw above, one can compute interval multiplicities from the persistent homology (persistence module) thus obtained. However, computing persistent homology directly from an arbitrary filtration of topological spaces is generally inefficient. To address this, we introduce the \emph{essential-cover} technique, which computes the invariants by focusing on the essential structure of poset $\bfP$ and building a new filtration indexed by another poset. This approach subsequently enables the use of efficient computational algorithms of standard one-parameter or zigzag persistence if the newly constructed filtration is indexed by the poset of Dynkin type $\bbA$, hence making it possible to compute the interval multiplicity starting from the filtration level.

Roughly speaking, the essential cover $\ze\colon Z\to \bfP$ is an order-preserving map, and we say that $\ze$ \emph{essentially covers} an interval $I$ (\textnormal{resp.} \emph{relative to compression systems}) if morphisms in $\bfP$ that appear in formula~\eqref{eq:int-multiplicity} have preimages in $Z$, subject to some mild technical conditions. We refer the reader to~\cref{dfn:ess-cov2} for the detailed definition. Then we have the next main result.

\begin{namedthm*}{Main result C (\cref{thm:ess-cover-equality})}
\label{thm:intro3}
Let $M \in \mod \k[\bfP]$, and $I$ an interval of $\bfP$. If $\ze\colon Z\to \bfP$ essentially covers $I$, then we have
\begin{align}
d_M(V_{I}) = \bar{d}_{R(M)}(R(V_{I})),\label{eq:ess-cov-a}
\end{align}
where $R$ denotes the restriction functor induced by $\ze$, and $\bar{d}_N(L)$ denotes the maximal number of copies of $L$ that can be taken as a direct summand of $N$ such that no further copies of $L$ remain in the complement as direct summands.
\end{namedthm*}

We remark that in \eqref{eq:ess-cov-a}, If $L$ is indecomposable, then $\bar{d}_N(L)$ is just the usual multiplicity of $L$ in $N$. Note that it may happen that $R(V_I)$ is decomposable.
This is why $\bar{d}$ is used here instead of $d$ (see~\cref{rmk:snake}).

For every interval of $\bfP$, \cref{eq:ess-cov-a} provides us with a method to transfer the computation of the interval multiplicity to the computation of corresponding multiplicity over another ``essential'' poset $Z$. When $\bfP$ is a special type of poset, the ``essential'' poset $Z$ can be taken as either a single zigzag poset or a directed tree formed by connecting several zigzag posets. This makes it possible to utilize algorithms designed for computing zigzag persistence (for example, \cite{deyFastComputationZigzag2022,milosavljevicZigzagPersistentHomology2011,carlssonZigzagPersistentHomology2009}) to compute interval multiplicities. The idea of utilizing the computation of zigzag persistence for computing invariants has also been considered in the literature. For example, \citet{deyComputingGeneralizedRanks2024} recently achieve computing the generalized rank invariant by the unfolding process.

Our last main result gives a formula for computing interval multiplicities of a persistence module by utilizing its (minimal) projective presentation or injective copresentation, without knowing the structure (linear maps) of persistence module over arbitrary finite posets. Notably, the computation of the minimal (co)presentations of 2-parameter persistent homology has been actively studied in the literature, and many fast algorithms are currently available for this purpose. \cite{lesnickComputingMinimalPresentations2022} first develop a way of computing the minimal projective presentation of 2-parameter persistent homology. Later, \cite{Kerber-Rolle2021,fugacciCompression2parameterPersistent2023} improve the pioneering work of Lesnick and Wright by some techniques such as multi-chunk~\cite{fugacciChunkReductionMultiparameter2019}. Regarding the minimal injective copresentation, \cite{bauerEfficientTwoparameterPersistence2023} propose a cohomological algorithm for computing minimal free resolutions of 2-parameter persistent cohomology.

Let us denote by $\sfP$ the extension of the \emph{Yoneda embedding} $Y\up \colon \k[\bfP\op] \to \prj \k[\bfP]$, $x \mapsto P_x\coloneqq \k[\bfP](x, \blank)$, where $P_x$ denotes the projective indecomposable $\k[\bfP]$-module at $x$. Similarly, by $\sfP'$ we denote the extension of the Yoneda embedding $Y\down \colon \k[\bfP] \to \prj (\k[\bfP\op])$,
$x \mapsto P'_x\coloneqq \k[\bfP](\blank, x)$. See~\cref{cor:Yoneda} for details.

\begin{namedthm*}{Main result D (\cref{thm:formula-pp(M)-icp(ass),thm:formula-icp(M)-pp(ass)})}
\label{thm:intro4}
Let $I$ be an interval of $\bfP$. If $M \in \mod \k[\bfP]$ has
a projective presentation
$$
\Ds_{j\in [n]} P_{y_j} \ya{\sfP(\al)} \Ds_{i\in [m]} P_{x_i} \to M \to 0,
$$
where $\sfP(\al)$ is called the {\it presentation matrix} of $M$, and $\al$ denotes a matrix whose entries are morphisms in $\k[\bfP]$ and column indices are $(x_1,\ldots,x_m)\eqqcolon x$. Then we have the following formula for $d_M(V_I)$:
\begin{align}
d_M(V_I)
&= \rank \left[\begin{array}{c|c|cc}
 \sfP'(\bfg_1)(x) & \bfzero & \sfP'(\src_{1}(I) \ds \src(\uuset I))(\al) & \bfzero\\
 \hline
\sfP'(\bfg_3)(x)
& \sfP'(\bfg_2)(x) &
\bfzero & \sfP'(\snk(I))(\al)
 \end{array}\right] \nonumber\\
 &\quad - \rank \left[\begin{array}{c|c|cc}
 \sfP'(\bfg_1)(x) & \bfzero & \sfP'(\src_{1}(I) \ds \src(\uuset I))(\al) & \bfzero\\
 \hline
\bfzero
& \sfP'(\bfg_2)(x) &
\bfzero & \sfP'(\snk(I))(\al)
 \end{array}\right]. \label{eq:intro-thm4}
 \end{align}
Here, all block matrices in~\eqref{eq:intro-thm4} are defined in~{\rm \cref{rmk:order-smnds}}, and can be fast written down by~{\rm \cref{prp:fast-way-block-matrices}}. Note that for the 2D-grid case, we can replace $\src_1(I)$ with $\src_1^\circ(I)$ in~\eqref{eq:intro-thm4}.

Dually, if $M$ has an injective copresentation
$$
0 \to M \to \Ds_{i\in [m']} Q_{x'_i}\ya{\sfQ(\al')} \Ds_{j\in [n']} Q_{y'_j},
$$
where $\sfQ(\al')$ is called the copresentation matrix of $M$, and $\al'$ denotes a matrix whose entries are morphisms in $\k[\bfP]$ and row indices are $(x'_1,\ldots,x'_{m'})\eqqcolon x'$. Then we have the following formula for $d_M(V_I)$:
$$
\begin{aligned}
d_M(V_I)
&= \rank \left[\begin{array}{c|c|cc}
 \sfP(\bfg_1)(x') & \bfzero & \sfP(\src(I))(\al') & \bfzero\\
 \hline
\sfP(\bfg_3)(x')
& \sfP(\bfg_2)(x') &
\bfzero & \sfP(\snk(\ddset I) \ds \snk_1(I))(\al')
 \end{array}\right]\\
 &\quad - \rank \left[\begin{array}{c|c|cc}
 \sfP(\bfg_1)(x') & \bfzero & \sfP(\src(I))(\al') & \bfzero\\
 \hline
\bfzero
& \sfP(\bfg_2)(x') &
\bfzero & \sfP(\snk(\ddset I) \ds \snk_1(I))(\al')
 \end{array}\right].
 \end{aligned}
$$
Note that
for the 2D-grid case, we can replace $\snk_1(I)$ with $\snk_1^\circ(I)$.
\end{namedthm*}

The formulas we present are general and apply to all intervals of a poset. However, as the size of poset increases, the number of its intervals also grows. In practical computations, it is therefore inefficient to iterate over every interval. For this reason, in~\cref{sect:Reducing candidates of the interval direct summands} we provide a criterion in~\cref{lem:critical-i}: given a persistence module $M$, all of its interval summands are constrained by the tops of the projective modules appearing in a minimal projective resolution of $M$ and the socles of the injective modules appearing in a minimal injective coresolution of $M$. This means that it suffices to consider what we will call a critical set of intervals; every interval with nonzero multiplicity can occur only within such a set. In practical computations, this allows us to rule out many intervals before applying the formulas (see also~\cref{prp:criterion-int-sum}).

Finally, we will provide several examples to demonstrate the use of essential-cover technique. In~\cref{exm:ess-cover-filt} we show how to compute the interval multiplicity in 2D-grid case from the level of filtration. \cref{exm:ess-cover-filt-2} shows that in some cases, the essential cover of intervals starts from a directed tree formed by connecting several zigzag posets, not a single zigzag poset. In Examples~\ref{exm:int-rk_D_4_cases} and \ref{exm:int-rk_D_4_cases_2nd} we compute the interval multiplicity in another type of posets, namely the posets of Dynkin type $\mathbb{D}$. Furthermore, we investigate the computation in \emph{bipath posets}, the posets that always possess the interval decomposability studied initially by~\citet{aoki2023summand}. We propose an alternative way of computing the bipath persistence diagram from a given bipath filtration. Compared with the original algorithm provided in \cite{aokiBipathPersistence2024}, an advantage of our approach is that we do not need to do the basis changes at the global minimum and maximum of bipath poset.

\subsection{Organization}
We outline the paper as follows. In Section~\ref{section-2}, we will introduce some preliminaries. In Section~\ref{section-3}, we will give the formula for computing the interval multiplicities in the general finite poset setting (\cref{ssec:formula-gen-case}),
and particularly in the 2D-grid setting (\cref{ssec:2Dgrid-case}). \ref{ssec:2Dgrid-case} has a simpler formula and easier to grasp than \ref{ssec:formula-gen-case}.
The reader may read \ref{ssec:2Dgrid-case} first by looking at
Example \ref{exm:2Dgrid} to have a rough outline.
It contains enough information to apply the formula for 2D-grids.
The details of proofs written in \ref{ssec:formula-gen-case} can be read afterward. In Section~\ref{section-4}, we develop the essential-cover technique for the sake of practical data analysis. In~\cref{sec:multi-by-pp}, we give formulas for computing interval multiplicities by using (co)presentations. In~\cref{section-5}, we show some examples of the use of essential-cover technique in different types of underlying posets.

\section{Preliminaries}
\label{section-2}
Throughout this paper, $\k$ is a field, $\bfP = (\bfP, \le)$ is a finite poset.
The category of finite-dimensional $\k$-vector spaces is denoted by $\mod \k$.

\begin{dfn}
\label{dfn:left-right-mod}
A $\k$-linear category $\calC$ is said to be {\em finite}
if it has only finitely many
objects and for each pair $(x,y)$ of objects,
the Hom-space $\calC(x,y)$ is finite-dimensional.

Covariant functors $\calC \to \mod \k$
are called {\em left $\calC$-modules}.
They together with natural transformations between them as morphisms form a $\k$-linear category, which is denoted by $\mod \calC$.

Similarly, contravariant functors $\calC \to \mod \k$ are called
\emph{right $\calC$-modules},
which are usually identified with covariant functors
$\calC\op \to \mod \k$.
The category of right $\calC$-modules is denoted by $\mod \calC\op$.

We denote by $D$ the usual $\k$-duality $\Hom_\k(\blank, \k)$,
which induces the duality functors
$\mod \calC \to \mod \calC\op$ and $\mod \calC\op \to \mod \calC$.
\end{dfn}

\begin{dfn}
\label{dfn:inc-cat}
The poset $\bfP$ is regarded as a category as follows.
The set $\bfP_0$ of objects is defined by $\bfP_0\coloneqq \bfP$.
For each pair $(x,y) \in \bfP \times \bfP$,
the set $\bfP(x,y)$ of morphisms from $x$ to $y$
is defined by $\bfP(x,y)\coloneqq \{p_{y,x}\}$ if $x \le y$,
and $\bfP(x,y)\coloneqq \emptyset$ otherwise, where
we set $p_{y,x}\coloneqq (y,x)$.
The composition is defined by $p_{z,y}p_{y,x} = p_{z,x}$
for all $x, y, z \in \bfP$ with $x \le y \le z$.
The identity $\id_x$ at an object $x \in \bfP$ is
given by $\id_x = p_{x,x}$.
\begin{enumerate}
\item
The \emph{incidence category} $\k[\bfP]$ of $\bfP$ is
defined as the $\k$-linearization of the category $\bfP$.
Namely, it is
a $\k$-linear category defined as follows.
The set of objects $\k[\bfP]_0$ is equal to $\bfP$,
for each pair $(x,y) \in \bfP \times \bfP$, the set of morphisms $\k[\bfP](x,y)$ is the vector space with basis $\bfP(x,y)$; thus it is a one-dimensional vector space $\k p_{y,x}$ if $x \le y$, or zero otherwise. The composition is defined as the $\k$-bilinear extension of that of $\bfP$. Note that $\k[\bfP]$ is a finite $\k$-linear category.

\item
Covariant ($\k$-linear) functors $\k[\bfP] \to \mod\k$ are called
\emph{persistence modules} over or indexed by $\bfP$.
\end{enumerate}
In the sequel,
we set $[\le]_\bfP\coloneqq \{(x,y) \in \bfP \times \bfP \mid x \le y\}$,
and $A\coloneqq \k[\bfP]$ (therefore, $A_0 = \bfP$),
and so the category of finite-dimensional persistence modules is denoted by $\mod A$.
\end{dfn}

\begin{dfn}
\label{dfn:ss-intervals}
Let $I$ be a nonempty full subposet of $\bfP$.
\begin{enumerate}
\item
For any $(x, y) \in [\le]_\bfP$, we set $[x,y]\coloneqq \{z \in \bfP \mid x \le z \le y\}$, and call it the \emph{segment} from $x$ to $y$ in $\bfP$.
\item
A \emph{source} (resp.\ \emph{sink}) of $I$ is nothing but a minimal (resp.\ maximal) element in $I$.
The set of all sources (sinks) of $I$ is denoted by $\src(I)$
(resp.\ $\snk(I)$).
If $I$ has the maximum (resp.\ minimum) element, then it is denoted by $\max(I)$ (resp.\ $\min(I)$). By convention, we set $\src (\emptyset) \coloneqq \emptyset$
and $\snk (\emptyset) \coloneqq \emptyset$.
\item
$I$ is said to be \emph{connected} if for all $x, y\in I$, there is a sequence of elements $x = z_0, z_1, \ldots, z_{n-1}, z_n = y$ in $I$ satisfying that every two consecutive elements $z_i$ and $z_{i+1}$ are comparable. Namely, either $z_i \leq z_{i+1}$ or $z_{i+1}\leq z_i$ holds for $i = 0, \ldots, n-1$.
\item
$I$ is said to be \emph{convex} if
for any $x, y \in I$ with $x \le y$, we have $[x,y] \subseteq I$.
\item
$I$ is called an \emph{interval} if $I$ is connected and convex.
\item
The set of all intervals of $\bfP$ is denoted by $\bbI(\bfP)$, or simply by $\bbI$. We regard $\bbI$ as a poset $\bbI = (\bbI, \le)$ with the inclusion relation: $I \le J \Leftrightarrow I \subseteq J$
for all $I, J \in \bbI$.
Since $\bfP$ is finite, $\bbI$ is also finite.
\end{enumerate}
\end{dfn}

Note that any segment $[x,y]$ is an interval with source $x$ and sink $y$. Following \cite{BBH2024approximations}, we introduce the subsequent definition.

\begin{dfn}
\label{dfn:interval-antichain-exp}
A subset $K$ of $\bfP$ is called an \emph{antichain} in $\bfP$ if every two distinct elements of $K$ are incomparable under the partial order of $\bfP$. We denote by $\Ac(\bfP)$ the set of all antichains in $\bfP$.
For any $K, L \in \Ac(\bfP)$, we define $K \le L$ if for all $x \in K$, there exists $z_x \in L$ such that $x \le z_x$, and for all $z \in L$, there exists $x_z \in K$ such that $x_z \le z$. In this case, we define $[K, L]\coloneqq \{y \in \bfP \mid x \le y \le z \text{ for some }
x \in K \text{ and for some }z \in L\}$. Therefore, every interval $I$ of $\bfP$ can be expressed by $[\src(I), \snk(I)]$.
\end{dfn}

To each interval, one may associate a persistence module as follows.

\begin{dfn}
\label{dfn:intv}
Let $I$ be an interval of $\bfP$.
\begin{enumerate}
\item
A persistence module $V_I$ over $\bfP$ is defined as follows. For $x\in \bfP$,
\[
V_{I}(x)\coloneqq\begin{cases}
	\k, & \textnormal{if } x\in I,\\
	0, & \textnormal{otherwise},
\end{cases}
\]
and for $p\in \k[\bfP](x, y)$,
\[
V_{I}(p)\coloneqq\begin{cases}
	k \id_\k, & \textnormal{if } (x, y)\in [\leq]_{I}\ \textnormal{and } p \coloneqq kp_{y,x}\ \textnormal{for some } k\in \k,\\
	0, & \textnormal{otherwise}.
\end{cases}
\]
It is easy to check that $V_I$ is indecomposable.
\item
A persistence module isomorphic to $V_I$ for some $I \in \bbI$
is called an \emph{interval module}.
\item
A persistence module is said to be \emph{interval-decomposable} if
it is isomorphic to a finite direct sum of interval modules.
Thus 0 is trivially interval-decomposable.
\end{enumerate}
\end{dfn}

We will use the notation $d_M(L)$ to denote
the multiplicity of an indecomposable direct summand $L$
of a module $M$ in its indecomposable decomposition
as explained in the following well-known theorem.

\begin{thm}[Krull--Schmidt]
\label{thm:KS}
Let $\calC$ be a finite $\k$-linear category, and
fix a complete set $\calL=\calL_\calC$
of representatives of isoclasses of indecomposable
objects in $\mod \calC$.
Then every finite-dimensional left $\calC$-module $M$ is
isomorphic to the direct sum $\Ds_{L \in \calL} L^{d_M(L)}$
for some unique function $d_M \colon \calL \to \bbZ_{\ge 0}$.
Therefore another finite-dimensional left $\calC$-module $N$
is isomorphic to $M$ if and only if $d_M = d_N$.
In this sense, the function $d_M$ is a complete invariant
of $M$ under isomorphisms.
\end{thm}
In the sequel, we simply call $d_{M}(L)$ the interval multiplicity of $L$ (in $M$) whenever $L$ is an interval module. In one-parameter persistent homology, this function
$d_M$ corresponds to the persistence diagram of $M$.

\begin{ntn}
\label{ntn:Yoneda}
Let $M \in \mod A$, and $x, y \in \bfP$.
\begin{enumerate}
\item
We set $P_x\coloneqq A(x,\blank)$ (resp.\ $P'_x\coloneqq A\op(x,\blank)$)
to be the projective indecomposable $A$-module (resp.\ $\k[\bfP\op]$-module) corresponding to the vertex $x$, and $Q_x\coloneqq D(A(\blank, x))$ (resp.\ $Q'_x\coloneqq D(A\op(\blank, x))$) to be the injective indecomposable $A$-module (resp.\ $A\op$-module) corresponding to the vertex $x$.

\item
By the Yoneda lemma, we have an isomorphism
\[
M(x) \to \Hom_{A}(P_x, M),\quad
m \mapsto \ro^M_m\ (m \in M(x)),
\]
where $\ro^M_m \colon P_x \to M$ is a morphism
$(\ro^M_{m,y} \colon P_x(y) \to M(y))_{y \in \bfP}$
defined by
$\ro^M_{m,y}(p)\coloneqq M(p)(m)\, (= p\cdot m)$ for all $y \in \bfP$
and $p \in P_x(y) = A(x,y)$,
where $M(p) \colon M(x) \to M(y)$ is a structure linear map of $M$.
Sometimes we just write $\ro^M_{m}(p)\coloneqq M(p)(m)$ by omitting $y$.

Similarly, by considering an $A\op$-module $N$ to be a right $A$-module,
we have an isomorphism
\[
N(x) \to \Hom_{A\op}(P'_x, N),\quad
m \mapsto \la^N_m\ (m \in N(x)),
\]
where $\la^N_m \colon P'_x \to N$ is defined by
$\la^N_m(p)\coloneqq N(p)(m)\, (= m \cdot p)$.

\item
For a morphism $p_{y,x} \colon x \to y$ in $\bfP$,
we set $M_{y,x}$ to be the linear map $M(p_{y,x}) \colon M(x) \to M(y)$.

\item
Since $p_{y,x} \in A(x,y) = P_x(y)$,
we can set $\sfP_{y,x}\coloneqq \ro_{p_{y,x}}\coloneqq \ro_{p_{y,x}}^{P_x} \colon P_y \to P_x$.
We note here that $P_{y,x} = 0$ if $x \nleq y$ in $\bfP$.
Similarly, we set $p\op_{x,y}\coloneqq p_{y,x} \in \bfP\op(y,x) = \bfP(x,y)$
for all $(x,y) \in [\le]_\bfP$.
It induces a morphism $\sfP'_{x,y}\coloneqq \ro_{p\op_{x,y}} \colon P'_x \to P'_y$ in $\mod A\op$.
\end{enumerate}
\end{ntn}

To shorten the formula, we introduce the following notation.

\begin{ntn}
\label{ntn:notation-M(morphism)}
Let $M \in \mod A$, and
$$
\mu \colon \Ds_{j\in [n]} P_{y_j} \to \Ds_{i\in [m]} P_{x_i}
$$
a morphism between projective modules of the form
$\mu\coloneqq \bmat{a_{ji}\sfP_{y_j,x_i}}_{(i,j) \in [m]\times [n]}$ with $a_{ji} \in \k,\ ((i,j) \in [m]\times [n])$
for some $x_1, x_2 \dots, x_m, y_1, y_2, \dots, y_n \in \bfP$.
Then we set $M(\mu) \colon \Ds_{i \in [m]}M(x_i) \to \Ds_{j \in [n]}M(y_j)$ to be the linear map defined by the matrix
$$
M(\mu)\coloneqq \bmat{a_{ji}M_{y_j, x_i}}_{(j,i) \in [n]\times [m]}.
$$
The right-hand side is obtained from $\mu$ by first replacing the letter $\sfP$ with $M$, and then making the transpose.
\end{ntn}

We here make the following remark.

\begin{rmk}
\label{rmk:form-add-hull}
Since both $\sfP_{y,x} = \ro_{p_{y,x}} \colon P_y \to P_x$ and
$M_{y,x} = M(p_{y,x}) \colon M(x) \to M(y)$
are defined by $p_{y,x}$ for all $x, y \in \bfP$,
we can think that both $\mu$ and $M(\mu)$ come from the common matrix
$\al\coloneqq \bmat{a_{ji}p_{y_j,x_i}}_{(j,i) \in [n] \times [m]}$
with entries in $\k[\bfP](x_i, y_j) = \k p_{y_j, x_i}$
regarded as a morphism
$(x_1, \dots, x_m) \to (y_1, \dots, y_n)$, where we set $a_{ji}p_{y_j,x_i}\coloneqq 0$ unless
$x_i \le y_j$.
This point of view is formalized as the formal additive hull $\Ds \k[\bfP]$ of $\k[\bfP]$.
We refer this to~\cref{section-4}.
By this reason, we also set
$\sfP(\al)\coloneqq \mu$ and $M(\al)\coloneqq M(\mu)$, more explicitly
\begin{equation}
\label{eq:dfn-sfP-M}
\begin{aligned}
\sfP(\bmat{a_{ji}p_{y_j,x_i}}_{(j,i) \in [n] \times [m]})&\coloneqq
\bmat{a_{ji}\sfP_{y_j,x_i}}_{(i,j) \in [m] \times [n]}
\colon \Ds_{j\in [n]}P_{y_j} \to \Ds_{i\in [m]}P_{x_i},\\
M(\bmat{a_{ji}p_{y_j,x_i}}_{(j,i) \in [n] \times [m]})&\coloneqq
M(\bmat{a_{ji}\sfP_{y_j,x_i}}_{(i,j) \in [m] \times [n]})
\colon \Ds_{i \in [m]}M(x_i) \to \Ds_{j\in [n]}M(y_j).
\end{aligned}
\end{equation}
In the above, notice the difference of $p$ and $\sfP$,
and also the positions of $i, j$ and $[m], [n]$.
This formulation makes it possible to unify all various cases of formulas
\eqref{eq:formula-d_M(V_I)-general-non-inj} by using empty matrices.
\end{rmk}

We cite the following from \cite[Lemma~4.8]{asashiba2024interval} including its proof
for the convenience of the reader.
For any $C, M \in \mod A$,
the following lemma makes it possible to compute
the dimension of $\Hom_{A}(C,M)$
by using a projective presentation of $C$ and the module structure of $M$. Later we will mainly apply this to the case where $C$ is a term of the almost split sequence starting from an interval module $V_I$ for some $I \in \bbI$.

\begin{lem}
\label{lem:dim-Hom-coker}
Let $C, M \in \mod A$.
Assume that $C$ has a projective presentation
\begin{equation}
\label{eq:proj-pres-C}
\Ds_{j\in [n]} P_{y_j} \ya{\mu} \Ds_{i\in [m]} P_{x_i} \ya{\ep} C
 \to 0
 \end{equation}
for some
$x_1, x_2 \dots, x_m, y_1, y_2, \dots, y_n \in \bfP$, and
$\mu\coloneqq \bmat{a_{ji}\sfP_{y_j,x_i}}_{(i,j) \in [m]\times [n]}$ with $a_{ji} \in \k,\ ((i,j) \in [m]\times [n])$.
Then we have
\begin{equation}
\label{eq:dimHom-formula}
\dim \Hom_{A}(C, M) = \sum_{i=1}^m \dim M(x_i) - \rank M(\mu).
\end{equation}
\end{lem}

\begin{proof}
Set $Y\coloneqq\Ds_{j\in [n]} P_{y_j}, X\coloneqq\Ds_{i\in [m]} P_{x_i}$ for short.
Then we have an exact sequence $Y \ya{\mu} X \ya{\ep} C \to 0$, which yields an exact sequence
\[
0 \to \Hom_{A}(C, M) \to \Hom_{A}(X, M) \ya{\Hom_{A}(\mu,M)} \Hom_{A}(Y,M).
\]
Hence $\Hom_{A}(C, M) \iso \Ker\Hom_{A}(\mu,M)$.
Now we have
\[
\begin{aligned}
&\Ker\Hom_{A}(\mu,M) = \{f \in \Hom_{A}(X,M) \mid f\mu = 0\}\\
&= \left.\left\{\bmat{f_1,\dots,f_m}\in \Hom_{A}(\Ds_{i\in [m]} P_{x_i},M) \right| \bmat{f_1,\dots,f_m}\cdot \bmat{a_{ji}\sfP_{y_j, x_i}}_{(i,j)} = 0 \right\}\\
&\iso \left.\left\{\sbmat{f_1\\\vdots\\f_m} \in \Ds_{i\in [m]} \Hom_{A}(P_{x_i},M) \right|
\bmat{\sum\limits_{i\in [m]} a_{ji} f_i \sfP_{y_j, x_i}}_{j\in [n]} = 0 \right\}\\
&\iso \left.\left\{\sbmat{b_1\\\vdots\\b_m} \in \Ds_{i\in [m]} M(x_i) \right|
\bmat{\sum\limits_{i\in [m]} a_{ji}M_{y_j, x_i}(b_i)}_{j\in [n]} = 0 \right\}\\
&= \left.\left\{\sbmat{b_1\\\vdots\\b_m} \in \Ds_{i\in [m]} M(x_i) \right| {}^t\!\!\left(\bmat{a_{ji}M_{y_j, x_i}}_{(i,j)}\right)\cdot \sbmat{b_1\\\vdots\\b_m} = 0 \right\}\\
& = \Ker\left(M(\mu)
\colon \Ds_{i\in [m]} M(x_i) \to \Ds_{j\in [n]} M(y_j)\right).
\end{aligned}
\]
Hence $\dim \Hom_{A}(C, M) = \sum_{i\in [m]} \dim M(x_i) -\rank M(\mu)$.
\end{proof}

\begin{rmk}
By using notation introduced in Remark \ref{rmk:form-add-hull}, the lemma above is stated as follows:
Let $C, M \in \mod A$, and assume that $C$ has a projective presentation of the form \eqref{eq:proj-pres-C}
with $\mu = \sfP(\al)$ for some matrix $\al = \bmat{\al_{y_j, x_i}}_{(j,i) \in [m]\times[n]}$
with entries in morphisms in $\k[\bfP]$.
Then $\dim \Hom_A(C,M) = \sum_{i=1}^m \dim M(x_i) - \rank M(\al)$.
\end{rmk}

\begin{dfn}
\label{dfn:2D-grid}
For each positive integer $n$, we denote by $\bbA_n$ the poset
$\{1, 2, \dots, n\}$ with the usual linear order $i < i+1\ (i = 1, 2, \dots, n-1)$, and for each poset
$P_1, P_2$, we regard the direct product $P_1 \times P_2$
as the poset with the partial order defined by
$(x, y) \le (x', y')$ if and only if
$x \le x'$ and $y \le y'$ for all $(x, y), (x', y') \in P_1 \times P_2$.
We set $G_{m,n}\coloneqq \bbA_m \times \bbA_n$, and call it a {\em 2D-grid}.
For example, $G_{5,2}$ has the following Hasse quiver:
\[
\begin{tikzcd}
(1,2) & (2,2) & (3,2) & (4,2) & (5,2)\\
(1,1) & (2,1) & (3,1) & (4,1) & (5,1)
\Ar{1-1}{1-2}{}
\Ar{1-2}{1-3}{}
\Ar{1-3}{1-4}{}
\Ar{1-4}{1-5}{}
\Ar{2-1}{2-2}{}
\Ar{2-2}{2-3}{}
\Ar{2-3}{2-4}{}
\Ar{2-4}{2-5}{}
\Ar{2-1}{1-1}{}
\Ar{2-2}{1-2}{}
\Ar{2-3}{1-3}{}
\Ar{2-4}{1-4}{}
\Ar{2-5}{1-5}{}
\end{tikzcd}.
\]
\end{dfn}

\section{The formula of interval multiplicities}
\label{section-3}
Throughout this section,
$I$ is an interval of $\bfP$ and $M \in \mod A$.
The purpose of this section is to compute the multiplicity $d_M(V_I)$ of $V_I$ in $M$.

\begin{dfn}
\label{dfn:up-set and down-set}
Let $x\in \bfP$, and let $U$ be a subset of $\bfP$. Then we set
$$
\begin{aligned}
\uset_\bfP\, x\coloneqq \{y \in \bfP \mid y \ge x\}, & \text{ and}\ \dset_\bfP\, x\coloneqq \{y \in \bfP \mid y \le x\},\\
\uset_\bfP\, U\coloneqq \bigcup_{x \in U} \uset_\bfP x, & \text{ and}\
\dset_\bfP\, U\coloneqq \bigcup_{x \in U} \dset_\bfP x.
\end{aligned}
$$
We say that $U$ is an up-set (resp. a down-set) of $\bfP$
if $U = \uset_\bfP\, U$ (resp. $U = \dset_\bfP\, U$).
Up-sets and down-sets are naturally considered as full subposets of $\bfP$ without specifying in the sequel.
\end{dfn}

\begin{rmk}
\label{rmk: remarks of up-set and down-set}
(1) Since $\bfP$ is finite, any up-set $U$ can be written as $U = \uset_\bfP\, \src(U) = \bigcup_{x\in \src(U)}\uset_\bfP\, x$. Dually, any down-set $U$ can be written as $U = \dset_\bfP\, \snk(U) = \bigcup_{x\in \snk(U)}\dset_\bfP\, x$.

(2) If $X = \uset_\bfP\, U$, then $\src(X) = \src(U)$. Dually, if  $X = \dset_\bfP\, U$, then $\snk(X) = \snk(U)$.

(3) It is easy to see that $\uset_\bfP\! \uset_\bfP\, U = \uset_\bfP\, U$ and $\dset_\bfP\! \dset_\bfP\, U = \dset_\bfP\, U$.

(4) If $U$ is an up-set (resp. down-set) of $\bfP$ and $x$ is any element of $U$, then $\uset_{\bfP}\, x = \uset_{U}\, x$ (resp. $\dset_{\bfP}\, x = \dset_{U}\, x$).
\end{rmk}

We simply write $\uset, \dset$ for $\uset_{\bfP}, \dset_{\bfP}$,
respectively if there seems to be no confusion.
To compute $d_M(V_I)$, we apply \cite[Theorem 3]{Asashiba2017} below.

\begin{thm}
\label{thm:dMV}
Let $M$ and $L$ be two finite-dimensional modules
over a finite-dimensional algebra $A$, and assume that $L$ is indecomposable. When $L$ is non-injective, let
\begin{equation}
\label{eq:almost-split-sequence}
0 \to L \to E \to \ta\inv L \to 0
\end{equation}
be an almost split sequence starting
from $L$.
Then we have the following formulas.
\begin{enumerate}[label=\rm (\arabic*)]
\item
If $L$ is injective, then
\begin{equation}
\label{eq:dMV-inj}
d_M(L) = \dim\Hom_A(L, M) - \dim\Hom_A(L/\soc L, M).
\end{equation}
\item
If $L$ is non-injective, then
\begin{equation}
\label{eq:dMV-noninj}
d_M(L) = \dim\Hom_A(L, M) - \dim\Hom_A(E, M) + \dim\Hom_A(\ta\inv L, M).
\end{equation}
\end{enumerate}
\end{thm}

In the next subsection, we will give our result in general case.
This will be specialized in~\cref{ssec:2Dgrid-case} for the case of 2D-grids.
The latter has a simpler formula and easier to grasp than the former.
The reader may read~\cref{ssec:2Dgrid-case} first by looking at
Example \ref{exm:2Dgrid} to have rough outline.
It contains enough information to apply the formula for 2D-grids.
The details of proofs written in~\cref{ssec:formula-gen-case} can be read afterward.

\subsection{The general poset case}
\label{ssec:formula-gen-case}

Without loss of generality, we may assume that the poset $\bfP$ is connected.

\begin{ntn}
\label{ntn: notations for general case}
For a totally ordered set $T$, we set
$\sub_2T\coloneqq \{\{i,j\} \subseteq T \mid i \ne j\}$
to be the set of two-element subsets of $T$, and
for any $\bfa \in \sub_2T$,
we set $\udl{\bfa}\coloneqq\min \bfa$ and
$\ovl{\bfa}\coloneqq \max \bfa$.
Thus $\bfa = \{\udl{\bfa},\, \ovl{\bfa}\}$.

Let $U$ be a subset of $\bfP$.
\begin{enumerate}
\item
Suppose $|\src(U)| = n$ (resp. $|\snk(U)| = m$). We give a total order on the set $\src(U)$ (resp.\ $\snk(U)$)
by giving a poset isomorphism
$a \colon [n] \to \src(U)$, $i \mapsto a_i$
(resp.\ $b \colon [m] \to \snk(U)$, $i \mapsto b_i$), and
apply the notation above for $T = \src(U)$ and $\snk(U)$.
\item
Let $\{a, b\} \in \bfP$ with $a \ne b$.
Then we denote by $\vee'\{a, b\}$ (resp.\ $\wedge'\{a, b\}$)
the set of minimal upper (resp.\ maximal lower) bounds of $a$ and $b$ in $\bfP$.
Since $\bfP$ is a finite poset,
if $\vee'\{a, b\}$ (resp.\ $\wedge'\{a, b\}$)
consists of a single element, then it coincides with the join
$a \vee b = \vee\{a, b\}$ (resp.\ the meet $a \wedge b = \wedge\{a, b\}$).
We call $\vee'\{a, b\}$  (resp.\ $\wedge'\{a, b\}$)
the \emph{pre-join} (resp.\ \emph{pre-meet}) of $\{a, b\}$.
For example, if $\bfP$ is presented by the following Hasse diagram,
then the pre-join of $\{a, b\}$ is given by $\vee'\{a,b\} = \{c, d\}$:
$$
\begin{tikzcd}
& e \\
c && d \\
a & {} & b
\arrow[from=2-1, to=1-2]
\arrow[from=2-3, to=1-2]
\arrow[from=3-1, to=1-2]
\arrow[from=3-1, to=2-1]
\arrow[from=3-1, to=2-3]
\arrow[from=3-3, to=1-2]
\arrow[from=3-3, to=2-1]
\arrow[from=3-3, to=2-3]
\end{tikzcd}
$$
We adopt this notation
to each $\bfa \in \sub_2\src(U)$
(resp.\ $\bfb \in \sub_2\snk(U)$).
Thus $\vee'\bfa$ (resp.\ $\wedge'\bfb$) is the pre-join of $\bfa$
(resp.\ the pre-meet of $\bfb$), more explicitly
$$
\vee'\bfa \coloneqq \src (\uset_{\bfP}\, \udl{\bfa} \cap \uset_{\bfP}\, \ovl{\bfa}),\quad
\wedge'\bfb \coloneqq \snk (\dset_{\bfP}\, \udl{\bfb} \cap \dset_{\bfP}\, \ovl{\bfb}).
$$
We fix a linear order on $\vee'\bfa$ (resp.\ $\wedge'\bfb$)\footnote{
For example, we fix a linear order on $\bfP$, and define a linear order on every
subset $S$ of $\bfP$ in such a way that the inclusion from $S$ to $\bfP$ becomes
an order-preserving map.}.
Note that if $U$ is an up-set (resp.\ a down-set),
then $\vee'\bfa \subseteq U$ (resp. $\wedge'\bfb \subseteq U$) by the definition of up-sets (resp. down-sets).
\item We further set $\src_1(U)$ (resp.\ $\snk_1(U)$) to be the disjoint union of
the pre-joins (pre-meets) of the two-element subsets of $\src(U)$ (resp.\ $\snk(U)$):
\[
\begin{aligned}
\src_{1}(U)&\coloneqq\bigsqcup_{\bfa \in \sub_2\src(U)}\!\! \vee'\bfa
\ = \{\bfa_c\coloneqq(\bfa, c) \mid \bfa \in \sub_2\src(U),\, c \in \vee'\bfa\},\\
\snk_{1}(U)&\coloneqq\bigsqcup_{\bfb \in \sub_2\snk(U)}\!\! \wedge'\bfb
\ = \{\bfb_d\coloneqq (\bfb, d) \mid \bfb \in \sub_2\snk(U),\, d \in \wedge'\bfb\}.
\end{aligned}
\]
Note here that the family $(\vee' \bfa)_{\bfa \in \sub_2\src(U)}$
(resp.\ $(\wedge' \bfb)_{\bfb \in \sub_2\snk(U)}$) does not need to be disjoint.
For example, consider the case where $U$ is presented as follows:
$$
\begin{tikzcd}
	d && e \\
	a & b & c
	\arrow[from=2-1, to=1-1]
	\arrow[from=2-2, to=1-1]
	\arrow[from=2-2, to=1-3]
	\arrow[from=2-3, to=1-1]
	\arrow[from=2-3, to=1-3]
\end{tikzcd}
$$
Then $\sub_2\src(U) = \{\{a,b\}, \{a,c\},\{b,c\}\}$,
and $\vee'\{a,b\} = \{d\}$, $\vee'\{a,c\} = \{d\}$, $\vee'\{b,c\} = \{d,e\}$.
Hence $\src_1(U) = \{\{a,b\}_d,\, \{a,c\}_d,\, \{b,c\}_d,\, \{b,c\}_e\}$.

By the definition above, we see that
if $\src(W) = \src(U)$ (resp.\ $\snk(W) = \snk(U)$) for a $W \subseteq \bfP$,
then $\src_1(W) = \src_1(U)$ (resp.\ $\snk_1(W) = \snk_1(U)$).

Furthermore, we equip $\src_{1}(U)$ with another total order $\plex$, defined by
$\bfa_c \plex \bfa'_{c'}$ if and only if
$(\udl{\bfa}, \ovl{\bfa}, c) \le_{\lex} (\udl{\bfa'}, \ovl{\bfa'}, c')$, where $\leq_{\lex}$ denotes the lexicographic order from left to right.
Similarly, we give a total order to $\snk_{1}(W)$.
These total orders will be used to express matrices having
$\src_{1}(U)$ or $\snk_{1}(W)$ as an index set later.

\item For any non-empty subset $X$ of $\bfP$,
or a disjoint union $X = \bigsqcup_{s \in S} X_s\coloneqq \{s_x \mid s \in S, \, x \in X_s\}$
of non-empty subsets $X_s$ of $\bfP$ with non-empty
index set $S$, we set
$P_X\coloneqq \Ds_{t \in X}P_t$ and $P'_X\coloneqq \Ds_{t \in X}P'_t$,
where $P_{t}\coloneqq P_x,\, P'_{t}\coloneqq P'_x$
if $t = s_x \in X = \bigsqcup_{s \in S} X_s$ with $s \in S$ and $x \in X_s$.
In addition, we set $P_X$ and $P'_X$ to be the zero modules if $X=\emptyset$.
\end{enumerate}
\end{ntn}

The following lemmas are necessary in the sequel.

\begin{lem}
Let $U$ be either an up-set or a down-set of $\bfP$.
Then $U$ is convex in $\bfP$.
Moreover, if $U$ is connected, then $U$ is an interval.
\qed
\end{lem}

\begin{lem}
\label{lem:up-set of conn is again conn}
Let $I$ be a connected subposet of $\bfP$.
Then both $\uset I$ and $\dset I$ are again connected.
\qed
\end{lem}

\begin{lem}
\label{lem:relation between interval and up-set (down-set)}
Let $I$ be an interval of $\bfP$.
Then there exist up-sets $U,\, U'$ and down-sets $W,\, W'$ of $\bfP$ such that
$I = U\setminus U' = W\setminus W'$, which are given by
\begin{equation}
\label{eq:dfn-of-uuset-ddset}
U\coloneqq \uset I,\, U'\coloneqq \uuset I\coloneqq \uset I \setminus I; \text{ and }
W\coloneqq \dset I,\, W'\coloneqq \ddset I\coloneqq \dset I \setminus I.
\end{equation}
\end{lem}

\begin{proof}
We only need to show that
$\uuset I\coloneqq \uset I \setminus I$ is an up-set of $\bfP$.
Take any $u\in \uuset I$ and $x\in \bfP$ with $u\le x$.
Then there exists $y\in I$ with $y\le u$ since $u\in \uset I$.
Hence $I \ni y\le u\le x$, which shows that $x\in \uset I$.
If $x \in I$, then the convexity of $I$ shows that $u \in I$, a contradiction.
Thus $x\notin I$, and hence $x \in \uuset I$.
The proof for $\ddset I\coloneqq \dset I \setminus I$ to be a down-set of $\bfP$ is similar.
\end{proof}

\begin{rmk}
We note here that $\uset I$ (resp.\ $\dset I$) is connected,
thus an interval by the previous lemmas.
However, $\uuset I$ (resp.\ $\ddset I$) may not be connected in general.
\end{rmk}

Let $U$ be a connected up-set, and $W$ a connected down-set of $\bfP$. Note that $V_U$ is projective (resp.\ $V_W$ is injective) if and only if
$|\src(U)| = 1$ (resp. $|\snk(W)| = 1$) because $V_U$ and $V_W$ are indecomposable modules and
$\dim V_U/\rad V_U = |\src(U)|$ (resp. $\dim \soc V_W = |\snk(W)|$).
To show that the set $\src_1(U)$ (resp.\ $\snk_1(W)$) is not empty
if $V_U$ is not projective (resp.\ $V_W$ is not injective),
we review a fundamental property of finite posets.

\begin{dfn}
\label{dfn:Alexandrov}
Let $S$ be a finite poset. A topology on $S$ is defined by setting the set of up-sets to be the open sets
of $S$,
which is called the \emph{Alexandrov topology} on $S$.
It is easy to see that it has a basis $\{\uset_S x \mid x \in S\}$.
\end{dfn}

The following lemmas are easy to show and the proofs are left to the reader.

\begin{lem}
\label{lem:top-conn}
Let $S$ be a finite poset considered as a topological space by
the Alexandrov topology on $S$.
Then $S$ is a connected space if and only if $S$ is a connected poset.
\qed
\end{lem}

\begin{lem}
\label{lem:local-conn}
Let $S$ be a finite poset considered as a topological space by
the Alexandrov topology on $S$.
Then $S$ is a locally connected space.
\qed
\end{lem}

Under the preparation above, we prove the following.

\begin{prp}
Let $U$ be a connected up-set, $W$ a connected down-set of $\bfP$.
\begin{enumerate}[label=\rm (\arabic*)]
\item
$|\src(U)| \ge 2$ if and only if $\src_1(U) \ne \emptyset$.
\item
$|\snk(W)| \ge 2$ if and only if $\snk_1(W) \ne \emptyset$.
\end{enumerate}
\end{prp}

\begin{proof}
(1) Since the implication ($\Leftarrow$) is trivial, we show the implication ($\Rightarrow$). Set $\src(U) = \{a_1, \dots, a_n\}$ and
assume that $n \ge 2$. By Remark~\ref{rmk: remarks of up-set and down-set} (1), we have
\begin{equation}
\label{eq:cup-QI}
U = \uset a_1 \cup (\uset a_2 \cup \cdots \cup \uset a_n).
\end{equation}
Now suppose that $\src_1(U) = \emptyset$.
Then for any $\{i, j\} \in \sub_2[n]$,
we have $\src(\uset a_i \cap \uset a_j) = \emptyset$, and hence $\uset a_i \cap \uset a_j = \emptyset$.
This shows that
\begin{equation}
\label{eq:cap-QI}
\uset a_1 \cap (\uset a_2 \cup \cdots \cup \uset a_n) = \emptyset.
\end{equation}
The equalities \eqref{eq:cup-QI} and \eqref{eq:cap-QI} show that
the topological space $U$ with Alexandrov topology
is not connected (also by noticing Remark~\ref{rmk: remarks of up-set and down-set} (4)).
Hence $U$ is not a connected poset by Lemma \ref{lem:top-conn},
a contradiction. As a consequence, $\src_1(U) \ne \emptyset$.

(2) This is shown similarly.
\end{proof}

With the above preliminaries, we first give a projective presentation of $V_U\in \mod A$, where $U$ is a connected up-set of $\bfP$.

\begin{prp}
\label{prp:proj.pre.for conn. up-set module}
Let $U$ be a connected up-set of $\bfP$.
Then $V_U$ has the following (not necessarily minimal) projective presentation:
$$
P_{\src_1(U)} \ya{\ep_1^U} P_{\src(U)} \ya{\ep_0^U} V_U \to 0,
$$
where $\ep_0^U = \bmat{\ro^{V_U}_{1_a}}_{a\in \src(U)}$, and we set
$1_u \coloneqq 1_\k \in \k = V_U(u)$ for all $u \in U$, and
\begin{align}\label{eq:matrix form of ep^U_1}
\ep^{U}_{1}\coloneqq
\bmat{\tilde{\sfP}_{a,\bfa_c}}_{(a,\bfa_c) \in \src(U) \times \src_1(U)}
\end{align}
with the entries given by
\begin{align}\label{eq:entries of ep^U_1}
\tilde{\sfP}_{a,\bfa_c}\coloneqq
\begin{cases}
\sfP_{c,a} & (a = \udl{\bfa}),\\
-\sfP_{c,a} & (a = \ovl{\bfa}),\\
\mathbf{0} & (a \not\in \bfa),
\end{cases}
\end{align}
for all $\bfa_c \in \src_1(U)$ and $a \in \src(U)$. Here and subsequently, we write the matrices following the lexicographic order $\plex$ (see~{\rm \cref{ntn: notations for general case} (3)}) of indices.

We here remark that $\ep^U_0$ is a projective cover of $V_U$.
\end{prp}

\begin{proof}
We refer the reader to the proof of~\cite[Proposition 5.10]{asashiba2024interval},
and substitute $I^\xi$ by $U$.
\end{proof}

Given an up-set $U$ of $\bfP$ ($U$ might be non-connected),
we consider its decomposition into connected components and apply Proposition~\ref{prp:proj.pre.for conn. up-set module} on each connected component.
Following this spirit, we let $U\coloneqq U_1\sqcup \cdots \sqcup U_k$.
By Lemma~\ref{lem:local-conn} and \cite[Theorem~25.3]{MR3728284},
each component is again an open set, thus an up-set of $\bfP$. Then the following is easy to show.

\begin{prp}
\label{prp:proj.pre.for general up-set module}
Let $\bfP$ be a finite poset,
and $U = U_1\sqcup \cdots \sqcup U_k$ an up-set of $\bfP$ with $k$ connected components $(k\ge 1)$. Set $V_U\coloneqq V_{U_1}\ds \cdots \ds V_{U_k}$. Then $V_U$ has the following (not necessarily minimal) projective presentation:
$$
P_{\src_1(U)} \ya{\ep_1^U} P_{\src(U)} \ya{\ep_0^U} V_U \to 0,
$$
where $\ep_i^U \coloneqq \ep_i^{U_1}\ds \cdots \ds \ep_i^{U_k}$ $(i=0,1)$.
\qed
\end{prp}

The following lemma will be frequently used later.

\begin{lem}
\label{lem:proj.presentation for factor module}
Let $M, M'\in \mod A$.
Assume that
\begin{enumerate}[font=\normalfont,label=(\arabic*)]
\item $M'$ is a submodule of $M$, i.e., we have a short exact sequence
$$
0\to M' \ya{\iota} M \ya{\pi} M/M' \to 0;
$$

\item $M$ has a projective presentation (not necessarily a minimal one)
$$
P^{M}_{1} \ya{\ep_1^M} P^{M}_{0} \ya{\ep_0^M} M \to 0; \text{and}
$$

\item $M'$ has an epimorphism $\ep_0^{M'}\colon P^{M'}_{0} \to M'$ with $P^{M'}_{0}$ projective.
\end{enumerate}
Then the factor module $M/M'$ has a projective presentation (not necessarily a minimal one)
$$
P^{M}_{1} \ds P^{M'}_{0}\ya{\eta_1} P^{M}_{0} \ya{\eta_0} M/M' \to 0.
$$
Here $\eta_0 \coloneqq \pi \ep_0^{M}$ and $\eta_1 \coloneqq \bmat{\ep_1^M, \eta_{11}}$, where
$\eta_{11} \colon P^{M'}_{0}\to P^{M}_{0}$
is a lift of $\iota \ep_0^{M'}$ along $\ep_0^M$, i.e.,
a morphism satisfying the equality $\ep_0^M \eta_{11}=\iota \ep_0^{M'}$,
the existence of which is guaranteed by the projectivity of $P_0^{M'}$.
\end{lem}

\begin{proof}
Express $\ep^{M}_1$ as the composite
$\ep^{M}_1 = \iota_0 \ta_1 \colon P^M_1 \ya{\ta_1} \Im \ep^M_1 \ya{\io_0} P^M_0$, where $\iota_0$ is the inclusion and $\ta_1$ is an epimorphism obtained from $\ep^{M}_1$ by restricting the codomain.
Consider the following commutative diagram of solid arrows
with exact rows surrounded by dashed lines:
\begin{equation}
\label{eq:snake-lem-nsrc}
\begin{tikzcd}
\Nname{lu}0 & P_1^M & P_1^M \ds P^{M'}_{0}  & P^{M'}_{0} &\Nname{ru} 0\\
\Nname{ld}0 & \Im \ep^M_1 & P^{M}_{0} & M & \Nname{rd}0\\
& 0 & \Cok \et_1 & M/M' & 0\\
&& 0 & 0
\Ar{1-1}{1-2}{}
\Ar{1-2}{1-3}{"{\sbmat{\id \\ \mathbf{0}}}"}
\Ar{1-3}{1-4}{"{\sbmat{\mathbf{0},\, \id}}"}
\Ar{1-4}{1-5}{}
\Ar{2-1}{2-2}{}
\Ar{2-2}{2-3}{"\iota_{0}"}
\Ar{2-3}{2-4}{"\ep^{M}_{0}"}
\Ar{2-4}{2-5}{}
\Ar{3-2}{3-3}{}
\Ar{3-3}{3-4}{"\overline{\ep^{M}_{0}}"}
\Ar{3-4}{3-5}{}
\Ar{1-2}{2-2}{"\ta_1"}
\Ar{1-3}{2-3}{"\et_1"}
\Ar{1-4}{2-4}{"\iota \ep_0^{M'}"}
\Ar{1-4}{2-3}{"\eta_{11}" ', dashed}
\Ar{2-2}{3-2}{}
\Ar{2-3}{3-3}{"\cok \et_1" '}
\Ar{2-4}{3-4}{"\pi"}
\Ar{3-3}{4-3}{}
\Ar{3-4}{4-4}{}
\ar[to path={([yshift=8pt]lu.north east)--([yshift=8pt]ru.north east)--
(rd.south east)--(ld.south west)--([yshift=8pt]lu.north west)--([yshift=8pt]lu.north east)}, dash, dashed, rounded corners]
\end{tikzcd}
\end{equation}
Notice that the right-most column
$$
P^{M'}_{0}\ya{\iota \ep_0^{M'}} M\ya{\pi} M/M'\to 0
$$
is exact since $\Im \iota \ep_0^{M'} = \Im \iota = M' = \Ker\pi$.
By applying the snake lemma to this diagram, we obtain that
$\overline{\ep^{M}_{0}}\colon \Cok\et_1 \to M/M'$ is an isomorphism.
Since $\overline{\ep^{M}_{0}}\circ \cok \et_1 = \pi \ep^{M}_{0} = \eta_{0}$,
the central column yields the exact sequence
\begin{equation*}
P^{M'}_{0}\ds P_1^M \ya{\et_1} P^{M}_{0}\ya{\eta_{0}} M/M'\to 0. \tag*{\qedhere}
\end{equation*}
\end{proof}

\begin{ntn}
\label{ntn:choice function}
Let $I$ be an interval of $\bfP$.

(1) Note that $\src(I) = \src(\uset I)$, and hence also $\src_1(I) = \src_1(\uset I)$.
Therefore, $\ep_1^{\uset I} \colon P_{\src_1(\uset I)} \to P_{\src(\uset I)}$ is denoted by
$\ep_1^{\uset I} \colon P_{\src_1(I)} \to P_{\src(I)}$, where the definition of $\ep_1^{\uset I}$
is given by \eqref{eq:matrix form of ep^U_1} for $U\coloneqq \uset I$.

(2) Note that for each $a'\in \src(\uuset I)$, we have $\src(I)\,\cap\, \dset a'\neq \emptyset$
because $a' \in \uset I$.
Fixing one element in $\src(I)\,\cap\, \dset a'$ for each $a'\in \src(\uuset I)$
yields a map $\bfc\colon \src(\uuset I)\to \src(I) = \src(\uset I)$. We call such $\bfc$ a \emph{choice map}.
\end{ntn}

The following is used in the computations below.

\begin{lem}
\label{lem:comp-ro-P}
Let $M \in \mod A$, $x, y \in \bfP$, and $a \in M(x)$.
Then the composite of morphisms $P_y \ya{\sfP_{y,x}} P_x \ya{\ro^M_a} M$ is given by
$$
\ro^M_a \cdot \sfP_{y,x} = \ro^M_{M(p_{y,x})(a)}.
$$
\end{lem}

\begin{proof}
Let $z \in \bfP$ and $q \in P_y(z)$.
Then by Notation \ref{ntn:Yoneda}, we have
$$
(\ro^M_a \cdot \sfP_{y,x})(q)
= \ro^M_a(q\cdot p_{y,x})
= M(q\cdot p_{y,x})(p)
= (M(q)M(p_{y,x}))(p).
$$
On the other hand,
$$
\ro^M_{M(p_{y,x})(a)}(q) = M(q)(M(p_{y,x})(a)).
$$
Therefore, the assertion holds.
\end{proof}

Now we are in a position to give a projective presentation of $V_I$ for any interval $I$ of $\bfP$.

\begin{prp}
\label{prp:proj.presentation for V_I}
Let $I$ be an interval of $\bfP$.
Then $V_I$ has the following (not necessarily minimal) projective presentation:
\begin{equation}
\label{eq:pp-V_I}
P_{\src_1(I)}\ds P_{\src(\uuset I)} \ya{\ep_1} P_{\src(I)} \ya{\ep_0} V_I \to 0.
\end{equation}
Here $\ep_0 = \bmat{\ro^{V_I}_{1_a}}_{a \in \src(I)}$, where we set
$1_u \coloneqq 1_\k \in \k = V_I(u)$ for all $u \in I$;
$\ep_1 \coloneqq \ep_1(\bfc)\coloneqq \bmat{\ep^{\uset I}_{1}, \ep_{11}}$,
where $\ep_{11} \coloneqq \ep_{11}(\bfc) \coloneqq \bmat{\de_{a,\bfc(a')}\sfP_{a',\bfc(a')}}_{(a,a') \in \src(I) \times \src(\uuset I)}$, and
$\ep^{\uset I}_{1}$ is defined as in~{\rm \cref{ntn:choice function}}.
Note that $\ep_0$ is a projective cover of $V_I$.
\end{prp}

\begin{proof}
By Lemma~\ref{lem:relation between interval and up-set (down-set)},
we can write $I = U\setminus U'$ in terms of two up-sets $U=\uset I$ and $U'=\uuset I$,
and hence $V_I \iso V_{U}/V_{U'}$.
By Proposition~\ref{prp:proj.pre.for conn. up-set module}
and Lemma~\ref{lem:up-set of conn is again conn},
$V_{U}$ has the following (not necessarily minimal) projective presentation:
\begin{equation}
\label{eq:projective presentation for V_uset_I}
P_{\src_1(U)} \ya{\ep_1^{U}} P_{\src(U)} \ya{\ep_0^{U}} V_{U} \to 0.
\end{equation}
By Remark~\ref{rmk: remarks of up-set and down-set} (2),
$P_{\src(U)} = P_{\src(I)}$, and $P_{\src_1(U)} = P_{\src_1(I)}$.
Since the natural projection $\pi\colon V_{U}\to V_I$ is just the restriction on $I$, it follows
by Lemma~\ref{lem:proj.presentation for factor module}, that
$\ep_0 = \pi \ep_0^{U}
= \pi \circ \bmat{\ro^{V_U}_{1_a}}_{a \in \src(U)}
= \bmat{\ro^{V_I}_{1_a}}_{a \in \src(I)}$.

On the other hand, the up-set $U'$ is the disjoint union $U_1\sqcup \cdots \sqcup U_k$
of some connected up-sets $U_1, \dots, U_k$ with $k\in \bbZ_{\ge 1}$,
and hence $V_{U'} \iso  V_{U_1}\ds \cdots \ds V_{U_k}$.
By Propositions~\ref{prp:proj.pre.for conn. up-set module}
and \ref{prp:proj.pre.for general up-set module},
there is an epimorphism starting from a projective module:
\begin{equation}
P_{\src(U')} \ya{\ep_0^{U'}} V_{U'}, \label{eq:epi for V_U'}
\end{equation}
where $\ep_0^{U'} \coloneqq \bmat{\ro^{V_{U_1}}_{1_a}}_{a \in \src(U_1)} \ds \cdots \ds \bmat{\ro^{V_{U_k}}_{1_a}}_{a \in \src(U_k)} = \bmat{\ro^{V_{U'}}_{1_a}}_{a \in \src(U')}$.
If the equality
\begin{equation}
\label{eq:lift-U}
\ep_0^{U} \ep_{11} = \iota \ep_0^{U'} \colon P_{\src(U')} \to V_U
\end{equation}
holds,
where $\iota\colon V_{U'} \iso  V_{U_1}\ds \cdots \ds V_{U_k}\to V_{U}$ is the inclusion,
then by
combining \eqref{eq:projective presentation for V_uset_I}, \eqref{eq:epi for V_U'}
and Lemma~\ref{lem:proj.presentation for factor module},
we obtain a projective presentation of $V_I$
having the form
\begin{equation}
  P_{\src_1(I)}  \ds P_{\src(U')} \ya{\bmat{\ep^{U}_{1}, \ep_{11}}} P_{\src(I)} \ya{\ep_0^I} V_I \to 0
\end{equation}
as claimed.
The equality \eqref{eq:lift-U} is verified as follows:
\begin{align*}
\ep_0^{U} \ep_{11}
& = \bmat{\ro^{V_U}_{1_a}}_{a \in \src(I)} \cdot \bmat{\de_{a,\bfc(a')}\sfP_{a',\bfc(a')}}_{(a,a') \in \src(I) \times \src(U')} \\
& = \bmat{\sum\limits_{a\in \src(I)} \de_{a,\bfc(a')}\ro^{V_U}_{1_a}\cdot \sfP_{a',\bfc(a')}}_{a'\in \src(U')}\\
& = \bmat{\ro^{V_U}_{1_{\bfc(a')}}\cdot \sfP_{a',\bfc(a')}}_{a'\in \src(U')}\\
& \overset{*}{=} \bmat{\ro^{V_U}_{V_U(p_{a',\bfc(a')})(1_{\bfc(a')})}}_{a'\in \src(U')}\\
& = \bmat{\ro^{V_U}_{1_{a'}}}_{a'\in \src(U')}\\
& = \iota\circ \bmat{\ro^{V_{U'}}_{1_{a'}}}_{a'\in \src(U')}\\
& =\iota \ep_0^{U'}.
\end{align*}
In the above, the equality ($\overset{*}{=}$) follows by Lemma \ref{lem:comp-ro-P}.
\end{proof}

\subsubsection{The case where \texorpdfstring{$V_I$}{VI} is injective}

Assume that $V_I$ is injective in $\mod A$.
Then $b\coloneqq \max I$ exists, and we have $I = \dset b$.
Since $V_I$ is indecomposable injective, $\soc V_I = V_{\{b\}}$ is a simple module
at $b$, and $\ep^{b}_0 = \ro_{1_{b}}\colon P_{b}\to V_{\{b\}}$
is the projective cover of $V_{\{b\}}$.
Hence by Lemma~\ref{lem:proj.presentation for factor module}
and Proposition~\ref{prp:proj.presentation for V_I},
we have the following.

\begin{ntn}
\label{ntn:choice function 2}
Let $I$ be an interval of $\bfP$ with the maximum element $b$. Fixing one element $a\in \src(I)$ induces another choice map $\bfc'\colon \{b\}\to \src(I)$ by $\bfc'(b) \coloneqq a$.
\end{ntn}

\begin{thm}
\label{thm:dMV-VI-inj-general}
Let $I$ be an interval of $\bfP$. Assume that $V_I$ is injective, i.e., $I = \dset b$ with $b = \max I$.
Then $V_I$ and $V_I/\soc V_I$ have projective presentations of the following forms.
$$
\begin{aligned}
&P_{\src_1(I)} \ds P_{\src(\uuset I)} \ya{\ep_1} P_{\src(I)} \ya{\ep_0} V_I \to 0,\\
&(P_{\src_1(I)}\ds P_{\src(\uuset I)}) \ds P_{b} \ya{\ep'_1} P_{\src(I)}  \to V_I/\soc V_I \to 0.
\end{aligned}
$$
Here $\ep_1 \coloneqq \bmat{\ep^{\uset I}_{1}, \ep_{11}}$ and
$\ep'_1 \coloneqq \bmat{\ep_{1}, \ep''_1}$,
where $\ep_{11}, \ep^{\uset I}_{1}$ are given in~{\rm \cref{prp:proj.presentation for V_I}},
and $\ep''_1\coloneqq \ep''_1(\bfc')\coloneqq \bmat{\de_{a, \bfc'(b)} \sfP_{b,\, \bfc'(b)}}_{(a, b)\in \src(I)\times \{b\}}$. Therefore, we have
\begin{equation}
d_M(V_I) = \rank\bmat{M(\ep_{1})\\M(\ep''_1)}
- \rank\bmat{M(\ep_{1})}.\label{eq:formula for injective modules}
\end{equation}
\end{thm}

\begin{proof}
We only need to show the statement for $V_I/\soc V_I$.
By Lemma~\ref{lem:proj.presentation for factor module},
it suffices to check $\ep_0 \ep''_{1}=\iota \ep^{\{b\}}_0$,
in which $\iota\colon V_{\{b\}} \to V_I$ is the inclusion.
\begin{align*}
\ep_0 \ep''_{1} & = \bmat{\ro^{V_I}_{1_a}}_{a \in \src(I)} \cdot \bmat{\de_{a, \bfc'(b)} \sfP_{b,\, \bfc'(b)}}_{(a, b)\in \src(I)\times \{b\}}
= \sum_{a\in \src(I)} \de_{a, \bfc'(b)}\ro^{V_I}_{1_a}\cdot \sfP_{b, a}\\
& = \ro^{V_I}_{1_{\bfc'(b)}}\cdot \sfP_{b, \bfc'(b)}
\overset{*}{=} \ro^{V_I}_{V_I(p_{b, \bfc'(b)})(1_{\bfc'(b)})}
= \ro^{V_I}_{1_{b}} = \iota \ro^{V_{\{b\}}}_{1_{b}}
= \iota \ep^{\{b\}}_0,
\end{align*}
where we applied Lemma \ref{lem:comp-ro-P} to have the equality $(\overset{*}{=})$.
Finally, \eqref{eq:formula for injective modules} follows by applying Lemma~\ref{lem:dim-Hom-coker} and Theorem~\ref{thm:dMV}.
\end{proof}

\subsubsection{The case where \texorpdfstring{$V_I$}{VI} is non-injective}

Throughout this subsection, we assume that $V_I$ is non-injective, and let the sequence
$$
   0 \to V_I \to E \to \ta\inv V_I \to 0
$$
be an almost split sequence starting from $V_I$. We identify $\k[\bfP\op]$ with $A\op = \k[\bfP]\op$ in an obvious way.
To apply Theorem \ref{thm:dMV} and Lemma \ref{lem:dim-Hom-coker},
we need to compute projective presentations of $\ta\inv V_I$ and $E$.
We first do it for $\ta\inv V_I$.

We denote by $(\blank)^t$ the contravariant functors
\[
\begin{aligned}
\Hom_{A}(\blank, A(\cdot, ?)) &\colon \mod A \to \mod A\op,\\
&M \mapsto \Hom_{A}({}_{?}M, A(\cdot, ?)), \text{ and}\\
\Hom_{A\op}(\blank, A\op(\cdot, ?)) &\colon \mod A\op \to \mod A,\\
&M \mapsto \Hom_{A\op}(M_{?}, A(?, \cdot)),
\end{aligned}
\]
which are dualities between
$\prj A$ and $\prj A\op$, where $\prj B$ denotes the full subcategory
of $\mod B$ consisting of projective modules for any finite $\k$-category $B$.
We use the notation $P'_x$ provided in Notation~\ref{ntn:Yoneda}.
By the Yoneda lemma, we have
\begin{equation}
\label{eq:identifn-prime-t}
P^t_x = \Hom_{A}(A(x, ?), A(\cdot, ?)) \iso A(\cdot, x) = A\op(x, \cdot) = P'_x
\end{equation}
for all $x \in \bfP$.
By this natural isomorphism, we usually identify $P'_x$ with $P^t_x$, and $\sfP'_{x,y}$ with $(\sfP_{y,x})^t$ for all $x, y \in \bfP$. For this reason, we write $P^t$ instead of $P'$ in the sequel if there is no confusion.

Remember that $\ta\inv M = \Tr D M$ for all $M \in \mod A$,
where for each $N \in \mod A\op$,
the {\em transpose} $\Tr N$ of $N$ is defined as the cokernel of some $f^t$
with $P_1 \ya{f} P_0 \to N \to 0$ a minimal projective presentation of $N$.
By applying Proposition~\ref{prp:proj.presentation for V_I}, we first obtain a projective presentation
of $DV_{I}$ as follows.
For this sake, we note that there exists an isomorphism $DV_I \iso V_{I\op}$
in $\mod A\op$, and by Lemma~\ref{lem:relation between interval and up-set (down-set)},
$I = W \setminus W'$, where $W\coloneqq \dset I$ and $W'\coloneqq \ddset I = \dset I \setminus I$
are two down-sets of $\bfP$.
By the duality,
$W\op = (\dset_{\bfP}\, I)\op = \uset_{\bfP\op}\, I\op$
and $(W')\op = \uset_{\bfP\op}\, I\op \setminus I\op$.
Hence $DV_I \iso V_{I\op} \iso V_{W\op}/V_{(W')\op}$, where
$$
\begin{aligned}
\src(W\op) &= \snk(W) = \snk(I),\quad \src_1(W\op) = \snk_1(W) = \snk_1(I),\ \text{and}\\
\src((W')\op) &= \snk(W') = \snk(\ddset I).
\end{aligned}
$$

\begin{ntn}
\label{ntn:choice function 3}
Let $I$ an interval of $\bfP$.

(1) Note that $\snk(I) = \snk(\dset I)$, and hence also $\snk_1(I) = \snk_1(\dset I)$.

(2) Note that for each $b'\in \snk(\ddset I)$, we have $\snk(I)\cap \uset b'\neq \emptyset$
because $b' \in \dset I$.
Fixing one element $b\in \snk(I)\cap \uset b'$ for each $b'\in \snk(\ddset I)$
yields a choice map $\bfd\colon \snk(\ddset I)\to \snk(I) = \snk(\dset I)$
that sends $b'$ to $b$.
\end{ntn}

\begin{prp}
\label{prp:mpp-ta-inv-VI-general}
Let $I$ be an interval of $\bfP$.
Then the interval module $D V_I$ has the following (not necessarily minimal) projective presentation
\footnote{We changed the order of direct summands as in
$P'_{\snk(\ddset I)} \ds P'_{\snk_1(I)}$ because we wanted to put matrices dependent
on choice maps closer to each other in the final formula.}
in $\mod A\op$:
\begin{equation}
\label{eq:proj-pre-DV_I-nsrccase-lem w/o conditions}
P'_{\snk(\ddset I)} \ds P'_{\snk_1(I)} \ya{\ps_1}
P'_{\snk(I)} \ya{\ps_0} DV_I \to 0.
\end{equation}
Here $\ps_0 \coloneqq \bmat{\la^{V_I}_{1_b}}_{b \in \snk(I)}$ and
$\ps_1\coloneqq \ps_1(\bfd)\coloneqq \bmat{\ps_{11}, \ps^{\dset I}_{1}}$,
where $\ps_{11}\coloneqq \ps_{11}(\bfd)$ is given by
$\ps_{11}\coloneqq \bmat{\de_{b,\bfd(b')}\sfP'_{b',\bfd(b')}}_{(b,b') \in \snk(I) \times \snk(\ddset I)}$, and
\begin{align}
\label{eq:matrix form of ps^dset_1 w/o conditions, dual version}
\psi^{\dset I}_{1}\coloneqq
\bmat{\vec{\sfP}_{b,\bfb_d}}_{(b, \bfb_d) \in \snk(I) \times \snk_1(I)},
\end{align}
where the entry is given by
\begin{align}\label{eq:entries of ps^dset_1 w/o conditions, dual version}
\vec{\sfP}_{b,\bfb_d}\coloneqq
\begin{cases}
\sfP'_{d,b} & (b=\udl{\bfb}),\\
-\sfP'_{d,b} & (b=\ovl{\bfb}),\\
\mathbf{0} & (b \not\in \bfb),
\end{cases}
\end{align}
for all $\bfb_d \in \snk_1(I)$ and $b \in \snk(I)$. \qed
\end{prp}

The canonical isomorphism \eqref{eq:identifn-prime-t} allows us
to make the identifications:
$P'_{\snk(\ddset I)} \ds P'_{\snk_1(I)} = P^t_{\snk(\ddset I)} \ds P^t_{\snk_1(I)}
$ and $P'_{\snk(I)} = P^t_{\snk(I)}$.
Note here that $\psi_0$ is a projective cover of $DV_{I}$ in
\eqref{eq:proj-pre-DV_I-nsrccase-lem w/o conditions}
because it induces an isomorphism $\top P^t_{\snk(I)} \iso \top DV_{I}$,
but $\psi_1 \colon P^t_{\snk(\ddset I)} \ds P^t_{\snk_1(I)}\to \Im \psi_1$ is not always a projective cover. In any case, there exists a decomposition
\begin{equation}
\label{eq:decomp-P'snk1-I}
P^t_{\snk(\ddset I)} \ds P^t_{\snk_1(I)} = P^t_1 \ds P^t_2
\end{equation}
of the domain of $\ps_1$ such that the restriction
$\Psi \colon P^t_1 \to \Im \psi_1$  of $\ps_1$ is a projective cover.
With respect to the new decomposition $P^t_1 \ds P^t_2$ of the domain of $\ps_1$,
$\ps_1$ has the matrix expression $\psi_1 = \bmat{\Psi, \mathbf{0}}$.

\begin{prp}
\label{prp:proj-pre of tau-VI-ds-P2}
In the setting above, we can give a projective presentation of
$\ta\inv V_{I} \ds P_2$ as follows:
\begin{equation}
\label{eq:proj-resol-ta-invV_I+P-nsrccase-lem w/o conditions}
P_{\snk(I)} \ya{{\pi_{1}}=\sbmat{\Psi^t\\ \mathbf{0}}} P_1 \ds P_2 \ya{(\cok{\Psi^t})\ds \id_{P_2}} \ta\inv V_{I} \ds P_2 \to 0.
\end{equation}
Here by changing the decomposition of the middle term to the right hand side of the equality
$P_1 \ds P_2 = P_{\snk(\ddset I)} \ds P_{\snk_1(I)}$, we have
\begin{equation*}
\pi_1 = \pi_1(\bfd) = \bmat{\pi_{11} \\ \pi^{\dset I}_1}\coloneqq  \bmat{\ps^t_{11} \\ (\ps^{\dset I}_1)^t}\colon
P_{\snk(I)} \to P_{\snk(\ddset I)} \ds P_{\snk_1(I)}
\end{equation*}
that is induced from \eqref{eq:proj-pre-DV_I-nsrccase-lem w/o conditions},
where $\pi_{11}\coloneqq \pi_{11}(\bfd)$ is given by $\pi_{11}\coloneqq\ps^t_{11} = \bmat{\de_{b,\bfd(b')}\sfP_{\bfd(b'),b'}}_{(b',b) \in \snk(\ddset I)\times \snk(I)}$,
and by \eqref{eq:matrix form of ps^dset_1 w/o conditions, dual version}
and \eqref{eq:entries of ps^dset_1 w/o conditions, dual version},
the precise form of $\pi_1^{\dset I}\coloneqq (\ps^{\dset I}_1)^t$ is given as follows:
\begin{align*}
\pi^{\dset I}_1 =
\bmat{\vec{\sfP}_{b,\bfb_d}}_{(b, \bfb_d) \in \snk(I) \times \snk_1(I)}^t
= \bmat{\vec{\sfP}_{b,\bfb_d}^t}_{(\bfb_d, b) \in \snk_1(I) \times \snk(I)} \eqqcolon \bmat{\hat{\sfP}_{b,\bfb_d}}_{(\bfb_d, b) \in \snk_1(I) \times \snk(I)},
\end{align*}
where the entry is given by
\begin{align}\label{eq:entries of psi_1 w/o conditions, dual version}
\hat{\sfP}_{b,\bfb_d}\coloneqq
    \begin{cases}
    \sfP_{b,d} & (b=\udl{\bfb}),\\
    -\sfP_{b,d} & (b=\ovl{\bfb}),\\
    \mathbf{0} & (b \not\in \bfb),
    \end{cases}
\end{align}
for all $b \in \snk(I)$ and $\bfb_d \in \snk_1(I)$.
\end{prp}

\begin{proof}
By \eqref{eq:proj-pre-DV_I-nsrccase-lem w/o conditions}
and \eqref{eq:decomp-P'snk1-I}, $DV_{I}$ has a minimal projective presentation
\begin{equation}
\label{eq:min-proj-resol-DV_I-nsrccase-lem w/o conditions}
P^t_1 \ya{\Psi} P^t_{\snk(I)} \ya{\psi_{0}} DV_{I} \to 0.
\end{equation}
Hence by applying $(\blank)^t\coloneqq \Hom_{A\op}(\blank, A\op)$
to $\Psi$ in \eqref{eq:min-proj-resol-DV_I-nsrccase-lem w/o conditions},
we have a minimal projective presentation
\begin{equation}
\label{eq:min-proj-resol-ta-invV_I-nsrccase-lem w/o conditions}
P_{\snk(I)} \ya{\Psi^t} P_1 \ya{\cok{\Psi^t}} \ta\inv V_{I} \to 0
\end{equation}
of $\ta\inv V_{I} = \Tr D V_{I}$ in $\mod A$.
Hence the assertion follows.
\end{proof}

Note that in the projective presentation
\eqref{eq:proj-resol-ta-invV_I+P-nsrccase-lem w/o conditions} of $\ta\inv V_{I} \ds P_2$,
both of the projective terms and the form of the morphism $\ps_1^t$
between them is explicitly given, whereas those in the projective presentation
\eqref{eq:min-proj-resol-ta-invV_I-nsrccase-lem w/o conditions},
the forms of $P_1$ and $\Psi^t$ are not clear. Therefore, we will use the former presentation \eqref{eq:proj-resol-ta-invV_I+P-nsrccase-lem w/o conditions} in our computation.
Fortunately, as seen in \eqref{eq:formula-dRV-2-2 w/o conditions-general},
the unnecessary $P_2$ does not disturb it
because we can give an explicit form of projective presentation of $E \ds P_2$
as follows.

\begin{prp}
\label{prp:pp-E-ds-P_2-general}
Let $I$ be an interval of $\bfP$ with $\src(I) = \{a_1, \ldots, a_n\}$
and $\snk(I) = \{b_1, \ldots, b_m\} \ (m\ge 2)$.
Choose any choice maps $\bfc \colon \src(\uuset I) \to \src(I)$
and $\bfd \colon \snk(\ddset I) \to \snk(I)$, and
set $\ep_1\coloneqq \ep_1(\bfc),\ \pi_1\coloneqq \pi_1(\bfd)$.
Choose also any $i \in [n]$.
Then there exists some $j \in [m]$ such that $b_j \ge a_i$.
Let $\la\coloneqq \la(b_j, a_i)\coloneqq [\la_{b,a}]_{(b,a)\in \snk(I)\times \src(I)}$ to be the matrix with
$\la_{b,a} = \sfP_{b_j,a_i}$ if $(b,a) = (b_j, a_i)$ and $\la_{b,a} = 0$ otherwise. Then the following is a projective presentation of $E\ds P_2$:
\begin{equation}
\label{eq:pp-E+P2}
(P_{\src_1(I)} \ds P_{\src(\uuset I)}) \ds P_{\snk(I)} \ya{\mu_E}
P_{\src(I)}\ds (P_{\snk(\ddset I)} \ds P_{\snk_1(I)})
\ya{\ep_E} E\ds P_2  \to 0.
\end{equation}
Here $\mu_E$ is given by
\[
\mu_E\coloneqq \left[
\begin{array}{c|c}
\ep_1  &\la\\
\hline
\mathbf{0}& \pi_1
\end{array}
\right],
\]
where $\ep_1$ and $\pi_1$ are given in {\rm Propositions \ref{prp:proj.presentation for V_I}}
and {\rm \ref{prp:proj-pre of tau-VI-ds-P2}}, respectively.
\end{prp}

\begin{proof}
Without loss of generality, we may assume that $a_i = a_1$ and $b_j = b_1$.
Then
$$
\la = \bmat{\sfP_{b_1,a_1} & \mathbf{0}\\\mathbf{0}&\mathbf{0}}.
$$
By \cite[Section 3.6]{gabriel2006auslander}, an almost split sequence
\eqref{eq:almost-split-sequence} can be obtained as a pushout of the sequence
\eqref{eq:min-proj-resol-ta-invV_I-nsrccase-lem w/o conditions}
along a morphism $\et \colon P_{\snk(I)} \to V_{I}$ as follows:
\begin{equation}
\label{eq:pushout-general w/o conditions}
\begin{tikzcd}
\Nname{P_1}P_{\snk(I)} & \Nname{P_0}P_1 & \Nname{ta1}\ta\inv V_{I} & \Nname{01}0\\
\Nname{VI}V_{I} & \Nname{E}E & \Nname{ta2}\ta\inv V_{I} & \Nname{02}0
\Ar{P_1}{P_0}{"\Psi^t"}
\Ar{P_0}{ta1}{"\cok{\Psi^t}"}
\Ar{ta1}{01}{}
\Ar{VI}{E}{"\mu"}
\Ar{E}{ta2}{"\ep"}
\Ar{ta2}{02}{}
\Ar{P_1}{VI}{"\et"}
\Ar{P_0}{E}{"\th"}
\Ar{ta1}{ta2}{equal}
\end{tikzcd}.
\end{equation}
Here, $\et$ is the composite of morphisms
\[
P_{\snk(I)} \ya{\text{can.}} \top P_{\snk(I)} \isoto \soc \nu P_{\snk(I)}
\isoto \soc V_{I} \ya{\al} S \hookrightarrow \soc V_{I} \hookrightarrow V_{I},
\]
where $\nu$ is the Nakayama functor $\nu\coloneqq D\circ \Hom_{A}(\blank, A)$, $S$ is any simple $A$-$\End_{A}(V_{I})$-subbimodule of
$\soc V_{I}$, and $\al$ is a retraction.

Here we claim that
any simple $A$-submodule $S$ of $\soc V_{I}$ is automatically
a simple $A$-$\End_{A}(V_{I})$-subbimodule of $\soc V_{I}$. Indeed, this follows from the fact that
$\soc V_{I} = \Ds_{i \in [m]}V_{\{b_i\}}$, where $V_{\{b_i\}}$ are mutually non-isomorphic simple $A$-modules.
More precisely, it is enough to show that $f(S) \subseteq S$ for any $f \in \End_{A}(V_{I})\op$
because if this is shown, then $S$ turns out to be a right $\End_{A}(V_{I})$-submodule
and a simple $A$-$\End_{A}(V_{I})$-subbimodule of $\soc V_{I}$. By the fact above, $S \iso V_{\{b_i\}}$ for a unique $i \in [m]$, and hence
$\pr_j(S) = 0$ for all $j \in [m]\setminus \{i\}$,
where $\pr_j \colon \soc V_{I} \to V_{\{b_j\}}$ is the canonical projection.
Thus $S \subseteq V_{\{b_i\}}$, which shows that $S = V_{\{b_i\}}$ because the both hand sides are simple.
Now if $f = 0$, then $f(S) = 0 \subseteq S$; otherwise $f(S) \iso S$,
and then we have $f(S) = V_{\{b_i\}} = S$ by applying the argument above
to the simple $A$-submodule $f(S)$ of $\soc V_I$. This proves our claim.

Therefore, we may take $S\coloneqq V_{\{b_1\}}$, and
\[
\et\coloneqq \bmat{\ro^{V_I}_{1_{b_1}}, \mathbf{0}, \dots, \mathbf{0}} \colon P_{\snk({I})} = P_{b_1} \ds P_{b_2} \ds \cdots \ds P_{b_m} \to V_{I}.
\]
By assumption, $a_1 \le b_1$ in ${I}$.
Hence we have a commutative diagram
\[
\begin{tikzcd}
& P_{\src({I})}= P_{a_1} \ds \cdots \ds P_{a_n}\\
P_{\snk({I})} & V_{I}
\Ar{1-2}{2-2}{"{\ep_0 = \sbmat{\ro^{V_I}_{1_{a_{1}}},\dots, \ro^{V_I}_{1_{a_{n}}}}}"}
\Ar{2-1}{2-2}{"\et" '}
\Ar{2-1}{1-2}{"\et'\coloneqq{\Pzero}"}
\end{tikzcd}.
\]
Indeed, by Lemma \ref{lem:comp-ro-P}, we have
\[
\ep_0 \et'
= \bmat{\ro^{V_I}_{1_{a_1}}\sfP_{b_1,a_1}, \mathbf{0}, \dots, \mathbf{0}}
= \bmat{\ro^{V_I}_{V_{I}(p_{b_1,a_1})(1_{a_1})}, \mathbf{0}, \dots, \mathbf{0}}
= \bmat{\ro^{V_I}_{1_{b_1}}, \mathbf{0}, \dots, \mathbf{0}}
= \et.
\]
The pushout diagram \eqref{eq:pushout-general w/o conditions} yields the exact sequence
\begin{equation*}
    P_{\snk({I})} \ya{\sbmat{\et\\ \Psi^t}} V_{I} \ds P_1 \ya{\sbmat{\mu, -\th}} E \to 0.
\end{equation*}
Since $P_1 \ds P_2 = P_{\snk(\ddset I)} \ds P_{\snk_1(I)}$ and $\pi_1 = \sbmat{\Ps^t\\\bfzero}$,
this yields the exact sequence
\begin{equation*}
P_{\snk({I})} \ya{\sbmat{\et\\ \pi_{1}}} V_{I} \ds (P_{\snk(\ddset I)} \ds P_{\snk_1(I)}) \ya{\pi} E \ds P_2 \to 0,
\end{equation*}
where $\pi\coloneqq \bmat{\mu, -\th} \ds \bfzero$.
This is extended to the following commutative diagram with the bottom row exact:
\[
\begin{tikzcd}[row sep=35pt,column sep=20pt, ampersand replacement=\&]
(P_{\src_1(I)} \ds P_{\src(\uuset I)}) \ds P_{\snk({I})} \&[20pt]
P_{\src({I})}\ds (P_{\snk(\ddset I)} \ds P_{\snk_1(I)}) \& E \ds P_2 \& 0\\
(P_{\src_1(I)} \ds P_{\src(\uuset I)}) \ds P_{\snk({I})} \&
V_{I}\ds (P_{\snk(\ddset I)} \ds P_{\snk_1(I)}) \& E \ds P_2 \& 0
\Ar{1-1}{1-2}{"\mu_E"}
\Ar{1-2}{1-3}{"\ep_E"}
\Ar{1-3}{1-4}{}
\Ar{2-1}{2-2}{"\sbmat{\mathbf{0} & \et\\ \mathbf{0} & \pi_1}" '}
\Ar{2-2}{2-3}{"\pi" '}
\Ar{2-3}{2-4}{}
\Ar{1-1}{2-1}{equal}
 \Ar{1-2}{2-2}{"\sbmat{\ep_0 & \mathbf{0}\\ \mathbf{0} & \id}"}
 \Ar{1-3}{2-3}{equal}
\end{tikzcd},
\]
where we set $\mu_E\coloneqq \sbmat{\ep_1 & \et'\\\mathbf{0} & \pi_1}$ and
$\ep_E\coloneqq \pi \circ \sbmat{\ep_0 & \mathbf{0}\\ \mathbf{0} & \id}$,
which is an epimorphism as the composite of epimorphisms.
\par
It remains to show that $\ep_E$ is a cokernel morphism of $\mu_E$.
By the commutativity of the diagram and the exactness of
the bottom row, we see that $\ep_E \mu_E = 0$.
Let $\bmat{f,g}\colon P_{\src({I})}\ds (P_{\snk(\ddset I)} \ds P_{\snk_1(I)}) \to X$ be a morphism with $\bmat{f,g}\cdot \mu_E = 0$. Then $f \ep_1 = 0$.
Since $\ep_0$ is a cokernel morphism of $\ep_1$,
there exists some $f' \colon V_{I} \to X$ such that
$f = f' \ep_0$.
Then we have $\bmat{f,g} = \bmat{f',g}\cdot \sbmat{\ep_0 &\mathbf{0}\\ \mathbf{0}&\id}$.
Now $\bmat{f',g}\cdot \sbmat{\mathbf{0}& \et\\ \mathbf{0}& \pi_1} = \bmat{f',g} \sbmat{\ep_0 &\mathbf{0}\\ \mathbf{0}&\id} \mu_E = \bmat{f,g}\cdot \mu_E = 0$.
Hence $\bmat{f',g}$ factors through $\pi$, that is,
$\bmat{f',g} = h \pi$ for some $h \colon E \ds P_2 \to X$.
Therefore, we have $\bmat{f,g} = h \pi \sbmat{\ep_0 &\mathbf{0}\\ \mathbf{0}&\id} = h\, \ep_E$.
The uniqueness of $h$ follows from the fact that $\ep_E$ is an
epimorphism.
As a consequence, $\ep_E$ is a cokernel morphism of $\mu_E$.
\end{proof}

We are now in a position to state the formula of $d_M(V_I)$ in this case.

\begin{thm}
\label{thm:gen-formula-general}
Let $M \in \mod A$ and $I$ an interval of $\bfP$ with $\src(I) = \{a_1, \ldots, a_n\}$
and $\snk(I) = \{b_1, \ldots, b_m\}$.
Choose any choice maps $\bfc \colon \src(\uuset I) \to \src(I)$
and $\bfd \colon \snk(\ddset I) \to \snk(I)$, and
set $\ep_1\coloneqq \ep_1(\bfc),\ \pi_1\coloneqq \pi_1(\bfd)$
as in~{\rm \cref{prp:proj.presentation for V_I,prp:proj-pre of tau-VI-ds-P2}}.
Choose also any $(j, i) \in [m] \times [n]$ such that $b_j \ge a_i$, and
set $\la\coloneqq \la(b_j, a_i)$ as in~{\rm \cref{prp:pp-E-ds-P_2-general}}. Assume that $V_I$ is non-injective $($i.e., $m \ge 2)$.
Then
\begin{equation}
\label{eq:formula-d_M(V_I)-general-non-inj}
d_M(V_I) = \rank
\left[
\begin{array}{c|c}
M(\ep_1) & \bfzero\\
\hline
M(\la) & M(\pi_1)
\end{array}
\right]
- \rank M(\ep_1) - \rank M(\pi_1).
\end{equation}
\end{thm}

\begin{proof}
Because $V_{I}$ is not injective, the value of $d_{M}(V_{I})$ can be computed from the three terms of the almost split sequence \eqref{eq:almost-split-sequence} by using Theorem~\ref{thm:dMV} as follows:
\begin{equation}
\label{eq:formula-dRV-2-2 w/o conditions-general}
\begin{aligned}
d_{M}(V_{I}) = \dim \Hom_{A}(V_I, M) &- \dim \Hom_{A}(E, M)\\
&+ \dim \Hom_{A}(\ta\inv V_I, M)\\
= \dim \Hom_{A}(V_I, M) &- \dim \Hom_{A}(E \ds P_2, M)\\
&+ \dim \Hom_{A}(\ta\inv V_I \ds P_2, M),
\end{aligned}
\end{equation}
where $P_2$ is a direct summand of $P_{\snk(\ddset I)} \ds P_{\snk_1(I)}$ as in \eqref{eq:decomp-P'snk1-I}.
Hence the assertion follows by Lemma \ref{lem:dim-Hom-coker}, and Propositions \ref{prp:proj.presentation for V_I},
\ref{prp:proj-pre of tau-VI-ds-P2} and \ref{prp:pp-E-ds-P_2-general}.
\end{proof}

\begin{rmk}
    We note here that the formula \eqref{eq:formula-d_M(V_I)-general-non-inj} covers all cases, regardless of whether $V_I$ is injective or not.
\end{rmk}

Summarizing Theorems \ref{thm:dMV-VI-inj-general} and \ref{thm:gen-formula-general}, we obtain the following.

\begin{thm}
\label{thm:gen-formula-unified}
Let $M \in \mod A$ and $I$ an interval of $\bfP$ with $\src(I) = \{a_1, \ldots, a_n\}$
and $\snk(I) = \{b_1, \ldots, b_m\}$.
Choose any choice maps $\bfc \colon \src(\uuset I) \to \src(I)$
and $\bfd \colon \snk(\ddset I) \to \snk(I)$,
and any $(j, i) \in [m] \times [n]$ such that $b_j \ge a_i$.
Set $\la\coloneqq \la(b_j, a_i)$ as in~{\rm \cref{prp:pp-E-ds-P_2-general}}.
Then
\begin{equation}
\label{eq:formula-d_M(V_I)-general}
d_M(V_I) = \rank\left[
\begin{array}{c|c}
M(\ep_1) & \bfzero\\
\hline
M(\la) & M(\pi_1)
\end{array}
\right]
- \rank M(\ep_1) - \rank M(\pi_1).
\end{equation}
Here we collect definitions of
$\ep_1\coloneqq \ep_1(\bfc)$ and $\pi_1\coloneqq \pi_1(\bfd)$
given in~{\rm \cref{prp:proj.presentation for V_I,prp:proj-pre of tau-VI-ds-P2}}:
$\ep_1 \coloneqq \bmat{\ep^{\uset I}_{1}, \ep_{11}}$, where
$$
\begin{aligned}
\ep_{11}&\coloneqq \bmat{\de_{a,\bfc(a')}\sfP_{a',\bfc(a')}}_{(a,a') \in \src(I) \times \src(\uuset I)},
\text{ and}\\
\ep^{\uset I}_{1}&\coloneqq
\bmat{\tilde{\sfP}_{a,\bfa_c}}_{(a,\bfa_c) \in \src(I) \times \src_1(I)}
\end{aligned}
$$
with the entries given by
$$
\begin{aligned}
\tilde{\sfP}_{a,\bfa_c}\coloneqq
\begin{cases}
\sfP_{c,a} & (a = \udl{\bfa}),\\
-\sfP_{c,a} & (a = \ovl{\bfa}),\\
\mathbf{0} & (a \not\in \bfa),
\end{cases}
\end{aligned}
$$
for all $\bfa_c \in \src_1(I)$ and $a \in \src(I)$; and
$
\pi_1 \coloneqq \bmat{\pi_{11} \\ \pi^{\dset I}_1},
$
where
$$
\begin{aligned}
\pi_{11}&\coloneqq \bmat{\de_{b,\bfd(b')}\sfP_{\bfd(b'), b'}}_{(b',b) \in \snk(\ddset I)\times \snk(I)},
\text{ and}\\
\pi^{\dset I}_1
&\coloneqq\bmat{\hat{\sfP}_{b,\bfb_d}}_{(\bfb_d, b) \in \snk_1(I) \times \snk(I)},
\end{aligned}
$$
with the entries given by
$$
\begin{aligned}
\hat{\sfP}_{b,\bfb_d}\coloneqq
    \begin{cases}
    \sfP_{b,d} & (b=\udl{\bfb}),\\
    -\sfP_{b,d} & (b=\ovl{\bfb}),\\
    \mathbf{0} &(b \not\in \bfb),
    \end{cases}
\end{aligned}
$$
for all $b \in \snk(I)$ and $\bfb_d \in \snk_1(I)$.
\end{thm}

\subsection{The case of a 2D-grid}
\label{ssec:2Dgrid-case}
We specialize the general formula \eqref{eq:formula-d_M(V_I)-general}
to the case where $\bfP = G_{m,n}$ for some $m,n \ge 2$ to make the formula
easier to see.
Denote by the maximum element  $(m,n)$ (resp.\ minimum element $(0,0)$) of $\bfP$
by $\om$ (resp.\ $\hat{0}$).

In this subsection, by~\cref{dfn:interval-antichain-exp}, we will write $I = [\src(I), \snk(I)]$ for all $I \in \bbI$.

\begin{ntn}
\label{ntn:notation-in-2D-grid}
Set $\src(\uset I) = \src(I) = \{a_1, \dots, a_k\}$,
$\src(\uuset I) = \{a'_1, \dots, a'_{k'}\}$
and $\snk(\dset I) = \snk(I) = \{b_1, \dots, b_l\}$,
$\snk(\ddset I) = \{b'_1, \dots, b'_{l'}\}$
(see \eqref{eq:dfn-of-uuset-ddset} for the definitions of $\uuset I, \ddset I$).
Throughout this subsection,
we always assume
that
if $(x_i, y_i)$ is the coordinate of $a_i$ in $G_{m,n}$ for each $i \in [k]$,
then $i < j$ implies $y_i < y_j$, and the same for $a'_i,\, b_i$, and $b'_i$. Then we have
$$
\begin{aligned}
\uset I&= [\{a_1, \dots, a_k\}, \om],\ \uuset I = [\{a'_1, \dots, a'_{k'}\}, \om],
\text{ and}\\
\dset I&= [0, \{b_1,\dots, b_l\}],\ \ddset I = [0, \{b'_1,\dots, b'_{l'}\}].
\end{aligned}
$$
In this case, we set
$$
\begin{aligned}
\src^\circ_1(I)&= \{a_{i,i+1}\coloneqq a_i \vee a_{i+1} \mid i = 1,\dots, k-1\},
\text{ and}\\
\snk^\circ_1(I)&= \{b_{i,i+1}\coloneqq b_i \wedge b_{i+1} \mid i = 1 \dots, l-1\}.
\end{aligned}
$$
We note that $\src^\circ_1(I)\subseteq \src_1(I)$ and $\snk^\circ_1(I)\subseteq \snk_1(I)$.
\end{ntn}

With these notations, we have a specialization of
Theorem \ref{thm:gen-formula-unified}.
However, before stating this, we give the minimal projective presentations of
$V_I$, $\ta\inv V_I$, and the middle term $E$ of the almost split sequence
starting from $V_I$ without proofs (for the two latter modules, $V_I$ is assumed to be
non-injective.) because in the general case, we did not give them. We restate \cite[Proposition 39]{asashibaIntervalDecomposability2D2022} under these notations as follows.

\begin{lem}
Let $U$ be an up-set.
Then the interval module $V_{U}$ has the following minimal projective presentation:
$$
P_{\src^\circ_1({U})} \ya{\ep_1^{U}} P_{\src({U})} \ya{\ep_0^{U}} V_{U} \to 0,
$$
where $\ep_0^{U} = \bmat{\ro^{U}_{1_a}}_{a \in \src({U})}$ and
\begin{equation}
\label{eq:morph1_mini_proj_pres}
    \ep_1^{U}\coloneqq \bmat{
\sfP_{a_{12},a_1} \\
-\sfP_{a_{12},a_2}&\sfP_{a_{23},a_2}\\
&-\sfP_{a_{23},a_3}&\ddots \\
&&\ddots &\sfP_{a_{k-1,k},a_{k-1}}\\
&& & -\sfP_{a_{k-1,k},a_{k}}
}(\text{blank entries are zeros}).
\end{equation}
\end{lem}

\begin{rmk}
    We note the reader here that some columns are missing compared with \eqref{eq:matrix form of ep^U_1} because the above projective presentation is minimal in the 2D-grid setting and the missing columns are eliminated by a sequence of fundamental column operations. We refer the reader to \cite[Remark 5.27, Lemma 5.28]{asashiba2024interval} for the detailed explanations.
\end{rmk}

This together with the presentation of the up-set ${\uuset I}$ gives the following (see \cite[Proposition 41]{asashibaIntervalDecomposability2D2022}).

\begin{prp}
\label{prp:mpp-VI}
For each $a' \in \src({\uuset I})$,
set $\bfc(a')\coloneqq a \in \src(I)$ if and only if
$\pr_2(a) = \min\{\pr_2(c) \mid c \le a',\, c \in \src(I)\}$.
Then the interval module $V_I$ has the following minimal projective presentation:
\begin{equation}
\label{eq:mpp-VI}
P_{\src^\circ_1(I)} \ds P_{\src({\uuset I})} \ya{\ep_1\coloneqq \sbmat{\ep_1^{\uset I},\, \ep_{11}}} P_{\src(I)} \ya{\ep_0} V_I \to 0,
\end{equation}
where $\ep_0 = \bmat{\ro^{V_I}_{1_p}}_{p \in \src(I)}$,
$\ep_{11}\coloneqq \bmat{\de_{a,\bfc(a')}\sfP_{a',\bfc(a')}}_{(a,a') \in \src(I) \times \src({\uuset I})}$, and $\ep_1^{\uset I}$ is given in~\eqref{eq:morph1_mini_proj_pres} by letting $U = \uset I$.
\end{prp}

\subsubsection{The case where \texorpdfstring{$V_I$}{VI} is injective}

Assume that $V_I$ is injective.
Then $V_I \iso Q_{b_1} \coloneqq V_{\dset b_1}$, and hence
$I = [\hat{0}, b_1]$, $U = \bfP$, and ${\uuset I} = [\{a'_1, a'_2\}, \om]$,
where delete $a'_1$ or $a'_2$ if it does not exist in $\bfP$. Here and subsequently, we adopt the following void convention: if the set of sources or sinks of an interval is empty, we set that interval to be an empty set and the associated interval module to be $0$.
Moreover, $V_I/\soc V_I \iso V_J$ is also an interval module for the interval
$J = [\hat{0}, \{b_1 - (1,0), b_1 - (0,1)\}]$,
where  delete $b_1 - (1,0)$ or $b_1 - (0,1)$ if it is not in $\bfP$.
Then the pair $(\uset J, \uuset J)$ is a unique pair of up-sets in $\bfP$
such that $J = \uset J \setminus {\uuset J}$.
These are given by
$\uset J = \bfP$, and $\uuset J = [\{a'_1, a'_2, b_1\}, \om]$.
Hence by~\cref{thm:dMV,prp:mpp-VI,lem:dim-Hom-coker}, we have the following.

\begin{thm}
\label{thm:dMV-VI-inj}
Assume that $V_I$ is injective.
Then $V_I$ and $V_I/\soc V_I$ have minimal projective presentations
of the following forms.
$$
\begin{aligned}
&P_{\src({\uuset I})} \ya{\ep'_1} P_{\hat{0}} \to V_I \to 0,\\
&P_{\src({\uuset I})} \ds P_{b_1} \ya{\ep''_1} P_{\hat{0}} \to V_I/\soc V_I \to 0,
\end{aligned}
$$
where $\ep'_1 = [\sfP_{a'_1, \hat{0}},\,  \sfP_{a'_2, \hat{0}}]$, and
$\ep''_1 = [\sfP_{a'_1, \hat{0}},\, \sfP_{a'_2, \hat{0}},\, \sfP_{b_1, \hat{0}}]$.
Therefore, we have
$$
d_M(V_I) = \rank\bmat{M_{a'_1, \hat{0}}\\M_{a'_2, \hat{0}}\\M_{b_1, \hat{0}}}
- \rank\bmat{M_{a'_1, \hat{0}}\\M_{a'_2, \hat{0}}}.
$$
\end{thm}

\begin{prp}
\label{prp:mpp-ta-inv-VI}
For each $b' \in \snk({\ddset I})$,
set $\bfd(b')\coloneqq b \in \snk(I)$ if and only if
$\pr_2(b) = \min\{\pr_2(d) \mid b' \le d \in \snk(I)\}$.
Then the interval module $D V_I$ has the following minimal projective presentation:
$$
P'_{\snk({\ddset I})} \ds P'_{\snk^\circ_1(I)} \ya{\ps_1\coloneqq \sbmat{\ps_{11},\, \ps_1^{\dset I}}} P'_{\snk(I)} \ya{\ps_0} DV_I \to 0,
$$
where $\ps_0 = \bmat{\la^{I}_{1_b}}_{b \in \snk(I)}$,
$\ps_{11}\coloneqq \bmat{\de_{b,\bfd(b')}\sfP'_{\bfd(b'),b'}}_{(b,b') \in \snk(I) \times \snk({\ddset I})}$
and
$$
\ps_1^{\dset I}\coloneqq
\bmat{
\sfP'_{b_1,b_{12}} \\
-\sfP'_{b_2,b_{12}}&\sfP'_{b_2,b_{23}} \\
& -\sfP'_{b_3,b_{23}}& \ddots \\
&&\ddots& \sfP'_{b_{l-1},b_{l-1,l}}\\
&&&-\sfP'_{b_{l},b_{l-1,l}}
}.
$$
Therefore, $\ta\inv V_I = \Tr D V_I$ has a minimal projective presentation
of the following form:
\begin{equation}
\label{min-pp-ta-inv-VI}
P_{\snk(I)} \ya{\pi_1\coloneqq \sbmat{ \pi_{11}\\ \pi_1^{\dset I}}} P_{\snk(\ddset I)} \ds P_{\snk^\circ_1(I)} \ya{\pi_0} \ta\inv V_I \to 0,
\end{equation}
where $\pi_0$ is a projective cover of $\ta\inv V_I$,
$\pi_{11}\coloneqq \bmat{\de_{b,\bfd(b')}\sfP_{\bfd(b'),b'}}_{(b',b) \in \snk(\ddset I) \times \snk(I)}$
and
$$
\pi_1^{\dset I}\coloneqq
\bmat{
\sfP_{b_1,b_{12}} & -\sfP_{b_2,b_{12}}\\
&\sfP_{b_2,b_{23}} & -\sfP_{b_3,b_{23}}\\
&& \ddots & \ddots\\
&&& \sfP_{b_{l-1},b_{l-1,l}} & -\sfP_{b_{l},b_{l-1,l}}
}.
$$
\end{prp}

\begin{prp}
\label{prp:pp-E}
Under~{\rm \cref{ntn:notation-in-2D-grid}}, we may choose a pair $(b_1, a_1)\in \snk(I)\times \src(I)$ with $a_1\leq b_1$. Then the following is a projective presentation of $E$:
\begin{equation}
\label{eq:pp-E}
(P_{\src^\circ_1(I)} \ds P_{\src({\uuset I})} ) \ds P_{\snk(I)} \ya{\mu_E}
P_{\src(I)}\ds (P_{\snk(\ddset I)} \ds P_{\snk^\circ_1(I)})
\ya{\ep_E} E  \to 0.
\end{equation}
Here $\mu_E$ is given by
\[
\mu_E\coloneqq \left[
\begin{array}{c|c}
\ep_1  &\mat{\sfP_{b_1,a_1} & \mathbf{0}\\\mathbf{0}&\mathbf{0}}\\
\hline
\mathbf{0}& \pi_1
\end{array}
\right],
\]
where $\ep_1$ and $\pi_1$ are given in {\rm Propositions \ref{prp:mpp-VI}}
and {\rm \ref{prp:mpp-ta-inv-VI}}, respectively.
\end{prp}

\begin{thm}
\label{thm:gen-formula}
Let $M \in \mod A$ and $I$ an interval of $\bfP$.
Under~{\rm \cref{ntn:notation-in-2D-grid}}, we may choose a pair $(b_1, a_1)\in \snk(I)\times \src(I)$ with $a_1\leq b_1$. Then
\begin{equation}
\label{eq:formula-d_M(V_I)-gen}
d_M(V_I) = \rank\left[
\begin{array}{c|c}
M(\ep_1) & \bfzero\\
\hline
\mat{M_{b_1,a_1} & \mathbf{0}\\\mathbf{0}&\mathbf{0}} & M(\pi_1)
\end{array}
\right]
- \rank M(\ep_1) - \rank M(\pi_1).
\end{equation}
Here $M(\ep_1),\, M(\pi_1)$ have the forms
$$
M(\ep_1) = \bmat{M(\ep^{\uset I}_1)\\M(\ep_{11})} \text{ and }\
M(\pi_1) = \bmat{M(\pi_{11}), M(\pi^{\dset I}_1)},$$
where
$M(\ep_{11})$, $M(\ep^{\uset I}_1)$, $M(\pi_{11})$ and $M(\pi^{\dset I}_1)$ are given by
\begin{align*}
M(\ep_{11})&= \bmat{\de_{a,\bfc(a')}M_{a',\bfc(a')}}_{(a',a) \in \src(\uuset I) \times \src(I)},\\
M(\ep^{\uset I}_1)&= \bmat{
M_{a_{12},a_1} & -M_{a_{12},a_{2}}\\
&M_{a_{23},a_{2}} & -M_{a_{2,3},a_{3}}\\
&& \ddots & \ddots\\
&&& M_{a_{k-1,k},a_{k-1}} & -M_{a_{k-1,k},a_{k}}
},\\
M(\pi_{11})&= \bmat{\de_{b,\bfd(b')}M_{\bfd(b'),b'}}_{(b,b') \in \snk(I) \times \snk(\ddset I)}, \text{ and}\\
M(\pi^{\dset I}_1)&= \bmat{
M_{b_1,b_{12}} \\
-M_{b_2,b_{12}}&M_{b_2,b_{23}} \\
& -M_{b_3,b_{23}}& \ddots \\
&&\ddots& M_{b_{l-1},b_{l-1,l}}\\
&&&-M_{b_{l},b_{l-1,l}}
}.
\end{align*}
\end{thm}

\begin{rmk}
\label{rmk:reduced-form}
We set $M(\ep_1) = \bmat{M(\ep_1)_1, M(\ep_1)_2}$ and $M(\pi_1) = \bmat{M(\pi_1)_1\\M(\pi_1)_2}$,
where $M(\ep_1)_1$ has $\dim M(a_1)$ columns and $M(\pi_1)_1$ has $\dim M(b_1)$ rows.
Then the matrix $R(M,I)$ in the first term of \eqref{eq:formula-d_M(V_I)-gen} has the following form:
$$
R(M,I) = \bmat{
M(\ep_1)_1 & M(\ep_1)_2 &\bfzero\\
M_{b_1,a_1} & \mathbf{0} & M(\pi_1)_1\\
\mathbf{0}&\mathbf{0}& M(\pi_1)_2
}.
$$
We denote by $E_r$ the identity matrix of rank $r$.
By elementary column transformations within the second block column and
elementary row transformations within the first block row,
we can transform $M(\ep_1)_2$ to the normal form $E_{r_1} \ds \bfzero$; and
by elementary column transformations within the third block column and
elementary row transformations within the third block row,
we can transform $M(\pi_1)_2$ to the normal form $E_{r_2} \ds \bfzero$,
where the obtained matrix $R(M,I)_1$ is equivalent to $R(M, I)$, and has the form:
$$
R(M,I)_1 = \left[
\begin{array}{c|cc|cc}
M'_1 & E_{r_1} & \bfzero & \bfzero & \bfzero\\
M_1 & \bfzero & \bfzero & \bfzero & \bfzero\\
\hline
M_2 & \bfzero & \bfzero & M'_3 & M_3\\
\hline
\bfzero & \bfzero & \bfzero& E_{r_2} &\bfzero\\
\bfzero & \bfzero &\bfzero & \bfzero & \bfzero
\end{array}
\right]
\sim
\left[
\begin{array}{c|cc|cc}
\bfzero & E_{r_1} & \bfzero & \bfzero & \bfzero\\
M_1 & \bfzero & \bfzero & \bfzero & \bfzero\\
\hline
M_2 & \bfzero & \bfzero & \bfzero & M_3\\
\hline
\bfzero & \bfzero & \bfzero& E_{r_2} &\bfzero\\
\bfzero & \bfzero &\bfzero & \bfzero & \bfzero
\end{array}
\right].
$$
In the same way, we can transform $M_1$ and $M_3$ to the normal forms:
$$
R(M,I) \sim
\left[
\begin{array}{c|cc|c|c}
\bfzero & E_{r_1} & \bfzero & \bfzero & \bfzero\\
\hline
\mat{\hspace{-10pt}E_{r'_1}&\ \bfzero\\\hspace{-10pt}\bfzero &\ \bfzero} & \bfzero & \bfzero & \bfzero & \bfzero\\
\hline
\mat{M_a&M_b\\M_c & M_d} & \bfzero & \bfzero & \bfzero & \mat{E_{r'_2}&\bfzero\\\bfzero & \bfzero}\\
\hline
\bfzero & \bfzero & \bfzero& E_{r_2} &\bfzero\\
\bfzero & \bfzero &\bfzero & \bfzero & \bfzero
\end{array}
\right]
\sim
M_d \ds E_{r_1} \ds E_{r'_1} \ds E_{r_2} \ds E_{r'_2},
$$
where the last equivalence is obtained by transforming $M_a, M_b, M_c$ to $\bfzero$
by using $E_{r'_1}$ and $E_{r'_2}$.
Then we have
$$
d_M(V_I) = \rank M_d
$$
because $\rank R(M,I) = \rank M_d + r_1 + r'_1 + r_2 + r'_2$,
$\rank M(\ep_1) = r_1 + r'_1$, and $\rank M(\pi_1) = r_2 + r'_2$.
\end{rmk}

\begin{exm}
\label{exm:2Dgrid}
In the following diagrams,
let $\bfP = G_{4,2}$ be given by the quiver on the left,
and $M \in \mod A$ be given by
the diagram on the right:
$$
\begin{tikzcd}
1' & 2' & 3' & 4'\\
1 & 2 & 3 & 4
\Ar{1-1}{1-2}{}
\Ar{1-2}{1-3}{}
\Ar{1-3}{1-4}{}
\Ar{2-1}{2-2}{}
\Ar{2-2}{2-3}{}
\Ar{2-3}{2-4}{}
\Ar{2-1}{1-1}{}
\Ar{2-2}{1-2}{}
\Ar{2-3}{1-3}{}
\Ar{2-4}{1-4}{}
\end{tikzcd}
\quad
\begin{tikzcd}[column sep=20pt, ampersand replacement=\&]
\k^2 \& \k^2 \& \k^2 \& \k\\
\k \& \k^3 \& \k^3 \& \k^3
\Ar{1-1}{1-2}{"\sbmat{1&0\\0&1}"}
\Ar{1-2}{1-3}{"\sbmat{1&0\\0&1}"}
\Ar{1-3}{1-4}{"\bfzero"}
\Ar{2-1}{1-1}{"\bfzero"}
\Ar{2-2}{1-2}{"\sbmat{1&0&0\\0&1&0}"}
\Ar{2-3}{1-3}{"\sbmat{1&0&0\\0&1&0}"}
\Ar{2-4}{1-4}{"\bfzero"}
\Ar{2-1}{2-2}{"{}^t\sbmat{0&0&1}" '}
\Ar{2-2}{2-3}{"\id"}
\Ar{2-3}{2-4}{"\id"}
\end{tikzcd}.
$$
Let $I\coloneqq [\{2,1'\}, \{4,3'\}]$ be an interval of $\bfP$.
Then $\udim V_I\coloneqq \sbmat{1&1&1&0\\0&1&1&1}$, and
$a_1 = 2,\, a_2 = 1',\, a_{12} = 2',\, b_1 = 4,\, b_2 = 3',\, b_{12} = 3,
\, a'_1 = 4',\, b'_1 = 1$.
Therefore,
$$
\begin{aligned}
R(M,I) &=
\bmat{
M_{a_{12},a_1} & -M_{a_{12},a_2} & \bfzero& \bfzero\\
M_{a'_1,a_1} & \bfzero & \bfzero& \bfzero\\
M_{b_1,a_1} & \bfzero & M_{b_1, b'_1} & M_{b_1,b_{12}}\\
\bfzero & \bfzero & \bfzero & -M_{b_2, b_{12}}}
=
\bmat{
M_{2',2} & -M_{2',1'} & \bfzero& \bfzero\\
M_{4',2} & \bfzero & \bfzero& \bfzero\\
M_{4,2} & \bfzero & M_{4, 1} & M_{4,3}\\
\bfzero & \bfzero & \bfzero & -M_{3', 3}}
\\
&=
\left[
\begin{array}{c|c|c|c}
\smat{1&0&0\\0&1&0} & \smat{-1&0\\0&-1} &\bfzero & \bfzero\\
\hline
\bfzero & \bfzero & \bfzero & \bfzero\\
\hline
\smat{{1}&0&0\\0&{1}&0\\0&0&1} & \bfzero  & \smat{0\\0\\1} &\smat{\ 1&\ 0&\ 0\\\ 0&\ 1&\ 0\\\ 0&\ 0&\ 1}\\
\hline
\bfzero & \bfzero &  \bfzero & \smat{-1&0&0\\0&-1&0}
\end{array}
\right]
\sim
\left[
\begin{array}{c|c|c|c}
\smat{0&0&0\\0&0&0} & \smat{1&0\\0&1} &\bfzero & \bfzero\\
\hline
\bfzero & \bfzero & \bfzero & \bfzero\\
\hline
\smat{{1}&0&0\\0&{1}&0\\0&0&0} & \bfzero &\smat{0\\0\\0} &  \smat{0&0&0\\0&0&0\\0&0&1}\\
\hline
\bfzero & \bfzero & \bfzero & \smat{1&0&0\\0&1&0}
\end{array}
\right].
\end{aligned}
$$
Hence $d_M(V_I) = 2$.
Noting that $M \iso V_I^2 \ds V_{[1,4]} \ds V_{[4',4']}$,
we see that this gives a correct value.
Define another $M' \in \mod A$ from $M$ by changing the linear maps $M_{4',3'}$ and $M_{4',4}$ to be $\sbmat{1&0}$ and $\sbmat{1&0&0}$,
respectively,
then
$$
R(M',I) = \left[
\begin{array}{c|c|c|c}
\smat{1&0&0\\0&1&0} & \smat{-1&0\\0&-1} &\bfzero & \bfzero\\
\hline
\smat{1 & 0 & 0} & \bfzero & \bfzero & \bfzero\\
\hline
\smat{1&0&0\\0&{1}&0\\0&0&1} & \bfzero & \smat{0\\0\\1} & \smat{\ 1&\ 0&\ 0\\\ 0&\ 1&\ 0\\\ 0&\ 0&\ 1}\\
\hline
\bfzero & \bfzero & \bfzero & \smat{-1&0&0\\0&-1&0}
\end{array}
\right]
\sim
\left[
\begin{array}{c|c|c|c}
\smat{0&0&0\\0&0&0} & \smat{1&0\\0&1} &\bfzero & \bfzero\\
\hline
\smat{1 & 0 & 0} & \bfzero & \bfzero & \bfzero\\
\hline
\smat{{0}&0&0\\0&{1}&0\\0&0&0} & \bfzero & \smat{0\\0\\0} & \smat{0&0&0\\0&0&0\\0&0&1}\\
\hline
\bfzero & \bfzero & \bfzero & \smat{1&0&0\\0&1&0}
\end{array}
\right].
$$
Hence $d_{M'}(V_I) = 1$, which coincides with the answer obtained from the decomposition
$M' \iso V_I \ds V_{[1,4]} \ds V_{[\{2,1'\},4']}$.
These decompositions can be easily seen by drawing the structure quivers of $M, M'$:
\[\begin{tikzcd}[row sep=0pt, column sep = 15pt]
	{1'_x} & {2'_x} && {3'_x} && {4'} \\
	{1'_y} && {2'_y} && {3'_y} \\
	& {2_x} && {3_x} && {4_x} \\
	&& {2_y} && {3_y} & {4_y} \\
	1 && {2_z} && {3_z} & {4_z}
	\arrow[from=1-1, to=1-2]
	\arrow[from=1-2, to=1-4]
	\arrow[dashed, from=1-4, to=1-6]
	\arrow[from=2-1, to=2-3]
	\arrow[from=2-3, to=2-5]
	\arrow[from=3-2, to=1-2]
	\arrow[from=3-2, to=3-4]
	\arrow[from=3-4, to=1-4]
	\arrow[from=3-4, to=3-6]
	\arrow[dashed, from=3-6, to=1-6]
	\arrow[from=4-3, to=2-3]
	\arrow[from=4-3, to=4-5]
	\arrow[from=4-5, to=2-5]
	\arrow[from=4-5, to=4-6]
	\arrow[from=5-1, to=5-3]
	\arrow[from=5-3, to=5-5]
	\arrow[from=5-5, to=5-6]
\end{tikzcd},\]
where $M$ is given by solid arrows, and $M'$ is given by both solid and broken arrows,
bases of $M(i)$ are denoted by $i$ or $i_a$ \ ($a \in \{x, y, z\}$) for all $i \in \bfP$.
\end{exm}

\subsection{Reducing candidates of the interval direct summands}
\label{sect:Reducing candidates of the interval direct summands}

Given a module $M \in \mod A$, we reduce the number of intervals $I \in \bbI$ to compute
the multiplicity $d_M(V_I)$ by removing some intervals $I$ such that $V_I$ cannot be a
direct summand of $M$, namely $I$ with $d_M(V_I) = 0$, by an easy criterion.

We set the \emph{support} of $M$ to be
$$
\supp M\coloneqq \{x \in \bfP \mid M(x) \ne 0\}.
$$
We denote by $\rad M$ the \emph{radical} of $M$, which is, by definition, the intersection of
all maximal submodules of $M$, and set $\top M\coloneqq M/\rad M$, called the \emph{top} of $M$.
Dually, we set $\soc M$ to be the sum of all simple submodules of $M$,
called the \emph{socle} of $M$.
Recall that an epimorphism $P \to M$ with $P$ projective is a projective cover of
$M$ if and only if it induces an isomorphism $\top P \to \top M$.
Dually, a monomorphism
$M \to Q$ with $Q$ injective
is an injective hull of $M$
if and only if it restricts to an isomorphism
$\soc M \to \soc Q$.
It is easy to see that
if $N$ is a direct summand of $M$, then $\top N$ (resp.\ $\soc N$) is a direct summand of $\top M$ (resp.\ $\soc M$).
Hence
\begin{equation}
\label{eq:Q-dsmd-P}
\supp(\top N) \subseteq \supp(\top M) \quad\text{and}\quad
\supp(\soc N) \subseteq \supp(\soc M).
\end{equation}

Here we introduce the following notation:

\begin{ntn}
\label{ntn:crt0}
Let $M \in \mod A$.
We set \emph{the $0$-th critical set of intervals of} $M$ to be
$$
\crt_0(M)\coloneqq \{I \in \bbI \mid \src(I) \subseteq \supp(\top M),\, \snk(I) \subseteq \supp(\soc M)\}.
$$
\end{ntn}

The following is immediate from \eqref{eq:Q-dsmd-P}.

\begin{lem}
\label{lem:critical}
Let $M \in \mod A$ and $I \in \bbI$.
If $V_I$ is a direct summand of $M$, then
$$
I \in \crt_0(M).
$$
Namely, if $I \not\in \crt_0(M)$, then $d_M(V_I) = 0$.
\end{lem}

\begin{proof}
This follows by
$\src(I) = \supp(\top V_I) \subseteq \supp(\top M)$, and
$\snk(I) = \supp(\soc V_I) \subseteq \supp(\soc M)$.
\end{proof}

Consequently, to determine all interval summands of $M$, it suffices to consider only the intervals in $\crt_0(M)$.
We will further introduce smaller critical sets of intervals of $M$ below.

\begin{ntn}
\label{ntn:mpp-form}
For each $M \in \mod A$, we let
\begin{equation}
\label{eq:prj-cover}
0 \to \Om(M) \ya{\si_M} P_0(M) \ya{\ep_M} M \to 0
\end{equation}
be an exact sequence with $\ep_M$ a projective cover, and
$$
P_1(M) \ya{\partial^M_0} P_0(M) \ya{\ep_M} M \to 0
$$
a minimal projective presentation of $M$ (so that $P_1(M) = P_0(\Om(M))$ and $\partial^M_0 = \si_M \ep_{\Om(M)}$).
More generally, we let $P\down(M)\coloneqq (P_i(M), \partial_i^M \colon P_{i+1}(M) \to P_i(M))_{i \ge 0}$ be a minimal projective resolution of $M$.

Dually, we let
$$
 0 \to M \ya{\mu_M} Q^0(M) \ya{d_M^0} Q^1(M)
 $$
 be a minimal injective copresentation of $M$,
 and $Q\up(M)\coloneqq (Q^i(M), d^i_M \colon Q^i(M) \to Q^{i+1}(M))_{i \ge 0}$ a minimal injective coresolution of $M$.
\end{ntn}

The following gives a relation between $\top M$ (resp.\ $\soc M$) and
the minimal projective presentation (resp.\ minimal injective copresentation) of $M$.

\begin{rmk}
\label{rmk:criterion-by-proj}
Let $M \in \mod A$.
If $P_0(M) = \Ds_{i\in [m]} P_{x_i}^{(r_i)}$ with $x_1, \dots, x_m \in \bfP$
and $r_i \ge 1$ for all $i \in [m]$,
then since $\top M \iso \top P_0(M)$, we have
$$
\supp(\top M) = \{x_1, \dots, x_m\}.
$$
Dually, if
$Q^0(M) = \Ds_{j\in [m']} Q_{x'_j}^{(s_j)}$ with $x'_1, \dots, x'_{m'} \in \bfP$ and $s_j \ge 1$ for all $j \in [m']$,
then
$$
\supp(\soc M) = \{x'_1, \dots, x'_{m'}\}.
$$
\end{rmk}

The following is well-known.

\begin{lem}
Let $X, Y \in \mod A$.
Then
$$
0 \to \Om(X)\ds \Om(Y) \ya{\si_{X} \ds \si_Y}
P_0(X)\ds P_0(Y) \ya{\ep_X \ds \ep_Y} X \ds Y \to 0
$$
is an exact sequence with $\ep_X \ds \ep_Y$ a projective cover of
$X \ds Y$.
\end{lem}

This immediately shows the following:
\begin{prp}
\label{prp:ds-of-min-prj-pres}
Let $X, Y \in \mod A$.
Then
$$
P_1(X) \ds P_1(Y) \ya{\partial^X_0 \ds \partial^Y_0}
P_0(X)\ds P_0(Y) \ya{\ep_X \ds \ep_Y} X \ds Y \to 0
$$
is a minimal projective presentation of $X \ds Y$, and more generally,
$$
P\down(X) \ds P\down(Y)\coloneqq (P_i(X)\ds P_i(Y), \partial_i^X \ds \partial_i^Y)_{i \ge 0}
$$
is a minimal projective resolution of $X \ds Y$.
\end{prp}

The following is immediate from Proposition \ref{prp:ds-of-min-prj-pres} by the uniqueness of a minimal projective presentation of a module up to isomorphism of exact sequences:

\begin{cor}
\label{cor:supp-top}
Let $L, M \in \mod A$, and assume that $L$ is a direct summand of $M$.
Then for each
$i \ge 0$,
the following statements hold:
\begin{enumerate}[label=\rm (\arabic*)]
\item
$P_i(L)$ is a direct summand of $P_i(M)$.
Therefore,
$$
\supp(\top P_i(L)) \subseteq \supp(\top P_i(M)).
$$
\item Dually,
$Q^i(L)$ is a direct summand of $Q^i(M)$.
Therefore,
$$
\supp(\soc Q^i(L)) \subseteq \supp(\soc Q^i(M)).
$$
\end{enumerate}
\end{cor}

Using the fact above, we generalize $\crt_0(M)$ to define the following.

\begin{dfn}
Let $M \in \mod A$, $I \in \bbI$, and $i \ge 0$.
Then we define the $i$-th \emph{critical set of intervals} of $M$ to be
\begin{multline*}
\crt_i(M)\coloneqq \{I \in \bbI \mid \forall j \text{ with }0 \le j \le i,
\supp(\top P_j(V_I)) \subseteq \supp(\top P_j(M)),\\
\supp(\soc Q^j(V_I)) \subseteq \supp(\soc Q^j(M))\},
\end{multline*}
therefore, if $i \ge 1$, then inductively we have
\begin{multline*}
\crt_i(M) =\{I \in \crt_{i-1}(M) \mid \supp(\top P_i(V_I)) \subseteq \supp(\top P_i(M)),\\
\supp(\soc Q^i(V_I)) \subseteq \supp(\soc Q^i(M))\}.
\end{multline*}

For $i = 0$, note that $\crt_0(M)$ above coincides with that defined in Notation \ref{ntn:crt0}
by Proposition \ref{prp:proj.presentation for V_I}.
\end{dfn}

The following is immediate by Corollary \ref{cor:supp-top}.

\begin{lem}
\label{lem:critical-i}
Let $M \in \mod A$, $I \in \bbI$, and $i \ge 0$.
If $V_I$ is a direct summand of $M$, then
$$
I \in \crt_i(M).
$$
\end{lem}

In the case where $\bfP$ is a 2D-grid, Proposition \ref{prp:mpp-VI} gives
$P_i(V_I), Q^i(V_I)$ for $i = 0, 1$.
Hence we have the following by Corollary \ref{cor:supp-top}.

\begin{prp}
\label{prp:crt1-2d-grid}
Assume that $\bfP$ is a 2D-grid.
Let $M \in \mod A$, $I \in \bbI$, and
set
$P_1(M) = \Ds_{i \in [n]} P_{y_i}$,
$Q^1(M) = \Ds_{j \in [n']} Q_{y'_j}$,
with each
$x'_i, y'_j \in \bfP$.
Then
\begin{multline}
\label{eq:crt1}
\crt_1(M) = \{I \in \crt_0(M) \mid
\src^\circ_1(I) \cup \src(\uuset I) \subseteq \{y_j \mid j \in [n]\},\\
\snk^\circ_1(I) \cup \snk(\ddset I) \subseteq \{y'_j \mid j \in [n']\} \}.
\end{multline}
\end{prp}

For a general finite poset $\bfP$ and $I \in \bbI$,
we still do not have an exact form of $P_1(V_I)$, and hence we cannot use $\crt_1(M)$.
To improve this, we next give a refinement of Proposition \ref{prp:proj.presentation for V_I}
as follows.

\begin{prp}
\label{prp:mpp-crt_1}
Let $I \in \bbI$.
Then $V_I$ has a minimal projective presentation of the following form:
$$
P\ds P_{\src(\uuset I)} \to P_{\src(I)} \to V_I \to 0,
$$
where $P$ is a direct summand of $P_{\src_1(I)}$.
\end{prp}

\begin{proof}
We start with a projective presentation \eqref{eq:pp-V_I} of $V_I$ given in Proposition \ref{prp:proj.presentation for V_I}:
\begin{equation}
P_{\src_1(I)}\ds P_{\src(\uuset I)} \ya{\ep_1} P_{\src(I)} \ya{\ep_0} V_I \to 0.
\tag{\ref{eq:pp-V_I}}
\end{equation}
Set $M\coloneqq V_{\uset I},\ M'\coloneqq V_{\uuset I}$, identify $V_I = M/M'$,
and consider the canonical short exact sequence
$$
0 \to M' \ya{\io} M \ya{\pi} V_I \to 0.
$$
Then \eqref{eq:pp-V_I} was obtained by applying Lemma \ref{lem:proj.presentation for factor module} using the commutative diagram \eqref{eq:snake-lem-nsrc} with exact rows and exact columns, which
has the following form under Notation \eqref{eq:prj-cover} for $M$ and $M'$ with $P^M_0 = P_0(M)$:
\begin{equation}
\label{eq:snake-lem-mpr}
\begin{tikzcd}
\Nname{lu}0 & P^M_1 & P^M_1 \ds P_0(M')  & P_0(M') &\Nname{ru} 0\\
\Nname{ld}0 & \Om(M) & P_0(M) & M & \Nname{rd}0\\
& 0 & \Cok \ep_1 & V_I & 0\\
&& 0 & 0
\Ar{1-1}{1-2}{}
\Ar{1-2}{1-3}{"{\sbmat{\id \\ \mathbf{0}}}"}
\Ar{1-3}{1-4}{"{\sbmat{\mathbf{0},\, \id}}"}
\Ar{1-4}{1-5}{}
\Ar{2-1}{2-2}{}
\Ar{2-2}{2-3}{"\si_{M}"}
\Ar{2-3}{2-4}{"\ep_{M}"}
\Ar{2-4}{2-5}{}
\Ar{3-2}{3-3}{}
\Ar{3-3}{3-4}{"\overline{\ep_{M}}"'}
\Ar{3-4}{3-5}{}
\Ar{1-2}{2-2}{"\ga_1"'}
\Ar{1-3}{2-3}{"\ep_1"}
\Ar{1-4}{2-4}{"\iota \ep_{M'}"}
\Ar{1-4}{2-3}{"\eta_{11}" ', dashed}
\Ar{2-2}{3-2}{}
\Ar{2-3}{3-3}{"\cok \ep_1" '}
\Ar{2-4}{3-4}{"\pi"}
\Ar{3-3}{4-3}{}
\Ar{3-4}{4-4}{}
\Ar{1-2}{2-3}{"\ep^M_1"}
\Ar{2-3}{3-4}{"\ep_0"}
\end{tikzcd}
\end{equation}
Here we have $\ep_1 = [\ep^M_1, \et_{11}]$,
$P_0(M) = P_0(V_I) = P_{\src(I)}$, $P^M_1 = P_{\src_1(I)}$,
and $P_0(M') = P_{\src(\uuset I)}$.
In particular, we have
$\supp(\top P_0(M')) = \src(\uuset I)$.
Since
\eqref{eq:pp-V_I} is a projective resolution of $V_I$ with $\ep_0$ a projective cover of $V_I$,
we have $\Im \ep_1 = \Om(V_I)$, and the morphism $\ep_1$ restricts to an epimorphism $P^M_1 \ds P_0(M') \to \Om(V_I)$ that has a projective domain. Therefore, we have
\begin{equation}
\label{eq:minimality-prj-cov}
P^M_1 \ds P_0(M') = P' \ds P_0(\Om(V_I))
\end{equation}
for some $P' \subseteq \Ker \ep_1$.
By the Krull--Schmidt theorem,
it is enough to show that $P_0(M')$ is a direct summand of $P_0(\Om(V_I))$.

Decompose $P'$ by two parts as follows:
\begin{align}
&P' = P'_0 \ds P'_1,\ P'_0 \iso \Ds_{i\in [m]}P_{x_i},\ P'_1 \iso \Ds_{j \in [n]} P_{y_j}, \text{ where}\notag\\
&\{x_i \mid i \in [m]\} \subseteq \supp(\top P_0(M')) = \src(\uuset I),\ \text{and}
\label{eq:xi-in-P0M'}\\
&\{y_j \mid j \in [n]\} \cap \supp(\top P_0(M')) = \emptyset.
\label{eq:yj-out-P0M'}
\end{align}
Then we can show that
\begin{equation}
\label{eq:P'_0-incl}
P'_0 \subseteq P^M_1.
\end{equation}
Indeed,
choose an isomorphism $f \colon \Ds_{i\in [m]}P_{x_i} \to P'_0$,
and for each $i \in [m]$, let $\si_i \colon P_{x_i} \to \Ds_{i\in [m]}P_{x_i}$
be the canonical monomorphism.
Since $P' \subseteq \Ker \ep_1$, we have $\ep_{1}f\si_i = 0$.
Hence
$$
\begin{aligned}
0 &= \ep_M \ep_1 f\si_i = \ep_M [\ep^M_1, \et_{11}] f \si_i
= [\ep_M \si_M \ep_{\Om(M)}, \ep_M \et_{11}] f\si_i\\
&= [0, \io \ep_{M'}]f\si_i
=\io \ep_{M'}[0,\id]f\si_i,
\end{aligned}
$$
where $\ep_{M'}\colon P_0(M')\to M'$ is a projective cover. Here since $\io$ is a monomorphism, we have
$\ep_{M'}[0,\id]f\si_i = 0$, and hence
\begin{equation}
\label{eq:x_i-in-rad-M'}
u_i\coloneqq [0,\id]f\si_i(\id_{x_i}) \in \Im ([0,\id]f\si_i) \subseteq \Ker \ep_{M'} \subseteq \rad P_0(M').
\end{equation}
If $u_i \ne 0$, then \eqref{eq:x_i-in-rad-M'} shows that $x_i \in \supp(\rad P_0(M')) = \{x \in \bfP \mid \exists z\in \src(\uuset I),\ z < x\}$.
Thus $z < x_i$ for some $z\in \src(\uuset I)$.
But since $z \in \src(\uuset I) \subseteq \uuset I$,
we have $x_i \not\in \src(\uuset I)$, a contradiction
to \eqref{eq:xi-in-P0M'}.
Hence $u_i = 0$. Since $u_i$ is a generator of $\Im ([0,\id]f\si_i)$,
we have $[0,\id]f\si_i = 0$.
This holds for all $i \in [m]$, and thus $[0,\id]f = 0$.
Therefore, $P'_0 = \Im f \subseteq P^M_1$, and
\eqref{eq:P'_0-incl} holds, as desired.

By taking the intersection with $P^M_1$
to both hand sides of \eqref{eq:minimality-prj-cov},
the modularity shows that
$$
P^M_1 = P'_0 \ds \big[P^M_1 \cap (P'_1 \ds P_0(\Om(V_I)))\big].
$$
By setting $P''\coloneqq P^M_1 \cap (P'_1 \ds P_0(\Om(V_I)))$,
we have $P^M_1 = P'_0 \ds P''$, and hence
$$
P'_0 \ds P'' \ds P_0(M') = P'_0 \ds P'_1 \ds P_0(\Om(V_I)).
$$
This shows that
$$
P'' \ds P_0(M') \iso P'_1 \ds P_0(\Om(V_I)).
$$
Then again by the Krull--Schmidt theorem, \eqref{eq:yj-out-P0M'} shows that
\begin{equation*}
P_0(M') \text{ is a direct summand of } P_0(\Om(V_I)). \tag*{\qedhere}
\end{equation*}
\end{proof}

For the general finite poset case, \eqref{eq:crt1} is modified as follows, which become more coarse than the 2D-grid case.

\begin{ntn}
Let $M \in \mod A$.
We set \emph{the first rough critical set of intervals of} $M$ to be
$$
\crt'_1(M)\coloneqq \{I \in \crt_0(M) \mid \src(\uuset I) \subseteq \supp(\top P_1(M)),\, \snk(\ddset I) \subseteq \supp(\soc Q^1(M))\}.
$$
\end{ntn}

Then we immediately obtain the following by Proposition \ref{prp:mpp-crt_1}.

\begin{prp}
\label{prp:crt_1}
Let $M \in \mod A$ and $I \in \bbI$.
If $V_I$ is a direct summand of $M$, then
$$
I \in \crt'_1(M).
$$
\end{prp}

Furthermore, Theorem \ref{thm:gen-formula-unified} gives another easy criterion for
an interval module to be a direct summand of a given module as follows.

\begin{prp}
\label{prp:summand-criterion-Mba}
Let $M \in \mod A$ and $I \in \bbI$.
If $V_I$ is a direct summand of $M$,
then $M_{b,a} \ne 0$ for any $(a,b) \in \src(I) \times \snk(I)$ with $a \le b$.
\end{prp}

\begin{proof}
If $M_{b,a} = 0$ for some $(a,b) \in \src(I) \times \snk(I)$ with $a \le b$,
then $d_M(V_I) = 0$ by Theorem \ref{thm:gen-formula-unified}. Or more directly,
for any pair $(a,b) \in \src(I) \times \snk(I)$ with $a \le b$,
we have $(V_I)_{b,a} \ne 0$ as it is the identity map $\k \to \k$.
Hence if $V_I$ is a direct summand of $M$, say $M = L \ds N$
with $L \iso V_I$, then since $L_{b,a} \ne 0$, we have
$M_{b,a} = L_{b,a} \ds N_{b,a} \ne 0$.
\end{proof}

We remark that the statement above also follows from \cite[Theorem 5.23]{asashiba2024interval}
applied for the total compression system.

\begin{ntn}
\label{ntn:zp}
Let $M \in \mod A$.
We introduce a new invariant $\zp(M)$ of $M$,
called the \emph{set of zero pairs} of $M$ as follows:
$$
\zp(M)\coloneqq \{(a,b) \in \supp(\top M) \times \supp(\soc M) \mid a \le b, M_{b,a} = 0\}.
$$

We set \emph{the zp critical set of intervals of} $M$ to be
\begin{align*}
\crt_\zp(M)&\coloneqq \{I \in \bbI \mid
M_{b,a} \ne 0 \text{ for all } (a,b) \in \src(I)\times \snk(I)\text{ with } a \le b\}\\
&\ = \{I \in \bbI \mid
(\src(I)\times \snk(I)) \cap \zp(M) = \emptyset\}.
\end{align*}
For each $i \ge 0$, we also set
$$
\crt_{i,\zp}(M)\coloneqq \crt_i(M) \cap \crt_\zp(M) ,\ \crt'_{1,\zp}(M)\coloneqq \crt'_1(M) \cap \crt_\zp(M).
$$
\end{ntn}

\begin{rmk}
In Notation \ref{ntn:zp},
assume that a persistence module $M$ is given as the $i$-th homology of a filtration $\calF$ for some $i$
(see for instance, Example \ref{exm:Fugacci}).
Then the condition $M_{b,a} = 0$ is easily verified in the level of filtration $\calF$.
This is done by checking any $i$-cycle at $\calF(a)$ vanishes at $\calF(b)$. This verification is also possible by checking the fibered barcodes (see~\cite{lesnickInteractiveVisualization2D2015}, or equivalently, rank invariants) or a projective presentation of $M$.
\end{rmk}

Using this notation, we immediately obtain the following by
Propositions \ref{prp:crt_1} and \ref{prp:summand-criterion-Mba}.

\begin{prp}
\label{prp:criterion-int-sum}
Let $M \in \mod A$ and $I \in \bbI$.
If $V_I$ is a direct summand of $M$, then
$$
I \in \crt'_{1,\zp}(M).
$$
\end{prp}

With the formula~\eqref{eq:formula-d_M(V_I)-general} or \eqref{eq:formula-d_M(V_I)-gen}, we are able to compute the maximal interval-decomposable summand of persistence modules. Let $M$ be a persistence module over $\bfP$. The maximal interval-decomposable summand of $M$ is defined by
\begin{equation}
\label{eq:maximum-interval-decomposable-summand}
M_{\bbI} \coloneqq \bigoplus_{I\in \bbI}V_I^{d_M(V_I)} = \bigoplus_{I\in \crt_{1}(M)}V_I^{d_M(V_I)}.
\end{equation}
Note that the range of intervals $\crt_{1}(M)$ in~\eqref{eq:maximum-interval-decomposable-summand} can be replaced by $\crt'_{1,\zp}(M)$. $M_{\bbI}$ can be considered as another ``interval approximation'', which differs from the interval replacement defined in \cite{ASASHIBA2023100007, asashiba2024interval}, the interval approximation defined in \cite{hiraokaRefinementIntervalApproximations2025}, and the interval resolution defined in \cite{ASASHIBA2023107397}. It is obvious that the maximal interval-decomposable summand is also an invariant of modules, but it is incomplete.

The maximal interval-decomposable summand $M_{\bbI}$ also determines the interval decomposability of $M$ trivially, in the following way.

\begin{lem}
\label{lem:interval-decomposability}
For any module $M\in \mod A$ the following are equivalent:
\begin{enumerate}[label=\rm (\arabic*)]
\item $M$ is interval-decomposable.
\item
$\udim M = \sum_{I\in \bbI}d_M(V_I)\cdot \udim V_I = \sum_{I\in \crt_1(M)}d_M(V_I)\cdot \udim V_I$.
\item
$\dim M = \sum_{I\in \bbI}d_M(V_I)\cdot \dim V_I = \sum_{I\in \crt_1(M)}d_M(V_I)\cdot \dim V_I$.
\end{enumerate}
In the above, $\udim$ denotes the dimension vector.
\end{lem}

We would remark on the difference between our proposed method here and the algorithm provided in \cite{dey2023computing} checking the interval decomposability. The key idea of the algorithm in \cite{dey2023computing} is depending on \cite[Theorem 5.10]{ASASHIBA2023100007}, that is, a persistence module $M$ is interval-decomposable if and only if $M$ is isomorphic to the positive part of its interval placement $\de^\xi(M)_{+}$. Thus, they suggest picking up the interval $I$ appearing in $\de^\xi(M)_{+}$ and then checking whether the interval multiplicity $d_{M}(V_I)$ is zero or not by utilizing the \cite[Algorithm 3]{asashibaIntervalDecomposability2D2022}. However, no explicit formula is provided in \cite{asashibaIntervalDecomposability2D2022} and one has to compute everything involving computing the almost split sequence of $V_I$ by the computer program, causing a high computation cost.
On the contrary, this paper provides a direct way to compute the interval multiplicity $d_{M}(V_I)$. This eliminates the procedure to compute the almost split sequence that is required to perform \cite[Algorithm 3]{asashibaIntervalDecomposability2D2022} because this has already been computed theoretically to give the formula, and then the interval decomposability can be easily verified by checking the dimension equality in Lemma \ref{lem:interval-decomposability}. On the other hand, another advantage of our method is that we can find out the maximal interval-decomposable summand of a given persistence module $M$ (over any finite poset $\bfP$) that is not interval-decomposable, while the algorithm in \cite{dey2023computing} can not.

The source code implementing the computational procedures used in this study is publicly available via a GitHub repository at \url{https://github.com/Enhao-Liu/interval-replacement}.

\section{Essential cover}
\label{section-4}
In the previous section, Theorem \ref{thm:gen-formula-unified} provides a general and explicit formula for computing the interval multiplicity in theory, taking the persistence module as the input. Nevertheless, the persistence module is usually latent in practical analysis and hard to obtain in most situations. Thus, how to compute some algebraically defined invariants (for example, the interval rank invariant and interval multiplicity) directly from the given filtration over $\bfP$, without computing the persistent homology in advance, becomes a critical problem to be solved from the TDA perspective. This is also the key step to bringing our theory to the ground of applications. For this reason, we will introduce a potential technique in this section to achieve the purpose.

We first introduce the following notion to consider matrices with entries morphisms in a linear category in a natural way.

\begin{dfn}
\label{dfn:formal-add-hull}
(1) For each linear category $B$, a linear category $\Ds B$, called the
\emph{formal additive hull} of $B$, is defined as follows:

{\bf Objects.} The set of objects is given by
$$
(\Ds B)_0\coloneqq \{(x_i)_{i \in [l]} = (x_1,\dots, x_l) \mid x_1, \dots, x_l \in B_0,
\, l \ge 0\}.
$$
Note that if $l = 0$ above, then $[l] = \emptyset$, and $(x_i)_{i \in [l]}$
is an empty sequence $()$.
For each $x = (x_i)_{i \in [l]} \in (\Ds B)_0$, we set $|x|\coloneqq l$, and call it
the \emph{size} of $x$.

\medskip
{\bf Morphisms.} For any $x, y \in (\Ds B)_0$ with $x = (x_i)_{i \in [l]},\, y = (y_j)_{j \in [m]}$
the set of morphisms from $x$ to $y$ is defined by setting
$$
(\Ds B)(x,y)\coloneqq \big\{\bmat{\al_{ji}}_{(j,i) \in [m]\times [l]} \mid \al_{ji} \in B(x_i, y_j) \text{ for all }(j,i) \in [m]\times [l]\big\},
$$
where $\bmat{\al_{ji}}_{(j,i) \in [m]\times [l]}$ is a matrix of size $(m,l)$,
which is defined to be the triple $(m, l, (\al_{ji})_{(j,i) \in [m]\times [l]})$ of
integers $l, m \ge 0$ and a family of morphisms $\al_{ji} \in B(x_i, y_j)$.
Note that if $l = 0$, then $x = ()$, and we have
\begin{equation}
\label{eq:emptymat-to}
(\Ds B)((), y) = \{\sfJ_{m,0}\},
\end{equation}
where we set $\sfJ_{m,0}\coloneqq (m, 0, ())$;
if $m = 0$, then $y = ()$, and we have
\begin{equation}
\label{eq:to-emptymat}
(\Ds B)(x, ()) = \{\sfJ_{0,l}\},
\end{equation}
where we set $\sfJ_{0,l}\coloneqq (0, l, ())$.
In particular, we have $(\Ds B)((), ()) = \{\sfJ_{0,0}\}$,
where $\sfJ_{0,0} = (0,0,())$.
The matrices $\sfJ_{m,0}, \sfJ_{0,l}, \sfJ_{0,0}$ are called
the \emph{empty matrices} of size
$(m, 0),\, (0, l),\, (0,0)$, respectively.
We give a structure of a vector space to $(\Ds B)(x, y)$
by the usual addition and scalar multiplication of matrices.
In particular, if $l=0$ or $m=0$,
then $(\Ds B)(x, y)$ becomes a trivial vector space.

\medskip
{\bf Composition.} For any $x, y, z \in (\Ds B)_0$
with $x = (x_i)_{i \in [l]},\, y = (y_j)_{j \in [m]},\, z = (z_k)_{k \in [n]}$,
the composition
$$
(\Ds B)(y,z) \times (\Ds B)(x,y) \to (\Ds B)(x,z),\
(\be, \al) \mapsto \be \cdot \al
$$
is defined by the usual matrix multiplication
$$
\bmat{\be_{kj}}_{(k,j)\in [n]\times [m]} \cdot \bmat{\al_{ji}}_{(j,i)\in [m]\times [l]}
\coloneqq \bmat{\sum_{j\in [m]} \be_{kj}\al_{ji}}_{(k,i)\in [n]\times [l]}
$$
for all $\al = \bmat{\al_{ji}}_{(j,i)\in [m]\times [l]}$
and $\be = \bmat{\be_{kj}}_{(k,j)\in [n]\times [m]}$.
In particular, if $l=0$, then $\be \cdot \sfJ_{m,0} = \sfJ_{n,0}$;
if $m=0$, then $\sfJ_{n,0}\cdot \sfJ_{0,l} = (l,n,(0)_{(k,i)\in [n]\times [l]}) = 0_{n,l}$;
and if $n=0$, then $\sfJ_{0,m}\cdot \al = \sfJ_{0,l}$.
Thus
if morphisms $\be,\, \al$ have size $(k, p),\, (q, l)$
with $k,l,p,q \ge 0$, respectively,
and the composite $\be\cdot \al$ is defined, then $p = q$, and the size of $\be \cdot \al$ is $(k,l)$ as in the case of usual matrix multiplication.

As easily seen, $\Ds B$ is a linear category.
Note that equalities \eqref{eq:emptymat-to} and
\eqref{eq:to-emptymat} show that $()$ is a zero object in $\Ds B$.
Moreover, we have
$$
\begin{aligned}
&(x_i)_{i \in [m]} \iso (x_1) \ds \cdots \ds (x_m),\\
&(x_i)_{i \in [m]} \ds (y_j)_{j \in [n]} \iso (x_1,\dots,x_n,y_1,\dots, y_n), \text{ and}\\
&(x_1) \ds \cdots \ds (x_m) \iso (x_1 \ds \cdots \ds x_m)\text{ if $x_1 \ds \cdots \ds x_m$ exists in $B$}
\end{aligned}
$$
for all $x_1,\dots, x_m, y_1,\dots, y_n \in B_0$.
Thus $\Ds B$ turns out to be an additive category.

We regard $B$ as a full subcategory of $\Ds B$ by the embedding
$(f \colon x \to y) \mapsto (\bmat{f}\colon (x) \to (y))$ for all morphisms $f$ in $B$.
In the sequel, we will frequently consider the case
where $B = \k[S]$ for a finite poset $S$.

Note that if $B$ is additive, then we have an equivalence
$\et_B \colon \Ds B \to B$ that sends
$(x_i)_{i\in [m]}$ to $\Ds_{i\in [m]} x_i$,
and each morphism
\[
\bmat{\al_{ji}}_{(j,i)\in [n]\times [m]}
\colon (x_i)_{i\in [m]} \to (y_j)_{j \in [n]}
\]
in $\Ds B$ to
$\bmat{\al_{ji}}_{(j,i)\in [n]\times [m]}
\colon \Ds_{i\in [m]} x_i \to \Ds_{j \in [n]} y_j$ in $B$.
In particular, it sends $()$ to 0.

(2) Let $F \colon B \to C$ be a linear functor between linear categories.
Then a functor $\Ds F \colon \Ds B \to \Ds C$ is defined as follows:
We set
$(\Ds F)((x_i)_{i\in [m]})\coloneqq (F(x_i))_{i\in [m]}$
for each object $(x_i)_{i\in [m]} \in (\Ds B)_0$,
and for each morphism
$$
\al\coloneqq [\al_{ji}]_{(j,i)\in [n]\times [m]} \colon (x_i)_{i\in [m]} \to (y_j)_{j\in [n]},
$$
we set
$$
(\Ds F)(\al)\coloneqq [F(\al_{ji})]_{(j,i)\in [n]\times [m]} \colon (F(x_i))_{i\in [m]} \to (F(y_j))_{j\in [n]}.
$$
In particular, $(\Ds F)(())\coloneqq ()$, and
$F(\sfJ)\coloneqq \sfJ$  for all $\sfJ \in \{\sfJ_{n,0}, \sfJ_{0,m} \mid m, n \ge 0\}$.
For example, $\sfJ_{0,m} \colon (x_i)_{i \in [m]} \to ()$ is sent to
$\sfJ_{0,m} \colon (F(x_i))_{i \in [m]} \to ()$.
If there is no confusion, we denote $\Ds F$ simply by $F$.

Since $()$ is a zero object in $\Ds B$,
we may write $() = 0$ in $\Ds B$.
\end{dfn}

\begin{exm}
\label{exm:Ds-ze}
Let $\ze \colon Z \to \bfP$ be an order-preserving map between posets.
Then we have a linear functor $\k[\ze] \colon \k[Z] \to \k[\bfP]$,
which yields a linear functor
$\Ds \k[\ze] \colon \Ds \k[Z] \to \Ds \k[\bfP]$.
If $\al\coloneqq [\al_{ji}]_{(j,i)\in [n]\times [m]}$ is a morphism in $\Ds \k[Z]$,
we denote $(\Ds \k[\ze])(\al)$ simply by $\ze(\al) = [\ze(\al_{ji})]_{(j,i)\in [n]\times [m]}$.
\end{exm}

\begin{prp}
\label{prp:univ-formal-add-hull}
Let $B$ be a linear category and $\calC$ an additive linear category.
Then each linear functor $F \colon B \to \calC$ uniquely extends to
a linear functor $\hat{F} \colon \Ds B \to \calC$,
which we denote by the same letter $F$ if there seems to be no confusion.
\end{prp}

\begin{proof}
Define a linear functor $\hat{F} \colon \Ds B \to \calC$
as the composite $\hat{F}\coloneqq \et_\calC \circ (\Ds F)$.
Namely,
for each morphism $\al = \bmat{\al_{ji}}_{(j,i)\in [n]\times [m]}
\colon (x_i)_{i\in [m]} \to (y_j)_{j \in [n]}$ in $\Ds B$, we set
\[
\hat{F}(\al)\coloneqq \bmat{F(\al_{ij})}_{j,i} \colon
\Ds_{i \in [m]} F(x_i) \to \Ds_{j \in [n]} F(y_j).
\]
It is easy to see that this is the unique extension of $F$.
\end{proof}

Since each finitely generated projective module over $\k[\bfP]$
is isomorphic to a finite direct sum of
representable functors $P_x\coloneqq \k[\bfP](x, \blank)$, ($x \in \bfP$),
we have the following by applying the proposition above
to the case where $B =  \k[\bfP] = A$.

\begin{cor}
\label{cor:Yoneda}
The Yoneda embedding $Y^A \colon A\op \to \prj A$,
$x \mapsto P_x\coloneqq A(x, \blank)$
extends to an equivalence $\sfP \colon (\Ds A)\op \to \prj A$,
$(x_i)_{i \in [m]} \mapsto \Ds_{i \in [m]}P_{x_i}$.
Note that $\sfP$ maps each morphism
$p_{y,x} \colon x \to y$ in $\bfP$ to
$\sfP_{y,x} \colon P_y \to P_x$.
Therefore, it maps each morphism
$$
\bmat{p_{y_j,x_i}}_{(j,i)\in [n]\times [m]} \colon (x_i)_{i\in [m]} \to (y_j)_{j \in [n]}
$$
in $\Ds A$ to the morphism
$$
\bmat{\sfP_{y_j,x_i}}_{(i,j)\in [m]\times [n]} = {}^t\!\!\bmat{\sfP_{y_j,x_i}}_{(j,i)\in [n]\times [m]}
\colon \Ds_{j \in [n]} P_{y_j} \to \Ds_{i\in [m]} P_{x_i}
$$
in $\prj A$.

Similarly, the Yoneda embedding $Y_A \colon A \to \prj (A\op)$,
$x \mapsto P'_x\coloneqq A(\blank, x)$ extends to an equivalence
$\sfP' \colon \Ds A \to \prj A\op$,
$(x_i)_{i \in [m]} \mapsto \Ds_{i \in [m]}P'_{x_i}$.
\end{cor}

\begin{dfn}
\label{dfn:matrices-g}
Let $I$ be an interval of $\bfP$.
Choose any choice maps $\bfc \colon \src(\uuset I) \linebreak[2]\to \src(I)$
and $\bfd \colon \snk(\ddset I)\linebreak[2] \to \snk(I)$, and
set $\ep_1\coloneqq \ep_1(\bfc),\ \pi_1\coloneqq \pi_1(\bfd)$
as in Propositions \ref{prp:proj.presentation for V_I} and \ref{prp:proj-pre of tau-VI-ds-P2}.
Choose also any $(b, a) \in \snk(I) \times \src(I)$ such that $b \ge a$, and
set $\la\coloneqq \la(b, a)$ as in Proposition \ref{prp:pp-E-ds-P_2-general}.
Then by Corollary \ref{cor:Yoneda},
there exists a unique triple $(\bfg_1, \bfg_2, \bfg_3)$ of morphisms in $\DS \k[\bfP]$ such that
$$
\sfP(\bfg_1) = {}^t\ep_1,\
\sfP(\bfg_2) = {}^t\pi_1, \text{ and }\
\sfP(\bfg_3) = {}^t\la.
$$
We set
$$
\bfg(\bfc, \bfd, (b,a)) \coloneqq \left[
\begin{array}{c|c}
\bfg_1 & \bfzero\\
\hline
\bfg_3 & \bfg_2
\end{array}
\right].
$$
The following are the explicit forms of $\bfg_1,\, \bfg_2,\, \bfg_3$:

$\bfg_1 \coloneqq \bmat{\sbmat{\tilde{p}_{a,\bfa_c}}_{(\bfa_c, a) \in \src_1(I) \times \src(I)} \\
\sbmat{\de_{a,\bfc(a')}p_{a',\bfc(a')}}_{(a', a) \in  \src(\uuset I) \times \src(I)}}$ with the entries given by
$$
\begin{aligned}
\tilde{p}_{a,\bfa_c}\coloneqq
\begin{cases}
p_{c,a} & (a = \udl{\bfa}),\\
-p_{c,a} & (a = \ovl{\bfa}),\\
\mathbf{0} & (a \not\in \bfa),
\end{cases}
\end{aligned}
$$
for all $\bfa_c \in \src_1(I)$ and $a \in \src(I)$; and
$$
\bfg_2 \coloneqq \bmat{\sbmat{\de_{b,\bfd(b')}p_{\bfd(b'), b'}}_{(b,b') \in \snk(I) \times \snk(\ddset I)},\ \sbmat{\hat{p}_{b,\bfb_d}}_{(\bfb_d, b) \in \snk(I) \times \snk_1(I)}},
$$
with the entries given by
$$
\begin{aligned}
\hat{p}_{b,\bfb_d}\coloneqq
    \begin{cases}
    p_{b,d} & (b=\udl{\bfb}),\\
    -p_{b,d} & (b=\ovl{\bfb}),\\
    \mathbf{0} & (b \not\in \bfb),
    \end{cases}
\end{aligned}
$$
for all $b \in \snk(I)$ and $\bfb_d \in \snk_1(I)$; and $\bfg_3$ is the block matrix with the size $\snk(I)\times \src(I)$, the $(b, a)$-entry of $\bfg_3$, given by $p_{b,a}$, is the only non-zero entry.
\end{dfn}

\begin{ntn}
\label{ntn:W(g)}
Let $B$ be a linear category, $W$ a $B$-module, and $m, n$ positive integers,
and consider a morphism $\bfg = \bmat{g_{ji}}_{(j,i)\in [n]\times [m]}
\colon (x_i)_{i\in [m]} \to (y_j)_{j \in [n]}$
in $\Ds B$.
Then by applying the convention in Proposition \ref{prp:univ-formal-add-hull}
in the case where $\calC = \mod\k$,
we write
\[
W(\bfg)\coloneqq \hat{W}(\bfg) = \bmat{W(g_{ij})}_{j,i} \colon
\Ds_{i \in [m]} W(x_i) \to \Ds_{j \in [n]} W(y_j).
\]
\end{ntn}

By Definition \ref{dfn:matrices-g} and Notation \ref{ntn:W(g)}, Theorem \ref{thm:gen-formula-unified} can be restated as follows.

\begin{thm}
\label{thm:final-form-add-hull}
Let $M \in \mod A$ and $I$ an interval of $\bfP$. Choose any choice maps $\bfc \colon \src(\uuset I) \to \src(I)$
and $\bfd \colon \snk(\ddset I) \to \snk(I)$, and
any $(b, a) \in \snk(I) \times \src(I)$ with $b \ge a$.
Set $\bfg\coloneqq \bfg(\bfc, \bfd, (b, a))$ as in~{\rm \cref{dfn:matrices-g}}.
Then
\begin{equation}
\label{eq:formula-d_M(V_I)-general-mor}
d_M(V_I) = \rank M(\bfg) - \rank M(\bfg_1) - \rank M(\bfg_2).
\end{equation}
\end{thm}

Let $\ze\colon Z\to \bfP$ be an order-preserving map with $Z$ a poset.
The order-preserving map $\ze$ is uniquely extended to a linear functor
$\k[\ze] \colon \k[Z] \to \k[\bfP]$, which
induces a functor $R\colon \mod\k[\bfP]\to \mod\k[Z]$, $M \mapsto M \circ \k[\ze]$.
This is called the \emph{restriction functor} induced by $\ze$.

\begin{dfn}
Let $B$ be a linear category, and $1 \le m,\, n \in \bbZ$.
For each $(j,i) \in [n] \times [m]$, let $\al_{ji} \colon x_{ji} \to y_{ji}$
be a morphism in $B$.
Then the family $\pmat{\al_{ji}}_{(j,i) \in [n]\times [m]}$ is said to satisfy
the \emph{matrix condition} if
$\bmat{\al_{ji}}_{(j,i) \in [n]\times [m]}$ is a morphism in the category $\Ds B$,
namely, if
for any $i, p \in [m]$ and $j, q \in [n]$, we have
$x_{ji} = x_{qp}$ if $i = p$
and $y_{ji} = y_{qp}$ if $j=q$.
If this is the case, then by setting $x_i\coloneqq x_{1,i}$ and $y_j\coloneqq y_{j,1}$,
we have $\al_{ji} \in B(x_i, y_j)$ for all $(j,i) \in [n]\times [m]$, and
$\al\coloneqq \bmat{\al_{ji}}_{(j,i) \in [n] \times [m]} \colon (x_i)_{i\in [m]} \to (y_j)_{j\in [n]}$
becomes a morphism in $\Ds B$.
\end{dfn}

\begin{lem}
\label{lem:ze-covers-g}
Let $\ze \colon Z \to \bfP$ be an order-preserving map with $Z$ a poset, and take a morphism $\al\coloneqq \bmat{\al_{ji}}_{(j,i) \in [n] \times [m]} \colon (x_i)_{i\in [m]} \to (y_j)_{j\in [n]}$
in $\Ds \k[\bfP]$. Then the following are equivalent:
\begin{enumerate}[font=\normalfont,label=(\arabic*)]
\item
There exists a morphism $\al' \colon x' \to y'$ in $\Ds\k[Z]$
such that $\ze(\al') = \al$ (see~{\rm \cref{exm:Ds-ze}} for $\ze(\al')$).
\item
There exist maps $\ze' \colon \{x_i \mid i\in [m]\} \to Z$ and $\ze'' \colon \{y_j \mid j\in [n]\} \to Z$ with $\ze\ze' = \id$ and $\ze\ze'' = \id$ (i.e., these are sections of $\ze$)
such that
for any nonzero entry $\al_{ji} \colon x_i \to y_j$ of $\al$,
there exists some $\al'_{ji} \colon \ze'(x_i) \to \ze''(y_j)$ with $\ze(\al'_{ji}) = \al_{ji}$.
\end{enumerate}
\end{lem}

\begin{proof}
(1) \implies (2). Assume (1).
If the morphism $\al'$
has the form $\al' \colon (x'_{i})_{i \in [m]} \to (y'_j)_{j\in [n]}$,
then the map $x_i \mapsto x'_i$ (resp.\ $y_j \mapsto y'_j$) defines
the desired section $\ze'$ of $\ze$
(resp.\ $\ze''$ of $\ze$).

(2) \implies (1).
Assume (2).  Then for any $(j,i) \in [n] \times [m]$,
there exists a morphism $\al'_{ji} \colon \ze'(x_i) \to \ze''(y_j)$ such that
$\ze(\al'_{ji}) = \al_{ji}$.
Then the family $\pmat{\al'_{ji}}_{(j,i)\in [n]\times [m]}$ satisfies the matrix condition,
and hence $\al'\coloneqq \bmat{\al'_{ji}}_{(j,i)\in [n]\times [m]} \colon (\ze'(x_i))_{i\in [m]} \to (\ze''(y_j))_{j\in [n]}$
is a morphism in $\DS \k[Z]$, and satisfies $\ze(\al') = \al$.
\end{proof}

\begin{dfn}
\label{dfn:map-covers-morphism}
Let $\ze \colon Z \to \bfP$ be an order-preserving map with $Z$ a poset,
and
$\al\colon x \to y$ a morphism in $\Ds \k[\bfP]$.
We say that $\ze$ \emph{covers} $\al$ if one of the conditions in Lemma \ref{lem:ze-covers-g}
is satisfied.
Note that if this is the case, then $\ze$ covers all submatrices of $\al$.
\end{dfn}

\begin{dfn}
\label{dfn:muti-mat}
Let $I$ be an interval of $\bfP$, and
$\bfg\coloneqq \bmat{\bfg_1 & \mathbf{0}\\
\bfg_3 & \bfg_2\\
} \colon X \ds X' \to Y \ds Y'$ a morphism in $\Ds\k[\bfP]$.
Then $\bfg$ is called a \emph{multiplicity matrix} for $I$ if
for any $M \in \mod A$ we have
\begin{equation}
\label{eq:rank-multiplicity-formula}
    d_M(V_I) = \rank \bmat{M(\bfg_1) & \mathbf{0}\\
M(\bfg_3) & M(\bfg_2)\\
}
- \rank \bmat{M(\bfg_1) & \mathbf{0}\\
\bfzero & M(\bfg_2)\\
}.
\end{equation}
\end{dfn}

For example, if $\bfc \colon \src(\uuset I) \to \src(I)$ and
$\bfd \colon \snk(\ddset I) \to \snk(I)$ are choice maps, and
$(b, a) \in \snk(I)\times \src(I)$ is a pair with $b \ge a$,
then $\bfg(\bfc, \bfd, (b,a))$ is a multiplicity matrix for $I$.
See $\tilde{\bfg}$
in Example \ref{exm:int-rk_D_4_cases_2nd} for another type of
multiplicity matrix.

\begin{dfn}
\label{dfn:ess-cov2}
Let $\ze \colon Z \to \bfP$ be an order-preserving map with $Z$ a poset,
and $I$ an interval of $\bfP$.
Then we say that $\ze$ \emph{essentially covers} $I$
(or $\ze$ is an \emph{essential cover} of $I$) if
$\ze$ covers a multiplicity matrix for $I$.
\end{dfn}

\begin{rmk}
\label{rmk:1par-multpar}
In the Definition \ref{dfn:ess-cov2},
we allow the cases where any matrices among $\bfg_1,\, \bfg_2$ are $\bfg_3$ empty matrices. We also note the reader that the formula \eqref{eq:rank-multiplicity-formula} is a natural generalization of the well-known rank formula appearing in the one-parameter persistence case. In more detail, suppose $\bfP\coloneqq\bbA_{n}$ and the interval $I\coloneqq[s,t]$ ($s, t\in [n]$). If we set $\bfg_1\coloneqq\bmat{p_{t+1,s}}$, $\bfg_2\coloneqq\bmat{p_{t,s-1}}$, and $\bfg_3\coloneqq\bmat{p_{t,s}}$, then for any $M\in \mod\k[\bfP]$ the right-hand side of \eqref{eq:rank-multiplicity-formula} yields
\begin{align}
&\quad \rank\left[
\begin{array}{c|c}
M_{t+1,s} & \bfzero\\
\hline
M_{t,s} & M_{t,s-1}
\end{array}
\right]
- \rank M_{t+1,s} - \rank M_{t,s-1}\nonumber\\
& \overset{(a)}{=} \rank\left[
\begin{array}{c|c}
\bfzero & -M_{t+1,s-1}\\
\hline
M_{t,s} & M_{t,s-1}
\end{array}
\right]
- \rank M_{t+1,s} - \rank M_{t,s-1} \nonumber\\
& \overset{(b)}{=} \rank\left[
\begin{array}{c|c}
\bfzero & -M_{t+1,s-1}\\
\hline
M_{t,s} & \bfzero
\end{array}
\right]
- \rank M_{t+1,s} - \rank M_{t,s-1} \nonumber\\
& = \rank M_{t,s} + \rank M_{t+1,s-1} - \rank M_{t+1,s} - \rank M_{t,s-1} \nonumber\\
& = d_M(V_I),\nonumber
\end{align}
where the equality $\overset{(a)}{=}$ follows by the elementary row operation that adds
the product of $-M_{t+1, t}$ and the second row block to the first row block,
and the equality $\overset{(b)}{=}$ follows by the elementary column operation that
adds the product of the first column block and $-M_{s,s-1}$ to the second column. The last equality is known as the formula of the persistent Betti numbers and the multiplicity in persistent homology (see \cite[Chapter VII]{edelsbrunner2010computational}).
\end{rmk}

We cite the following statement from \cite[Lemma~6.8]{asashiba2024interval}.

\begin{lem}
\label{lem:direct-sum-rank}
Let $B$ be a linear category, $W$ a $B$-module, and $\bfg$ a morphism in $\Ds B$. Assume that we have a direct sum decomposition $W \iso W_1 \ds W_2$
of $B$-modules. Then we have an equivalence $W(\bfg) \iso W_1(\bfg) \ds W_2(\bfg)$ of linear maps. In particular, the equality
\[
\rank W(\bfg) = \rank W_1(\bfg) + \rank W_2(\bfg)\]
holds.
\end{lem}

Before giving the main theorem, we need the following notation.

\begin{ntn}\label{def:ds-counting}
Let $M\in \mod \k[\bfP]$. If $M\iso L^{n} \ds N$ with $n \ge 0$ such that $N$ has no direct summand isomorphic to $L$, then we set $\bar{d}_{M}(L)\coloneqq n$. In particular, if $L$ is indecomposable, then $\bar{d}_{M}(L)$ coincides with $d_{M}(L)$.
Moreover, by the Krull--Schmidt theorem, we easily see that
if $L = \Ds_{i\in [m]} L_i$ for some $m \ge 1$ with each $L_i$ indecomposable,
then $\bar{d}_M(L) = \min_{i\in [m]}d_M(L_i)$.
\end{ntn}

We are now in a position to state our main result in this section.

\begin{thm}
\label{thm:ess-cover-equality}
Let $\ze \colon Z \to \bfP$ be an order-preserving map with $Z$ a poset,
and $R \colon \mod \k[\bfP] \to \mod \k[Z]$ the restriction functor induced by $\ze$.
Take an interval $I$ of $\bfP$.
If $\ze$ essentially covers $I$, then for any $M \in \mod A$, we have
$$
d_M(V_I) = \bar{d}_{R(M)}(R(V_I)).
$$
\end{thm}

\begin{proof}
Assume that $\ze$ essentially covers $I$.
Then there exists $\bfg\coloneqq \bmat{\bfg_1 & \mathbf{0}\\
\bfg_3 & \bfg_2\\
}$ in $\Ds\k[\bfP]$
such that
for any $M \in \mod A$ we have a formula
$$
d_M(V_I) = \rank \bmat{M(\bfg_1) & \mathbf{0}\\
M(\bfg_3) & M(\bfg_2)\\
}
- \rank \bmat{M(\bfg_1) & \mathbf{0}\\
\bfzero & M(\bfg_2)\\
},
$$
and $\ze$ covers $\bfg$, say $\ze(\bfg') = \bfg$ for some $\bfg'$ in $\Ds\k[Z]$.
Let $g_{vu}$ (resp.\ $g'_{vu}$) be the $(v,u)$-entry of $\bfg$ (resp.\ $\bfg'$).
Then
\[
M(g_{vu}) = M(\ze(g'_{vu})) = R(M)(g'_{vu}).
\]
Thus we have $M(\bfg_i) = R(M)(\bfg'_i)$ for all $i = 1,2,3$. Hence
\begin{equation}
\label{eq:written-formula-by-mor-Z-00}
   d_M(V_I) =
\rank \bmat{R(M)(\bfg'_1) & \mathbf{0}\\
R(M)(\bfg'_3) & R(M)(\bfg'_2)\\
}
- \rank \bmat{R(M)(\bfg'_1) & \mathbf{0}\\
\bfzero & R(M)(\bfg'_2)\\
}.
\end{equation}
Set here $r\coloneqq d_M(V_I),\ s\coloneqq \bar{d}_{R(M)}(R(V_I))$.
Then it is enough to show that $r = s$.
By the former, we have $M \iso V_I^r \ds N$
for some module $N$ in $\mod \k[\bfP]$, which shows that
$R(M) \iso R(V_I)^r \ds R(N)$.
Hence we have $r \le s$.
On the other hand, by the latter
we have an isomorphism $R(M) \iso R(V_I)^s \ds L$
for some module $L$ in $\mod \k[Z]$.
Then by Lemma \ref{lem:direct-sum-rank}, we have the following equalities:
\begin{equation}
\label{eq:eq:1stRI(M)_(n,m)case-00}
\rank\bmat{R(M)(\bfg'_{1}) & \mathbf{0}\\
R(M)(\bfg'_3) & R(M)(\bfg'_2)}
= s \rank \bmat{R(V_I)(\bfg'_{1}) & \mathbf{0}\\
R(V_I)(\bfg'_3) & R(V_I)(\bfg'_2)} + \rank \bmat{L(\bfg'_{1}) & \mathbf{0}\\
L(\bfg'_3) & L(\bfg'_2)
},
\end{equation}
and
\begin{equation}
\label{eq:2ndRI(M)_(n,m)case-00}
\rank\bmat{R(M)(\bfg'_{1}) & \mathbf{0}\\
\mathbf{0} & R(M)(\bfg'_2)\\
}
= s \rank\bmat{R(V_I)(\bfg'_{1}) & \mathbf{0}\\
\mathbf{0} & R(V_I)(\bfg'_2)
} + \rank\bmat{L(\bfg'_{1}) & \mathbf{0}\\
\mathbf{0} & L(\bfg'_2)
}.
\end{equation}
Note that the formula \eqref{eq:written-formula-by-mor-Z-00} holds also for $M = V_I$.
Thus we have
\begin{equation}
\label{eq:written-formula-by-mor-Z-for-VI-00}
   d_{V_I}(V_I) =
\rank \bmat{R(V_I)(\bfg'_1) & \mathbf{0}\\
R(V_I)(\bfg'_3) & R(V_I)(\bfg'_2)\\
}
- \rank \bmat{R(V_I)(\bfg'_1) & \mathbf{0}\\
\bfzero & R(V_I)(\bfg'_2)\\
}.
\end{equation}
By the equalities \eqref{eq:written-formula-by-mor-Z-00},
\eqref{eq:eq:1stRI(M)_(n,m)case-00},
\eqref{eq:2ndRI(M)_(n,m)case-00}, and
\eqref{eq:written-formula-by-mor-Z-for-VI-00},
we see that
\begin{align*}
    r = d_{M}(V_I) & = s\cdot d_{V_{I}}(V_I)+\rank \bmat{L(\bfg'_{1}) & \mathbf{0}\\
L(\bfg'_3) & L(\bfg'_2)\\
}-\rank \bmat{L(\bfg'_{1}) & \mathbf{0}\\
\mathbf{0} & L(\bfg'_2)\\
}\nonumber \\
& = s + \rank \bmat{L(\bfg'_{1}) & \mathbf{0}\\
L(\bfg'_3) & L(\bfg'_2)\\
}-\rank \bmat{L(\bfg'_{1}) & \mathbf{0}\\
\mathbf{0} & L(\bfg'_2)\\
}\geq s.
\end{align*}
Hence we have $r = s$, and the proof is completed.
\end{proof}

\section{Interval multiplicities by presentations}
\label{sec:multi-by-pp}

For each $M \in \mod A$ and an each interval $I$ of $\bfP$,
we compute, in this section, the multiplicity $d_M(V_I)$
in terms of a projective presentation of $M$ rather than the structure linear maps of $M$.

In what follows, for each event $\mathcal{E}$ such as $(x \le y)$ for $x, y \in \bfP$, we denote by $\de_{\mathcal{E}}$ the $\k$-valued indicator function of $\mathcal{E}$: it takes value $1\in \k$ if $\mathcal{E}$ is true and $0\in\k$ otherwise. To shorten the notation, we write $x \le y, z$ for $x \le y$ and $x \le z$.

\subsection{The formula by projective presentations}

\begin{thm}
\label{thm:multi-pp}
Let $M \in \mod A$ and $I$ an interval of $\bfP$.
Then there exists a projective presentation
\begin{equation}
\label{eq:proj-pres-M}
    \sfP(y) \ya{\sfP(\al)} \sfP(x) \ya{\ep} M \to 0
\end{equation}
of $M$ for some morphism
$\al \colon x \to y$ in $\Ds\!A$,
where we set $x\coloneqq (x_i)_{i \in [m]},\, y\coloneqq (y_j)_{j \in [n]}$.

{\rm Case 1:} $V_I$ is non-projective.  In this case, let
\begin{align}
\label{eq:almost-split-sequence-tauVI}
   0 \to \ta V_I \ya{\mu_I} E_I \ya{\ep_I} V_I \to 0
\end{align}
be an almost split sequence ending in $V_I$.
Then we have the following formula:
\begin{equation}
\label{eq:formula-using-pp-nonproj}
d_M(V_I) =
\rank E_I(\al) - \rank V_I(\al) - \rank\, (\ta V_I)(\al).
\end{equation}

{\rm Case 2:} $V_I$ is projective.  In this case, $I = \uset a$ with $a = \min I$.
We may set $\al = [\al_{ji}]_{(j,i) \in [n]\times [m]}$, where
$\al_{ji} = a_{ji}p_{y_j,x_i}$ for some $a_{ji} \in \k$ and
$\al_{ji} = a_{ji} = 0$ unless $x_i \le y_j$ for all $(j,i) \in [n] \times [m]$.
We set $n_{M,I}\coloneqq \#\{i \in [m] \mid x_i = a\}$.
Then we have the following formula:
\begin{equation}
\label{eq:formula-using-pp-proj}
d_M(V_I) =
\rank [\de_{(a < x_i, y_j)}a_{ji}]_{(j,i)\in [n]\times [m]} - \rank [\de_{(a \le x_i, y_j)}a_{ji}]_{(j,i)\in [n]\times [m]} + n_{M,I}.
\end{equation}
Note that the right hand side is directly computed by information of $\al$.
\end{thm}

\begin{proof}
By Lemma \ref{lem:dim-Hom-coker} we can compute $d_M(V_I)$ as follows:

Case 1. By \cite[Theorem 17 (2.6)]{asashibaIntervalDecomposability2D2022}
(the dual of Theorem \ref{thm:dMV}) and Lemma \ref{lem:dim-Hom-coker}, we have
$$
\begin{aligned}
d_M(V_I) &= \dim \Hom_A(M, \ta V_I) - \dim\Hom_A(M, E_I)
+ \dim\Hom_A(M, V_I)\\
&= \left(\sum_{i \in [m]} \dim (\ta V_I)(x_i) - \rank(\ta V_I)(\al)\right)
- \left(\sum_{i \in [m]} \dim E_I(x_i) - \rank E_I(\al)\right)\\
&\quad\ + \left(\sum_{i \in [m]} \dim V_I(x_i) - \rank V_I(\al)\right)\\
&= \rank E_I(\al) - \rank V_I(\al) - \rank\, (\ta V_I)(\al)
\end{aligned}
$$
because $\sum_{i \in [m]}(\dim (\ta V_I)(x_i) - \dim E_I(x_i) + \dim V_I(x_i)) = 0$ by the exactness of the almost split sequence.

Case 2.  In this case, we have $V_I = V_{\uset a}\iso P_a$, $\rad P_a = V_{\uuset a}$,
and $V_I/V_{\uuset a} \iso V_{\{a\}}$.
By \cite[Theorem 17 (2.5)]{asashibaIntervalDecomposability2D2022}, we have
\begin{align*}
d_M(V_I) &= \dim \Hom_A(M, P_a) - \dim\Hom_A(M, \rad P_a)\\
&= \left(\sum_{i \in [m]} \dim V_{\uset a})(x_i) - \rank V_{\uset a}(\al)\right)
- \left(\sum_{i \in [m]} \dim V_{\uuset a}(x_i) - \rank V_{\uuset a}(\al)\right)\\
&= \sum_{i \in [m]} \dim V_{\{a\}}(x_i)
+\rank V_{\uuset a}(\al) - \rank V_{\uset a}(\al)\\
&= \sum_{i \in [m]} \de_{a, x_i}
+\rank V_{\uuset a}(\al) - \rank V_{\uset a}(\al)\\
&= \rank V_{\uuset a}(\al) - \rank V_{\uset a}(\al) + n_{M,I}.
\end{align*}
Hence the assertion follows from the following:
\begin{equation*}
V_{\uset a}(a_{ji}p_{y_j,x_i}) = \begin{cases}
a_{ji} & \text{if }a \le x_i, y_j,\\
0 & \text{otherwise},
\end{cases}
\quad \text{and}\quad
V_{\uuset a}(a_{ji}p_{y_j,x_i}) = \begin{cases}
a_{ji} & \text{if }a < x_i, y_j,\\
0 & \text{otherwise}.
\end{cases} \tag*{\qedhere}
\end{equation*}
\end{proof}

For convenience, $\sfP(\al)$ in~\eqref{eq:proj-pres-M} is called a \emph{presentation matrix} of $M$. We now exhibit an example of the application of Theorem \ref{thm:multi-pp}.

\begin{exm}
\label{exm:G42-mult-by-pp}
Let $\bfP = G_{4,2}$ and $I \in \bbI$ be as in Example \ref{exm:2Dgrid}.
Then each term of the almost split sequence
$0 \to \ta V_I \to E_I \to V_I \to 0$ ending in $V_I$ is given as follows:
$$
V_I:
\begin{tikzcd}[column sep=20pt]
\k & \k & \k & \k & 0\\
0 & \k & \k & \k & \k
\Ar{1-1}{1-2}{"1"}
\Ar{1-2}{1-3}{"1"}
\Ar{1-3}{1-4}{"1"}
\Ar{1-4}{1-5}{}
\Ar{2-1}{2-2}{}
\Ar{2-2}{2-3}{"1"}
\Ar{2-3}{2-4}{"1"}
\Ar{2-4}{2-5}{"1"}
\Ar{2-1}{1-1}{}
\Ar{2-2}{1-2}{"1"}
\Ar{2-3}{1-3}{"1"}
\Ar{2-4}{1-4}{"1"}
\Ar{2-5}{1-5}{}
\end{tikzcd},
\quad
\ta V_I:
\begin{tikzcd}[column sep=20pt]
\k & \k^2 & \k & \k & \k\\
0 & \k & \k & \k & \k
\Ar{1-1}{1-2}{"\sbmat{1\\0}"}
\Ar{1-2}{1-3}{"{[1,1]}"}
\Ar{1-3}{1-4}{"1"}
\Ar{1-4}{1-5}{"1"}
\Ar{2-1}{2-2}{}
\Ar{2-2}{2-3}{"1"}
\Ar{2-3}{2-4}{"1"}
\Ar{2-4}{2-5}{"1"}
\Ar{2-1}{1-1}{}
\Ar{2-2}{1-2}{"\sbmat{0\\1}"}
\Ar{2-3}{1-3}{"1"}
\Ar{2-4}{1-4}{"1"}
\Ar{2-5}{1-5}{"1"}
\end{tikzcd},
$$
and $E_I = E_1 \ds E_2$, where
$$
E_1:
\begin{tikzcd}[column sep=20pt]
\k & \k^2 & \k & \k & 0\\
0 & \k & \k & \k & \k
\Ar{1-1}{1-2}{"\sbmat{1\\0}"}
\Ar{1-2}{1-3}{"{[1,1]}"}
\Ar{1-3}{1-4}{"1"}
\Ar{1-4}{1-5}{}
\Ar{2-1}{2-2}{}
\Ar{2-2}{2-3}{"1"}
\Ar{2-3}{2-4}{"1"}
\Ar{2-4}{2-5}{"1"}
\Ar{2-1}{1-1}{}
\Ar{2-2}{1-2}{"\sbmat{0\\1}"}
\Ar{2-3}{1-3}{"1"}
\Ar{2-4}{1-4}{"1"}
\Ar{2-5}{1-5}{}
\end{tikzcd},
\quad
E_2:
\begin{tikzcd}[column sep=20pt]
\k & \k & \k & \k & \k\\
0 & \k & \k & \k & \k
\Ar{1-1}{1-2}{"1"}
\Ar{1-2}{1-3}{"1"}
\Ar{1-3}{1-4}{"1"}
\Ar{1-4}{1-5}{"1"}
\Ar{2-1}{2-2}{}
\Ar{2-2}{2-3}{"1"}
\Ar{2-3}{2-4}{"1"}
\Ar{2-4}{2-5}{"1"}
\Ar{2-1}{1-1}{}
\Ar{2-2}{1-2}{"1"}
\Ar{2-3}{1-3}{"1"}
\Ar{2-4}{1-4}{"1"}
\Ar{2-5}{1-5}{"1"}
\end{tikzcd}.
$$
Define an $M \in \mod A$ by $M\coloneqq V_I \ds V_{\{3'\}}$.
Then a projective presentation of $M$ is given by
$\sfP(y) \ya{\sfP(\al)} \sfP(x) \to M \to 0$, where the morphism
$\al \colon x \to y$ in $\Ds\! A$ is given by
$$
\bmat{p_{2',1'} & -p_{2',2} & 0\\
p_{4',1'} & 0 & 0\\
0 & 0 & p_{4',3'}}
\colon (1', 2, 3') \to (2',4',4').
$$
Therefore, by Theorem \ref{thm:multi-pp}, formula \eqref{eq:formula-using-pp-nonproj}, we have
$$
\begin{aligned}
d_M(V_I) &= \rank E_1(\al) + \rank E_2(\al) - \rank V_I(\al) - \rank \,(\ta V_I)(\al)\\
&= \rank \sbmat{1&0&0\\0& -1 & 0\\0&0&0\\0&0&0}
+\rank \sbmat{1&-1&0\\1&0&0\\0&0&1}
- \rank \sbmat{1& -1& 0\\9&0&0\\0&0&0}
- \rank \sbmat{1&0&0\\0&-1&0\\1&0&0\\0&0&1}\\
&= 2+3-1-3 = 1.
\end{aligned}
$$

Take $J \in \bbI$ as $J\coloneqq \uset 2$.
Then $V_J$ is projective, and $d_M(V_J)$ is computed directly from $\al$ above
by Theorem \ref{thm:multi-pp}, formula \eqref{eq:formula-using-pp-proj}.
In this case, $n_{M,J} = 1$, and we have
$$
d_M(V_J) = \rank\sbmat{0&0&0\\0&0&0\\0&0&1}
-\rank\sbmat{0&-1&0\\0&0&0\\0&0&1} + 1 = 0.
$$
\end{exm}

In the theorem above,
the rank of $C(\al)$ for an $A$-module $C$ is easily computed for $C = V_I$
because its structure linear maps are known.
However, for $C = E_I$ or $\ta V_I$ those are not known at first.
In that case, it would be convenient if $\rank C(\al)$ can be computed
by a projective presentation (or an injective copresentation) of $C$
because it can be obtained from the results obtained before.
In this connection,
we now give a formula of $\rank C(\al)$
using a projective presentation (or an injective copresentation) of $C$
for any $C \in \mod A$ and morphism $\al$ in $\Ds\! A$.

\begin{prp}
\label{prp:rank-using-pp}
Let $\al \colon x \to y$ be a morphism in $\Ds\! A$ and
$C \in \mod A$.
Assume that $C$ has a projective presentation
$$
\sfP(v) \ya{\sfP(\be)} \sfP(u) \ya{\pi} C \to 0
$$
for some morphism $\be \colon u \to v$
in $\Ds\! A$.
Then we have (see~\eqref{eq:dfn-sfP-M} for notations)
$$
\rank C(\al) = \dim C(y) - \dim \sfP(u)(y) + \rank [\sfP(\be)(y), \sfP(u)(\al)].
$$
\end{prp}

\begin{proof}
We apply the salamander lemma in the proof, for which we refer the reader
to~\cite{bergmanDiagramchasingDoubleComplexes2012}.
In particular, we use the notations introduced by~\cite{geraschenkoAntonGeraschenkoSalamander2007}.
In a double complex with a term $X$, we denote by
${}_=X,\, X^\parallel,\, {}^\square X$ and $X_\square$,
the horizontal homology, the vertical homology, the receptor,
and the donor at $X$, respectively (see Appendix for details).
By assumption, we have the following double complex
(at first ignore dashed edges, which mean isomorphisms):
\[\begin{tikzcd}[column sep=40pt]
	& 0 & 0 & 0 \\
	0 & {\sfP(v)(x)} & {\sfP(u)(x)} & \Nname{Cx}{{}_=C(x)_\square} & 0 \\
	0 & {\sfP(v)(y)} &\Nname{Puy} {{}_=\sfP(u)(y)_\square} & \Nname{Cy}{{}^\square _=C(y)^\parallel} & 0 \\
	& 0 & 0 & 0
	\arrow[from=1-2, to=2-2]
	\arrow[from=1-3, to=2-3]
	\arrow[from=1-4, to=2-4]
	\arrow[from=2-1, to=2-2]
	\arrow[from=2-2, to=2-3, "\sfP(\be)(x)"]
	\arrow[from=2-2, to=3-2, "\sfP(v)(\al)"']
	\arrow[from=2-3, to=2-4, "\pi_x"]
	\arrow[from=2-3, to=3-3, "\sfP(u)(\al)"]
	\arrow[from=2-4, to=2-5]
	\arrow[from=2-4, to=3-4, "C(\al)"]
	\arrow[from=3-1, to=3-2]
	\arrow[from=3-2, to=3-3, "\sfP(\be)(y)"']
	\arrow[from=3-2, to=4-2]
	\arrow[from=3-3, to=3-4, "\pi_y"' pos=0.7]
	\arrow[from=3-3, to=4-3]
	\arrow[from=3-4, to=3-5]
	\arrow[from=3-4, to=4-4]
\Ar{Cx}{Cx}{dashed, rounded corners, no head,
to path={
([xshift=3ex, yshift=1ex]Cx.south west)
--([yshift=0.3ex]Cx.south)
--([xshift=-3ex, yshift=0.8ex]Cx.south east)
}
}
\Ar{Cy}{Cy}{dashed, rounded corners, no head,
to path={
([xshift=3ex, yshift=-1.2ex]Cy.north west)
--([yshift=-0.2ex]Cy.north)
--([xshift=-2.4ex, yshift=-1.2ex]Cy.north east)
}
}
\Ar{Cy}{Puy}{dashed, rounded corners, no head,
to path={
([xshift=0.8ex, yshift=-1.2ex]Cy.north west)
--([xshift=-2.5ex,yshift=-1.2ex]Cy.north west)
--([xshift=2.5ex,yshift=1ex]Puy.south east)
--([xshift=-0.8ex,yshift=1ex]Puy.south east)
}
}
\end{tikzcd},\]
where the horizontal homologies
${}_=\sfP(u)(x)$, ${}_=C(x)$, ${}_=\sfP(u)(y)$, and ${}_=C(y)$
vanish.
By the salamander lemma, we have
$C(y)^\parallel \iso \sfP(u)(y)_\square$.
Indeed,
since $\pi_x$ is an epimorphism,
$0 = {}_=C(x) \iso C(x)_\square$ by Corollary \ref{cor:salamander-corner}. This shows that $C(y)^\parallel \iso {}^\square C(y)$ again by Corollary \ref{cor:salamander-corner}. We also have ${}^\square C(y) \iso \sfP(u)(y)_\square$
by Corollary \ref{cor:salamander-extra} because ${}_=\sfP(u)(y) = 0 = {}_=C(y)$.
Thus $C(y)^\parallel \iso \sfP(u)(y)_\square$, as desired.
Since by definition $\Cok C(\al) = C(y)^\parallel$,
we have $\Cok C(\al) \iso \sfP(u)(y)_\square$, where
the right hand side is, by definition, isomorphic to
$\sfP(u)(y)/\Im[\sfP(\be)(y), \sfP(u)(\al)]$.
Hence
\begin{align*}
\rank C(\al) &= \dim \Im C(\al) = \dim C(y) - \dim \Cok C(\al)\\
&= \dim C(y) -(\dim \sfP(u)(y) - \dim \Im [\sfP(\be)(y), \sfP(u)(\al)])
\\
&= \dim C(y) - \dim \sfP(u)(y) + \rank [\sfP(\be)(y), \sfP(u)(\al)]. \tag*{\qedhere}
\end{align*}
\end{proof}

We set $\sfQ\coloneqq D\sfP'$ to be the composite of
$\Ds\!A \ya{\sfP'} \mod A\op \ya{D} \mod A $.
Since $\sfP'$ is covariant, $\sfQ$ is contravariant.
Then any morphism between injectives in $\mod A$ can be written
as $\sfQ(\be) \colon \sfQ(u) \to \sfQ(v)$ for some
$\be \colon v \to u$ in $\Ds\! A$.
By using the duality $D$, we obtain the following
from Proposition \ref{prp:rank-using-pp}:

\begin{prp}
\label{prp:rank-using-icp}
Let $\al \colon x \to y$ be a morphism in $\Ds\! A$ and
$C \in \mod A$.
Assume that $C$ has an injective copresentation
$$
0 \to C \ya{\si} \sfQ(u) \ya{\sfQ(\be)} \sfQ(v)
$$
for some morphism $\be \colon v \to u$ in $\Ds\! A$.
Then we have
$$
\rank C(\al) = \dim C(x) - \dim \sfP'(u)(x) + \rank [\sfP'(\be)(x), \sfP'(u)(\al)].
$$
\end{prp}

To apply the proposition above, we record injective copresentations of
$V_I$, $E_I$ and $\ta V_I$ in the general finite poset case and the 2D-grid case.

\begin{prp}[General case]
\label{prp:min-inj-p-VI-gen}
Let $\bfP$ be a finite poset, and $I$ an interval of $\bfP$.
Assume that $V_I$ is non-projective
and let \eqref{eq:almost-split-sequence-tauVI}
be an almost split sequence ending in $V_I$.
Choose any choice maps $\bfc \colon \src(\uuset I) \linebreak[2]\to \src(I)$
and $\bfd \colon \snk(\ddset I)\linebreak[2] \to \snk(I)$, and choose also any $(b, a) \in \snk(I) \times \src(I)$ such that $b \ge a$.
Using these, define morphisms $\bfg_1$, $\bfg_2$ and $\bfg_3$
in $\DS \k[\bfP]$ as in~{\rm \cref{dfn:matrices-g}}, and set
$$
\bfg\coloneqq \bfg(\bfc, \bfd, (b,a)) \coloneqq \left[
\begin{array}{c|c}
\bfg_1 & \bfzero\\
\hline
\bfg_3 & \bfg_2
\end{array}
\right].
$$
Then there exists an injective module $Q$ such that
$V_I$, $\ta V_I \ds Q$ and $E_I \ds Q$ have the following injective copresentations.
\begin{align}
&0\to V_{I} \to \sfQ(\snk(I)) \ya{\sfQ(\bfg_2)} \sfQ(\snk(\ddset I) \ds \snk_{1}(I)),\\
&0\to \ta V_{I} \ds Q \to \sfQ(\src_{1}(I) \ds \src(\uuset I)) \ya{\sfQ(\bfg_1)} \sfQ(\src(I)), \label{eq:min-inj-p-tauVI-gen}\\
&0\to E_I \ds Q \to \sfQ(\src_{1}(I) \ds \src(\uuset I) \ds \snk(I)) \ya{\sfQ(\bfg)} \sfQ(\src(I) \ds \snk(\ddset I) \ds \snk_{1}(I)).
\end{align}
\end{prp}

\begin{proof}
By applying $D$ to \eqref{eq:almost-split-sequence-tauVI}, we have
the following almost split sequence of $DV_I$ in $\mod A\op$:
$$
0 \to DV_I \to DE_I \to \Tr V_I \to 0.
$$
Since $DV_I \iso V_{I\op}^{\bfP\op}$, this sequence becomes
an almost split sequence starting from $V_{I\op}^{\bfP\op}$,
and hence we can compute projective presentations
of these three terms by~\cref{prp:proj.presentation for V_I}
and \ref{prp:proj-pre of tau-VI-ds-P2} and
\eqref{eq:pp-E+P2}.
For $DV_I$ it is already given in
\eqref{eq:proj-pre-DV_I-nsrccase-lem w/o conditions}.
Altogether, these have the following forms:
\begin{align}
\label{eq:proj-pre-DV_I-nsrccase-lem w/o conditions-unif-form}
\sfP'(\snk(\ddset I) \ds \snk_1(I)) \ya{\sfP'(\bfg_2)}
\sfP'(\snk(I)) &\to DV_I \to 0,\\
\sfP'(\src(I)) \ya{\sfP'(\bfg_1)} \sfP'(\src_1(I) \ds \src(\uuset I)) &\to \Tr V_I \ds P'_2 \to 0,\\
\sfP'(\src(I) \ds \snk(\ddset I) \ds \snk_1(I)) \ya{\sfP'(\bfg)}
\sfP'(\src_1(I)\ds \src(\uuset I) \ds \snk(I)) &\to DE_I \ds P'_2 \to 0,
\end{align}
where $P'_2$ is a projective module in $\mod A\op$.
By applying $D$ to these projective presentations, we obtain the assertion.
\end{proof}

By specializing in the 2D-grid case, we have the following.
Note that in this case, we can take $Q = 0$.

\begin{prp}[2D-grid case]
\label{prp:min-inj-p-VI}
Let $\bfP$ be a 2D-grid, and $I$ an interval of $\bfP$. Assume that $V_I$ is non-projective
and let \eqref{eq:almost-split-sequence-tauVI}
be an almost split sequence ending in $V_I$.
Then using the same notations as above,
$V_I$, $\ta V_I$ and $E_I$ have the following injective copresentations.
\begin{align}
&0\to V_{I} \to \sfQ(\snk(I)) \ya{\sfQ(\bfg_2)} \sfQ(\snk(\ddset I) \ds \snk^\circ_{1}(I)),\\
&0\to \ta V_{I} \to \sfQ(\src_{1}^\circ(I) \ds \src(\uuset I)) \ya{\sfQ(\bfg_1)} \sfQ(\src(I)), \label{eq:min-inj-p-tauVI}\\
&0\to E_I \to \sfQ(\src^\circ_{1}(I) \ds \src(\uuset I) \ds \snk(I)) \ya{\sfQ(\bfg)} \sfQ(\src(I) \ds \snk(\ddset I) \ds \snk^\circ_{1}(I)).
\end{align}
\end{prp}

Finally, we have the following formula for $d_M(V_I)$.

\begin{thm}
\label{thm:formula-pp(M)-icp(ass)}
Let $I$ be an interval of $\bfP$, and $M \in \mod A$ with
a projective presentation
$$
\sfP(y) \ya{\sfP(\al)} \sfP(x) \ya{\ep} M \to 0
$$
for some morphism $\al \colon x \to y$ in $\Ds\!A$.
Keep the notations introduced in~{\rm \cref{prp:min-inj-p-VI-gen}}.  Then we have the following formula for $d_M(V_I)$:
\begin{equation}
\label{eq:formula-with-pp}
\begin{aligned}
d_M(V_I)
&= \rank \left[\begin{array}{c|c|cc}
 \sfP'(\bfg_1)(x) & \bfzero & \sfP'(\src_{1}(I) \ds \src(\uuset I))(\al) & \bfzero\\
 \hline
\sfP'(\bfg_3)(x)
& \sfP'(\bfg_2)(x) &
\bfzero & \sfP'(\snk(I))(\al)
 \end{array}\right]\\
 &\quad - \rank \left[\begin{array}{c|c|cc}
 \sfP'(\bfg_1)(x) & \bfzero & \sfP'(\src_{1}(I) \ds \src(\uuset I))(\al) & \bfzero\\
 \hline
\bfzero
& \sfP'(\bfg_2)(x) &
\bfzero & \sfP'(\snk(I))(\al)
 \end{array}\right].
 \end{aligned}
\end{equation}
Note that for the 2D-grid case, we can replace $\src_1(I)$ with $\src_1^\circ(I)$.
\end{thm}

\begin{proof}
{\bf Case 1.} $V_I$ is non-projective.

By Theorem \ref{thm:multi-pp} and Propositions \ref{prp:rank-using-icp} and \ref{prp:min-inj-p-VI-gen},
we have
$$
\begin{aligned}
\rank E_I(\al) + \rank Q(\al)&= \dim E_I(x) - \dim \sfP'(\src_{1}(I) \ds \src(\uuset I) \ds \snk(I))(x) \\
  &\quad + \rank[\sfP'(\bfg)(x), \sfP'(\src_{1}(I) \ds \src(\uuset I) \ds \snk(I))(\al)],\\
\rank V_I(\al) &= \dim V_I(x) - \dim \sfP'(\snk(I))(x) + \rank [\sfP'(\bfg_2)(x), \sfP'(\snk(I))(\al)],\\
\rank (\ta V_I)(\al) + \rank Q(\al)&= \dim (\ta V_I)(x) - \dim \sfP'(\src_{1}(I) \ds \src(\uuset I))(x) \\
  &\quad + \rank [\sfP'(\bfg_1)(x), \sfP'(\src_{1}(I) \ds \src(\uuset I))(\al)].
\end{aligned}
$$
Therefore
by \eqref{eq:formula-using-pp-nonproj}, we have
{\small
\[
\setlength{\jot}{1pt}
\mspace{-10mu}
\begin{aligned}
d_M(V_I)
&= \dim E_I(x)-\dim V_I(x)-\dim(\ta V_I)(x)
 + \dim \sfP'(\snk(I))(x)\\
&\quad + \dim \sfP'(\src_{1}(I)\ds \src(\uuset I))(x)
 - \dim \sfP'(\src_{1}(I)\ds \src(\uuset I)\ds \snk(I))(x)\\
&{}+ \rank[\sfP'(\bfg)(x),\sfP'(\src_{1}(I)\ds \src(\uuset I)\ds \snk(I))(\al)]\\
&{}- \rank[\sfP'(\bfg_2)(x),\sfP'(\snk(I))(\al)]
 - \rank[\sfP'(\bfg_1)(x),\sfP'(\src_{1}(I)\ds \src(\uuset I))(\al)]\\
&= \dim \sfP'(\snk(I))(x)
 + \dim \sfP'(\src_{1}(I)\ds \src(\uuset I))(x)
 - \dim \sfP'(\src_{1}(I)\ds \src(\uuset I)\ds \snk(I))(x)\\
&{}+ \rank[\sfP'(\bfg)(x),\sfP'(\src_{1}(I)\ds \src(\uuset I)\ds \snk(I))(\al)]\\
&{}- \rank[\sfP'(\bfg_2)(x),\sfP'(\snk(I))(\al)]
 - \rank[\sfP'(\bfg_1)(x),\sfP'(\src_{1}(I)\ds \src(\uuset I))(\al)]\\
&= \rank[\sfP'(\bfg)(x),\sfP'(\src_{1}(I)\ds \src(\uuset I)\ds \snk(I))(\al)]\\
&{}- \rank[\sfP'(\bfg_2)(x),\sfP'(\snk(I))(\al)]
 - \rank[\sfP'(\bfg_1)(x),\sfP'(\src_{1}(I)\ds \src(\uuset I))(\al)]\\
&= \text{RHS of \eqref{eq:formula-with-pp}}.
\end{aligned}
\]
}

{\bf Case 2.} $V_I$ is projective.

The assertion is proved in a way similar to Case 2 in Theorem \ref{thm:formula-icp(M)-pp(ass)} below.
\end{proof}

\begin{rmk}
\label{rmk:order-smnds}
Notice the domains and the codomains of the block matrices of
the big matrix in Theorem \ref{thm:formula-pp(M)-icp(ass)}, which are as follows:
$$
\begin{aligned}
\sfP'(\bfg_1)(x) &\colon \sfP'(\src(I))(x) \to \sfP'(\src_1(I) \ds \src(\uuset I))(x),\\
\sfP'(\bfg_2)(x) &\colon \sfP'(\snk(\ddset I) \ds \snk_1(I))(x) \to \sfP'(\snk(I))(x),\\
\sfP'(\bfg_3)(x) &\colon \sfP'(\src(I))(x) \to \sfP'(\snk(I))(x),\\
\sfP'(\src_1(I) \ds \src(\uuset I))(\al) &\colon \sfP'(\src_1(I) \ds \src(\uuset I))(y)
\to \sfP'(\src_1(I) \ds \src(\uuset I))(x),\\
\sfP'(\snk(I))(\al) &\colon \sfP'(\snk(I))'y) \to \sfP'(\snk(I))(x).
\end{aligned}
$$
In particular, $\sfP'(\bfg_1)(x)$ and
$\sfP'(\src_1(I) \ds \src(\uuset I))(\al)$ have the common codomain
$\sfP'(\src_1(I) \ds \src(\uuset I))(x)$ so that they are in the same row
in the big matrices.
When we write their concrete matrices,
we have to have the same order of the direct summands
of $\sfP'(\src_1(I) \ds \src(\uuset I))(x)$.
In that case, if $u = (u_1,\dots, u_m), v = (v_1, \dots, v_n)$ with $m, n \ge 1$
in $\Ds\k[\bfP]$, then we use the following order
\begin{multline*}
\sfP'(u)(v) = \sfP'(u)(v_1)\ds \cdots \ds \sfP'(u)(v_n) = \sfP'_{u_1}(v_1) \ds \cdots \ds \sfP'_{u_m}(v_1)\ds \sfP'_{u_1}(v_2) \ds \cdots \ds\\
\sfP'_{u_m}(v_2)\ds \cdots \ds \sfP'_{u_1}(v_n) \ds \cdots \ds \sfP'_{u_m}(v_n).
\end{multline*}
The same remark is made for the second row of the big matrices
(about the order of the summands of $\sfP'(\snk(I))(x)$).
\end{rmk}

\begin{exm}
\label{exm:proj-int-mult-comp}
We compute the same multiplicity as in Example \ref{exm:G42-mult-by-pp}
by using Theorem \ref{thm:formula-pp(M)-icp(ass)}.
Let $\bfP = G_{4,2}$ and $I \in \bbI$ be as in Example \ref{exm:2Dgrid}.
Then $\udim V_I\coloneqq \sbmat{1&1&1&0\\0&1&1&1}$, and
$a_1 = 2,\, a_2 = 1',\, a_{12} = 2',\, b_1 = 4,\, b_2 = 3',\, b_{12} = 3,
\, a'_1 = 4',\, b'_1 = 1$.
Therefore,
$$
\bfg =
\left[\begin{array}{cc|cc}
p_{a_{12},a_1} & -p_{a_{12},a_2} & \bfzero& \bfzero\\
p_{a'_1,a_1} & \bfzero & \bfzero& \bfzero\\
\hline
p_{b_1,a_1} & \bfzero & p_{b_1, b'_1} & p_{b_1,b_{12}}\\
\bfzero & \bfzero & \bfzero & -p_{b_2, b_{12}}
\end{array}\right]
=
\left[\begin{array}{cc|cc}
p_{2',2} & -p_{2',1'} & \bfzero& \bfzero\\
p_{4',2} & \bfzero & \bfzero& \bfzero\\
\hline
p_{4,2} & \bfzero & p_{4, 1} & p_{4,3}\\
\bfzero & \bfzero & \bfzero & -p_{3', 3}
\end{array}\right].
$$
Namely,
$$
\begin{aligned}
\bfg_1 &= \sbmat{p_{2',2} & -p_{2',1'}\\p_{4',2} & \bfzero} \colon
(2, 1') \to (2', 4'),\\
\bfg_2 &= \sbmat{p_{4, 1} & p_{4,3}\\\bfzero & -p_{3', 3}} \colon(1,3) \to (4,3'),\\
\bfg_3 &= \sbmat{p_{4,2} & \bfzero\\\bfzero & \bfzero} \colon
(2,1') \to (4,3').
\end{aligned}
$$
Moreover, a projective presentation of $M$ is given by
$\sfP(y) \ya{\sfP(\al)} \sfP(x) \to M \to 0$, where the morphism
$\al \colon x \to y$ in $\Ds\! A$ is given by
$$
\bmat{p_{2',1'} & -p_{2',2} & 0\\
p_{4',1'} & 0 & 0\\
0 & 0 & p_{4',3'}}
\colon (1', 2, 3') \to (2',4',4').
$$
We compute the matrix
$$
\sfM\coloneqq \left[\begin{array}{c|c|cc}
 \sfP'(\bfg_1)(x) & \bfzero & \sfP'(\src_{1}(I) \ds \src(\uuset I))(\al) & \bfzero\\
 \hline
\sfP'(\bfg_3)(x)
& \sfP'(\bfg_2)(x) & \bfzero & \sfP'(\snk(I))(\al)
 \end{array}\right],
$$
which is equal to
$$
\begin{aligned}
\left[\begin{array}{c|c|cc}
 \sfP'(\bfg_1)(1',2,3') & \bfzero & \sfP'(2',4')(\al) & \bfzero\\
 \hline
\sfP'(\bfg_3)(1',2,3')
& \sfP'(\bfg_2)(1',2,3') & \bfzero & \sfP'(4,3')(\al)
 \end{array}\right].
\end{aligned}
$$
Here, $\sfP'(\bfg_1)(1',2,3') = \sfP'(\bfg_1)(1') \ds \sfP'(\bfg_1)(2) \ds \sfP'(\bfg_1)(3')$ and
$$
\sfP'(\bfg_1)(1') = \sbmat{P'_{2',2}(1') & -P'_{2',1'}(1')\\P'_{4',2}(1') & \bfzero}
\colon P'_2(1') \ds P'_{1'}(1') \to P'_{2'}(1') \ds P'_{4'}(1'),
$$
which is a linear map
$$
\begin{tikzcd}[column sep=40pt, ampersand replacement=\&]
\k[\bfP](1',2) \ds \k[\bfP](1',1') \& \k[\bfP](1',2') \ds \k[\bfP](1',4')\\
0 \ds \k p_{1',1'} \& \k p_{2',1'} \ds \k p_{1',4'}\\
0 \ds \k \& \k \ds \k
\Ar{1-1}{1-2}{"\sfP'(\bfg_1)(1')"}
\Ar{1-1}{2-1}{equal}
\Ar{1-2}{2-2}{equal}
\Ar{2-1}{2-2}{"\sbmat{0 &-\la_{p_{2',1'}}\\0 & 0}"}
\Ar{3-1}{3-2}{"\sbmat{0 & -1\\0&0}"}
\Ar{2-1}{3-1}{"\iso"}
\Ar{2-2}{3-2}{"\iso"}
\end{tikzcd}
$$
Therefore, we may have an identification $\sfP'(\bfg_1)(1') = \sbmat{0 & -1\\0&0}$.
Similarly, we have identifications
$\sfP'(\bfg_1)(2) = \sbmat{1 & 0\\1&0}$ and
$\sfP'(\bfg_1)(3') = \sbmat{0 & 0\\0&0}$, and hence we have
$$
\sfP'(\bfg_1)(x) = \sbmat{0 & -1\\0&0} \ds \sbmat{1 & 0\\1&0} \ds \sbmat{0 & 0\\0&0}.
$$
Similarly,
$$
\begin{aligned}
\sfP'(\bfg_2)(x) &= \sbmat{0 & 0\\0&0} \ds \sbmat{0 & 1\\0&-1} \ds \sbmat{0 & 0\\0&0},\\
\sfP'(\bfg_3)(x) &= \sbmat{0 & 0\\0&0} \ds \sbmat{1 & 0\\0&0} \ds \sbmat{0 & 0\\0&0}.
\end{aligned}
$$
Moreover,
$\sfP'(2',4')(\al)\colon \sfP'(2',4')(y) \to \sfP'(2',4')(x)$ is
computed as follows.
$$
\begin{aligned}
& \sfP'(2',4')(y) \\
= & \sfP'_{2'}(y) \ds \sfP'_{4'}(y)\\
= & \sfP'_{2'}(2') \ds \sfP'_{2'}(4') \ds \sfP'_{2'}(4')\\
 & \ds \sfP'_{4'}(2') \ds \sfP'_{4'}(4') \ds \sfP'_{4'}(4')\\
\iso & \k \ds 0 \ds 0 \ds \k \ds \k \ds \k,
\end{aligned}
\qquad
\begin{aligned}
& \sfP'(2',4')(x) \\
= & \sfP'_{2'}(x) \ds \sfP'_{4'}(x)\\
= & \sfP'_{2'}(1') \ds \sfP'_{2'}(2) \ds \sfP'_{2'}(3')\\
 & \ds \sfP'_{4'}(1') \ds \sfP'_{4'}(2) \ds \sfP'_{4'}(3')\\
\iso &\k \ds \k \ds 0 \ds \k \ds \k \ds \k.
\end{aligned}
$$
Between these modules,
$$
\begin{aligned}
\sfP'(2',4')(\al) &= \sfP'_{2'}(\al) \ds \sfP'_{4'}(\al)\\
&=\sbmat{\sfP'_{2'}(p_{2',1'}) &\sfP'_{2'}(p_{4',1'}) & 0\\
-\sfP'_{2'}(p_{2',2}) & 0 & 0\\
0 & 0& \sfP'_{2'}(p_{4',3'})} \ds
\sbmat{\sfP'_{4'}(p_{2',1'}) &\sfP'_{4'}(p_{4',1'}) & 0\\
-\sfP'_{4'}(p_{2',2}) & 0 & 0\\
0 & 0& \sfP'_{4'}(p_{4',3'})}\\
&= \sbmat{1 & 0 & 0\\
-1 &0 & 0\\
0 &0&0} \ds
\sbmat{1&1&0\\
-1 &0 &0\\
0 & 0& 1}.
\end{aligned}
$$
Hence by reordering the summands of the codomain as in Remark \ref{rmk:order-smnds},
we have
$$
\sfP'(2',4')(\al) = \sbmat{
1&0&0&&&\\
&&&1&1&0\\
-1&0&0&&&\\
&&&-1&0&0\\
0&0&0&&&\\
&&&0&0&1
},
$$
where blanks are zeros.
Similarly, we have
$$
\begin{aligned}
\sfP'(4,3')(\al) &= \sfP'_{4}(\al) \ds \sfP'_{3'}(\al)\\
&=\sbmat{\sfP'_{4}(p_{2',1'}) &\sfP'_{4}(p_{4',1'}) & 0\\
-\sfP'_{4}(p_{2',2}) & 0 & 0\\
0 & 0& \sfP'_{4}(p_{4',3'})} \ds
\sbmat{\sfP'_{3'}(p_{2',1'}) &\sfP'_{3'}(p_{4',1'}) & 0\\
-\sfP'_{3'}(p_{2',2}) & 0 & 0\\
0 & 0& \sfP'_{3'}(p_{4',3'})}\\
&= \sbmat{0 & 0 & 0\\
0 &0 & 0\\
0 &0&0} \ds
\sbmat{1&0&0\\
-1 &0 &0\\
0 & 0& 0}.
\end{aligned}
$$
After reordering the summands of the codomain, we have
$$
\sfP'(4,3')(\al) = \sbmat{
0&0&0&&&\\
&&&1&0&0\\
0&0&0&&&\\
&&&-1&0&0\\
0&0&0&&&\\
&&&0&0&0
}.
$$
Therefore,
{\footnotesize
$$
\sfM = \left[\begin{array}{c|c|cc}
\sbmat{0 & -1\\0&0} \ds \sbmat{1 & 0\\1&0} \ds \sbmat{0 & 0\\0&0} &
\bfzero & \sbmat{
1&0&0&&&\\
&&&1&1&0\\
-1&0&0&&&\\
&&&-1&0&0\\
0&0&0&&&\\
&&&0&0&1
}& \bfzero\\
\hline
\sbmat{0 & 0\\0&0} \ds \sbmat{1 & 0\\0&0} \ds \sbmat{0 & 0\\0&0}
&\sbmat{0 & 0\\0&0} \ds \sbmat{0 & 1\\0&-1} \ds \sbmat{0 & 0\\0&0}
&\bfzero & \sbmat{
0&0&0&&&\\
&&&1&0&0\\
0&0&0&&&\\
&&&-1&0&0\\
0&0&0&&&\\
&&&0&0&0
}
\end{array}\right].
$$
}

A direct computation shows that $\rank \sfM = 8$, and the rank of the remaining big matrix is 7.
Hence by Theorem \ref{thm:formula-pp(M)-icp(ass)}, we have
$$
d_M(V_I) = 8-7 =1.
$$
\end{exm}

For easy use in the future, we summarize a fast way to write down the block matrices appearing in~\cref{thm:formula-pp(M)-icp(ass),exm:proj-int-mult-comp} as in Proposition \ref{prp:fast-way-block-matrices} below. For this sake, we first prepare the following notation.

\begin{ntn}
Let $\al \colon (x_i)_{i \in [m]} \to (y_j)_{j \in [n]}$ be a morphism
in $\Ds\! A$.
Then there exists a unique matrix $[a_{ji}]_{(j,i) \in [n]\times [m]}$ over $\k$ such that
$$
\al\coloneqq [a_{ji} p_{y_j, x_i}]_{(j,i) \in [n]\times [m]},
$$
where $a_{ji} = 0$ unless $x_i \le y_j$ in $\bfP$.
In this case, for each $u \in \bfP$, we also have
$$
\begin{aligned}
\sfP(\al)&= {}^t[a_{ji} \sfP_{y_j, x_i}]_{(j,i) \in [n]\times [m]},\\
\sfP'(\al)(u)&= [a_{ji} \sfP'(p_{y_j, x_i})(u)]_{(j,i) \in [n]\times [m]},\\
\sfP'_u(\al)&= {}^t[a_{ji} \sfP'_u(p_{y_j, x_i})]_{(j,i) \in [n]\times [m]}.
\end{aligned}
$$
We then set
$$
\begin{aligned}
\Mat(\al)&\coloneqq [a_{ji}]_{(j,i) \in [n]\times [m]},\\
 \Mat(\sfP'(\al)(u))&\coloneqq [a'_{ji}]_{(j,i) \in [n]\times [m]},
\end{aligned}
\qquad
\begin{aligned}
\Mat(\sfP(\al))&\coloneqq {}^t[a_{ji}]_{(j,i) \in [n]\times [m]},\\
\Mat(\sfP'_u(\al))&\coloneqq {}^t[a''_{ji}]_{(j,i) \in [n]\times [m]},
\end{aligned}
$$
where
$$
a'_{ji}\coloneqq \begin{cases}
a_{ji} & \text{if }\sfP'(p_{y_j, x_i})(u) \ne 0,\\
0 & \text{otherwise}
\end{cases},
\quad
a''_{ji}\coloneqq \begin{cases}
a_{ji} & \text{if }\sfP'_u(p_{y_j, x_i}) \ne 0,\\
0 & \text{otherwise}
\end{cases}
$$
for all $(j,i) \in [n]\times [m]$, and call each of them the \emph{coefficient matrix} of $\al$,
$\sfP(\al), \sfP'(\al)(u)$, and $\sfP'_u(\al)$, respectively.
\end{ntn}

\begin{prp}
\label{prp:fast-way-block-matrices}
Let $\bfP = (\bfP, \leq)$ be a poset, $I$ an interval of $\bfP$, and $M\in \mod A$.
Suppose that $\sfP(\al)$ is a presentation matrix of $M$
for some morphism $\al \colon (x_i)_{i \in [m]} \to (y_j)_{j \in [n]}$ in $\Ds\! A$ as in Theorem \ref{thm:multi-pp}, and
$\bfg_{t} \colon (z_j)_{j \in [s]} \to (w_i)_{i \in [r]} \ (t = 1,2,3)$ is a nonzero block of the multiplicity matrix for $I$.
Set the coefficient matrices of $\bfg_t$ and of $\sfP(\al)$ as follows:
\begin{equation}
\Mat(\bfg_t)\coloneqq\begin{blockarray}{cccc}
        & z_1 & \cdots & z_s \\
    \begin{block}{c[c|c|c]}
        w_1    &     &        &     \\
        \vdots & \bfg^{(1)}_{t} & \cdots & \bfg^{(s)}_{t} \\
        w_r    &     &        &     \\
    \end{block}
\end{blockarray},\
     \textnormal{and}\
\Mat(\sfP(\al))\coloneqq \begin{blockarray}{cccc}
        & y_1 & \cdots & y_n \\
    \begin{block}{c[c|c|c]}
        x_1    &     &        &     \\
        \vdots & \bsa^{(1)} & \cdots & \bsa^{(n)} \\
        x_m    &     &        &     \\
    \end{block}
\end{blockarray}
\end{equation}
as row vectors consisting of column vectors.

Then $\sfP'(\bfg_{t})\big((x_i)_{i \in [m]}\big) = \Ds^{m}_{i=1} \sfP'(\bfg_{t})(x_{i})$ is a diagonal block matrix, where for each $i \in [m]$, the $i$-{\rm th} block $\sfP'(\bfg_{t})(x_{i})$ has the following coefficient matrix:
\begin{equation}
\label{eq:P'(g)(u)}
    \left[
    \begin{array}{c|c|c}
         & & \\
         \de_{(x_{i}\leq z_{1})}\cdot\bfg^{(1)}_{t} & \cdots & \de_{(x_{i}\leq z_{s})}\cdot\bfg^{(s)}_{t} \\
         & &
    \end{array}
    \right].
\end{equation}
Intuitively, for each $j \in [s]$, if $x_{i} \leq z_j$ in $\bfP$, then the $j$-{\rm th} columns of $\sfP'(\bfg_{t})(x_{i})$ and of $\bfg_{t}$ coincide, otherwise the $j$-{\rm th} column of $\sfP'(\bfg_{t})(x_{i})$ is zero.

On the other hand, $\sfP'\big((w_i)_{i \in [r]}\big)(\al) = \Ds^{r}_{i=1} \sfP'_{w_{i}}(\al)$ is a diagonal block matrix, where for each $i \in [r]$, the $i$-{\rm th} block $\sfP'_{w_{i}}(\al)$ has the following coefficient matrix:
\begin{equation}
\label{eq:P'_u(al)}
    \left[
    \begin{array}{c|c|c}
         & & \\
         \de_{(y_{1}\leq w_{i})}\cdot\bsa^{(1)} & \cdots & \de_{(y_{n}\leq w_{i})}\cdot\bsa^{(n)}\\
         & &
    \end{array}
    \right].
\end{equation}
\end{prp}

\begin{proof}
Set $\Mat(\bfg_t)\coloneqq [b_{ij}]_{(i,j)\in [r]\times [s]}$ and
$\Mat(\al)\coloneqq [a_{ji}]_{(j,i)\in [n]\times [m]}$, and let $u \in \bfP$.
Then since $\sfP'(p_{w_i, z_j})(u) = A(u, p_{w_i,z_j})$, we have
$$
\sfP'(\bfg_t)(u) = [b_{ij} A(u, p_{w_i,z_j})]_{(i,j)\in [r]\times [s]},
$$
where for any pair $(w_i, z_j)$ with $z_j \le w_i$,
the morphism $A(u, p_{w_i,z_j}) \colon A(u, z_j) \to A(u, w_i)$ is nonzero if and only if $A(u, z_j) \ne 0$, if and only if $u \le z_j$.
Hence we have $\Mat(\sfP'(g_t)(u)) = [b'_{ij}]_{(i,j)\in [r]\times [s]}$, where
$$
b'_{ij} = \de_{(A(u, p_{w_i,z_j}) \ne 0)}b_{ij}
= \de_{(u \le z_j)}b_{ij}.
$$
Therefore, \eqref{eq:P'(g)(u)} follows by setting $u\coloneqq x_i$.

Similarly, since $\sfP'_u(p_{y_j, x_i}\!) = A(p_{y_j, x_i}, u)$, we have
$$
\sfP'_u(\al) = {}^t[a_{ji}A(p_{y_j, x_i}, u)]_{(j,i)\in [n]\times [m]},
$$
where for any pair $(x_i, y_j)$ with $x_i \le y_j$,
the morphism $A(p_{y_j, x_i}, u) \colon A(y_j, u) \to A(x_i, u)$ is nonzero
if and only if $A(y_j, u) \ne 0$, if and only if $y_j \le u$.
Hence we have $\Mat(\sfP'_u(\al)) = {}^t[a''_{ji}]_{(j,i) \in [n]\times [m]}$, where
$$
a''_{ji} = \de_{(A(p_{y_j.x_i},u) \ne 0)} a_{ji}
= \de_{(y_j \le u)} a_{ji}.
$$
By setting $u = w_i$, this shows \eqref{eq:P'_u(al)}.
\end{proof}

\begin{exm}
\label{exm:Fugacci}
We take a bifiltration example from~\cite{fugacciCompression2parameterPersistent2023}, as displayed in~\cref{fig:3x3-bifiltration}, to demonstrate our formulas. Set $M\coloneqq H_{1}(\blank;\bbZ/2\bbZ)\circ \calF$. Following the notation given in~\cref{thm:multi-pp}, the presentation matrix $\sfP(\al)$ is given by
\begin{align}
\label{eq:pp-of-M-exm}
    \begin{blockarray}{ccc}
        & \scalebox{0.8}{$(1, 2)$} & \scalebox{0.8}{$(2, 1)$} \\
        \begin{block}{c[cc]}
            \scalebox{0.8}{$(0, 0)$} & 1 & 0\\
            \scalebox{0.8}{$(1, 1)$} & 1 & 1\\
        \end{block}
    \end{blockarray},
\end{align}
and thus $x$ in~\eqref{eq:formula-with-pp} is given by a sequence of row indices of $\sfP(\al)$. Namely, $x = \big((0,0),(1,1)\big)$.

\begin{figure}[htbp]
    \centering
    \includegraphics[width=0.4\textwidth]{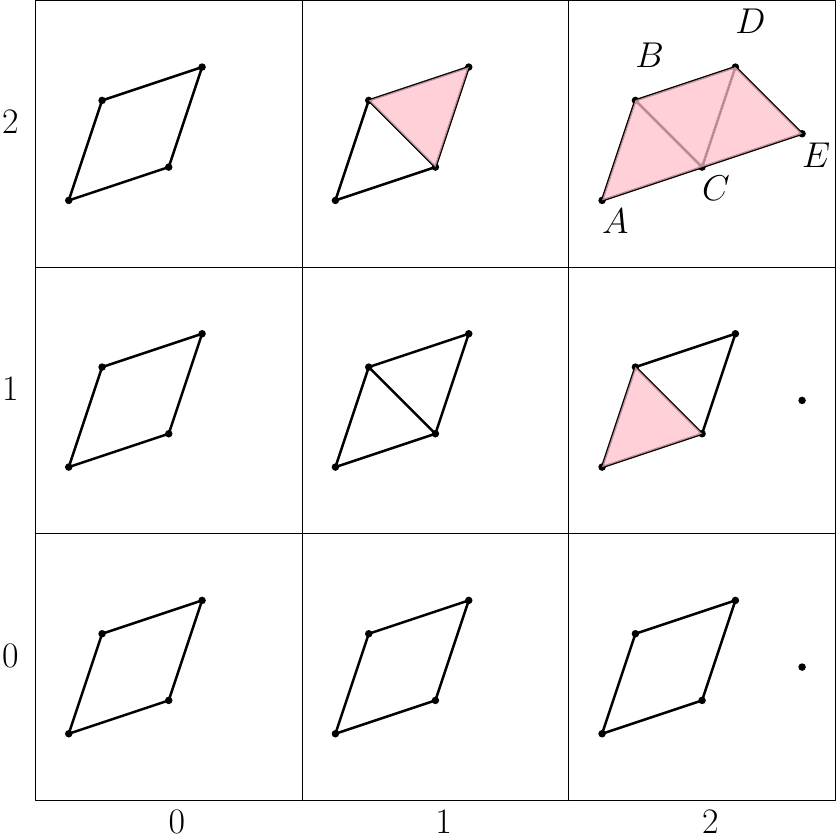}
    \caption{A $G_{3,3}$-filtration $\calF$}
    \label{fig:3x3-bifiltration}
\end{figure}

Now we consider an interval: $I = [\{(0, 2),(1, 1)\}, \{(1, 2), (2, 1)\}]$. Thus more visually, $\udim V_{I} = \sbmat{110\\011\\000}$. Three block matrices $\bfg_{1}$, $\bfg_{2}$, $\bfg_{3}$ of the multiplicity matrix for $I$ are given by
\begin{equation}
    \begin{blockarray}{ccc}
        & \scalebox{0.8}{$(0, 2)$} & \scalebox{0.8}{$(1, 1)$} \\
        \begin{block}{c[cc]}
            \scalebox{0.8}{$(1, 2)$} & 1 & -1\\
            \scalebox{0.8}{$(2, 2)$} & 1 & 0\\
        \end{block}
    \end{blockarray},\
    \begin{blockarray}{cccc}
        & \scalebox{0.8}{$(0, 1)$} & \scalebox{0.8}{$(2, 0)$} & \scalebox{0.8}{$(1, 1)$}\\
        \begin{block}{c[ccc]}
            \scalebox{0.8}{$(1, 2)$} & 0 & 0 & 1\\
            \scalebox{0.8}{$(2, 1)$} & 1 & 1 & -1\\
        \end{block}
    \end{blockarray},\
    \begin{blockarray}{ccc}
        & \scalebox{0.8}{$(0, 2)$} & \scalebox{0.8}{$(1, 1)$} \\
        \begin{block}{c[cc]}
            \scalebox{0.8}{$(1, 2)$} & 1 & 0\\
            \scalebox{0.8}{$(2, 1)$} & 0 & 0\\
        \end{block}
    \end{blockarray},
\end{equation}
respectively. By~\cref{prp:fast-way-block-matrices}, we have
$$
\begin{aligned}
\sfP'(\bfg_1)\big((0,0),(1,1)\big) &= \sbmat{1 & -1\\1 & 0} \ds \sbmat{0 & -1\\0 & 0},\\
\sfP'(\bfg_2)\big((0,0),(1,1)\big) &= \sbmat{0 & 0 & 1\\1 & 1 & -1} \ds \sbmat{0 & 0 & 1\\0 & 0 & -1},\\
\sfP'(\bfg_3)\big((0,0),(1,1)\big) &= \sbmat{1 & 0\\0&0} \ds \sbmat{0 & 0\\0&0},
\end{aligned}
\ \
\begin{aligned}
\sfP'\big((1,2),(2,2)\big)(\al)&= \sbmat{1 & 0\\1 & 0} \ds \sbmat{1 & 0\\1 & 1},\\
\sfP'\big((1,2),(2,1)\big)(\al)&= \sbmat{1 & 0\\1 & 0} \ds \sbmat{0 & 0\\0 & 1}.
\end{aligned}
$$
By~noticing \cref{rmk:order-smnds}, put
\[
\sfM = \left[\begin{array}{c|c|cc}
\sbmat{1 & -1\\1 & 0} \ds \sbmat{0 & -1\\0 & 0} &
\bfzero & \sbmat{
1&0&&\\
&&1&0\\
1&0&&\\
&&1&1
}& \bfzero\\
\hline
\sbmat{1 & 0\\0&0} \ds \sbmat{0 & 0\\0&0}
&\sbmat{0 & 0 & 1\\1 & 1 & -1} \ds \sbmat{0 & 0 & 1\\0 & 0 & -1}
&\bfzero & \sbmat{
1&0&&\\
&&0&0\\
1&0&&\\
&&0&1
}
\end{array}\right].
\]
The ranks of $\sfM$ and another matrix in~\cref{thm:formula-pp(M)-icp(ass)} are both $8$. Therefore, $d_{M}(V_{I}) = 8 - 8 = 0$.

Indeed, recalling~\cref{prp:criterion-int-sum,rmk:criterion-by-proj}. Since $\src(I) = \{(0, 2),(1, 1)\}$ is not a subset of $\{(0,0),(1,1)\}$ (the set of row indices of $\sfP(\al)$, which is exactly $\supp(\top M)$), we have that $I\notin \crt_0(M)$ and thereby $d_{M}(V_{I}) = 0$. This therefore provides a twofold confirmation of the validity of formula~\eqref{eq:formula-with-pp} and the correctness of the criterion stated in~\cref{prp:criterion-int-sum}.
\end{exm}

We remark that the presentation matrix of a bifiltration can be obtained by using RIVET developed by~\cite{lesnickComputingMinimalPresentations2022}, or mpfree developed by~\cite{Kerber-Rolle2021,fugacciCompression2parameterPersistent2023}.

\subsection{The formula by injective copresentations}

For an $M \in \mod A$, there also exists a way to compute an injective
copresentation $Q\up$ of $M$ in the filtration level.
Therefore, we next give a formula of $d_M(V_I)$ for any $I \in \bbI$ using $Q\up$.
For this purpose, we first give the dual statement of
Lemma \ref{lem:dim-Hom-coker}.

\begin{lem}
Let $C, M \in \mod A$.
Assume that $M$ has an injective copresentation
\begin{equation}
\label{eq:inj-copres-M}
0 \to M \ya{\si} \sfQ(x') \ya{\sfQ(\al')} \sfQ(y')
\end{equation}
for some $\al' \colon y' \to x'$ in $\Ds\! A$, where
$x' = (x'_i)_{i \in [m']}$ for some $m'$.
Then we have
\begin{equation}
\label{eq:dimHom-dual-formula}
\dim \Hom_{A}(C, M) = \sum_{i\in [m']} \dim C(x'_i) - \rank C(\al').
\end{equation}
\end{lem}

\begin{proof}
From \eqref{eq:inj-copres-M}, we have the following projective presentation
of $DM$:
$$
\sfP'(y') \ya{\sfP'(\al')} \sfP'(x') \ya{D\si} DM \to 0.
$$
Then by Lemma \ref{lem:dim-Hom-coker}, we have
$$
\dim \Hom_{A\op}(DM, DC) = \sum_{i \in [m']} \dim (DC)(x'_i) - \rank\, (DC)(\al'),
$$
which shows \eqref{eq:dimHom-dual-formula} because $\dim \Hom_{A\op}(DM, DC) = \dim \Hom_{A}(C, M)$,\,
$\dim (DC)(x'_i) = \dim C(x'_i)$ for all $i \in [m]$, and $\dim (DC)(\al') = \dim C(\al')$.
\end{proof}

For convenience, $\sfQ(\al')$ in~\eqref{eq:inj-copres-M} is called a \emph{copresentation matrix} of $M$.

\begin{thm}
\label{thm:multi-icp}
Let $M \in \mod A$ and $I$ an interval of $\bfP$, and assume
that $M$ has an injective copresentation \eqref{eq:inj-copres-M}
for some morphism
$\al' \colon y' \to x'$ in $\Ds\!A$,
where we set $x'\coloneqq (x'_i)_{i \in [m']},\, y'\coloneqq (y'_j)_{j \in [n']}$.

{\rm Case 1:} $V_I$ is non-injective.  In this case, let
$$
0 \to V_I \to E \to \ta\inv V_I \to 0
$$
be an almost split sequence starting from $V_I$.
Then we have the following formula:
\begin{equation}
\label{eq:formula-using-icp}
d_M(V_I) =
\rank E(\al') - \rank V_I(\al') - \rank\, (\ta\inv V_I)(\al').
\end{equation}

{\rm Case 2:} $V_I$ is injective.  In this case, $I = \dset a$ with $a = \max I$
and $V_I \iso Q_a$.
We may set $\al' = [\al'_{ji}]_{(j,i) \in [n']\times [m']}$, where
$\al'_{ji} = a'_{ji}p_{y_j,x_i}$ for some $a'_{ji} \in \k$ and
$\al'_{ji} = a'_{ji} = 0$ unless $x_i \le y_j$ for all $(j,i) \in [n'] \times [m']$.
We set $n'_{M,I}\coloneqq \#\{i \in [m'] \mid x_i = x\}$.
Then we have the following formula.
\begin{align}
\label{eq:formula-using-icp-inj}
&\quad d_M(V_I) =
\rank (V_{\dset b}/V_{\{b\}})(\al')  - \rank V_{\dset b}(\al')
+ \sum_{i \in [m]}\dim V_{\{b\}}(x_i) \nonumber\\
& =\rank [\de_{(a > x_i, y_j)}a'_{ji}]_{(j,i)\in [n']\times [m']} - \rank [\de_{(a \ge x_i, y_j)}a'_{ji}]_{(j,i)\in [n']\times [m']} + n'_{M,I}.
\end{align}
\end{thm}

\begin{proof}
This is proved in the same way as the proof of Theorem \ref{thm:multi-pp}.
\end{proof}

\begin{rmk}
\label{rmk:argument-inj-copres-case}
We have a statement corresponding to
Proposition \ref{prp:min-inj-p-VI-gen}
(resp.\ Proposition \ref{prp:min-inj-p-VI})
in this case, which is given by formulas
\eqref{eq:pp-V_I},
\eqref{eq:pp-E+P2} and
\eqref{eq:proj-resol-ta-invV_I+P-nsrccase-lem w/o conditions}
(resp.\
\eqref{eq:mpp-VI},
\eqref{eq:pp-E} and \eqref{min-pp-ta-inv-VI}).
\end{rmk}

By using \eqref{eq:formula-using-icp}, Proposition \ref{prp:rank-using-pp} and Remark \ref{rmk:argument-inj-copres-case}, we can prove the following:

\begin{thm}
\label{thm:formula-icp(M)-pp(ass)}
Let $I$ be an interval of $\bfP$, and $M \in \mod A$ with
an injective copresentation
$$
0 \to M \ya{\si} \sfQ(x') \ya{\sfQ(\al')} \sfQ(y')
$$
for some morphism $\al' \colon y' \to x'$ in $\Ds\!A$.
Keep the notations introduced in~{\rm \cref{prp:min-inj-p-VI-gen}}.  Then we have the following formula for $d_M(V_I)$:
$$
\begin{aligned}
d_M(V_I)
&= \rank \left[\begin{array}{c|c|cc}
 \sfP(\bfg_1)(x') & \bfzero & \sfP(\src(I))(\al') & \bfzero\\
 \hline
\sfP(\bfg_3)(x')
& \sfP(\bfg_2)(x') &
\bfzero & \sfP(\snk(\ddset I) \ds \snk_1(I))(\al')
 \end{array}\right]\\
 &\quad - \rank \left[\begin{array}{c|c|cc}
 \sfP(\bfg_1)(x') & \bfzero & \sfP(\src(I))(\al') & \bfzero\\
 \hline
\bfzero
& \sfP(\bfg_2)(x') &
\bfzero & \sfP(\snk(\ddset I) \ds \snk_1(I))(\al')
 \end{array}\right].
 \end{aligned}
$$
Note that
for the 2D-grid case, we can replace $\snk_1(I)$ with $\snk_1^\circ(I)$.
\end{thm}

\begin{proof}
{\rm Case 1.} $V_I$ is non-injective.

The assertion is proved in a way similar to Case 1 in Theorem \ref{thm:formula-pp(M)-icp(ass)}.

{\rm Case 2.} $V_I$ is injective.

Note in this case that
since $\snk(I) = \{b\}$ has only one element,
we have $C_2 \snk(I) = \emptyset$.
Thus $\snk_1(I) = \emptyset = \snk(\ddset I)$, and
$\bfg_2$ is an empty matrix.

By \eqref{eq:formula-using-icp-inj}, we have
$$
d_M(V_I) = \rank (V_I/\soc V_I)(\al') - \rank V_I(\al') + \sum_{i\in [m]}\dim(\soc V_I)(x_i).
$$
To compute the first two terms, we apply Proposition \ref{prp:rank-using-pp}
to the following projective presentations
of $V_I$ and $V_I/\soc V_I$ given in Theorem \ref{thm:dMV-VI-inj-general}:
$$
\begin{aligned}
\sfP(\src_1(I) \ds \src(\uuset I)) \ya{\sfP(\bfg_1)} \sfP(\src(I)) &\to V_I \to 0\\
\sfP(\src_1(I) \ds \src(\uuset I) \ds b) \ya{\sfP(\sbmat{\bfg_1\\\bfg_3})}
\sfP(\src(I)) &\to V_I/\soc V_I \to 0.
\end{aligned}
$$
Then we obtain
$$
\begin{aligned}
\rank (V_I/\soc V_I)(\al')
&=
\dim (V_I/\soc V_I)(x') - \dim \sfP(\src(I))(x')\\
&\quad + \rank \sbmat{\sfP\sbmat{\bfg_1\\\bfg_3}(x'), \ \sfP(\src(I))(\al')}\\
\rank V_I(\al')
&=
\dim V_I(x') - \dim \sfP(\src(I))(x')\\
&\quad + \rank \sbmat{\sfP(\bfg_1)(x'),\ \sfP(\src(I))(\al')}\\
\sum_{i\in [m]}\dim(\soc V_I)(x_i)
&= \dim (\soc V_I)(x').
\end{aligned}
$$
Altogether, we have
$$
d_M(V_I) = \rank \sbmat{\sfP\sbmat{\bfg_1\\\bfg_3}(x'), \ \sfP(\src(I))(\al')}
- \rank \sbmat{\sfP(\bfg_1)(x'),\ \sfP(\src(I))(\al')},
$$
which coincides with the formula above in the case
where $\snk_1(I) = \emptyset = \snk(\ddset I)$, and
$\bfg_2$ is an empty matrix.
\end{proof}

When applying Theorem \ref{thm:formula-icp(M)-pp(ass)},
the same remark as Remark \ref{rmk:order-smnds} on the order of direct summands should be kept in mind.

\section{Examples}
\label{section-5}

We first provide an example to explain how to use the essential-cover technique to compute interval multiplicities from the filtration level.

\subsection{The case of 2D-grids}

\begin{exm}
\label{exm:ess-cover-filt}
Let $\bfP = G_{5,2}$ and let $\calF$ be a $\bfP$-filtration shown in~\cref{fig:Filtration_ess_cov_int_mult_eq_1}, page~\pageref{fig:Filtration_ess_cov_int_mult_eq_1}. Set $M\coloneqq H_{1}(\blank;\bbZ/2\bbZ)\circ \calF$.
We compute the interval multiplicity of $V_{I}$ in $M$ where
\[
I\coloneqq\ \
\begin{tikzcd}[ampersand replacement=\&]
	{(1,2)} \& {(2,2)} \& {(3,2)} \\
	\& {(2,1)} \& {(3,1)} \& {(4,1)} \& {(5,1)}
	\arrow[from=1-1, to=1-2]
	\arrow[from=1-2, to=1-3]
	\arrow[from=2-2, to=1-2]
	\arrow[from=2-2, to=2-3]
	\arrow[from=2-3, to=1-3]
	\arrow[from=2-3, to=2-4]
	\arrow[from=2-4, to=2-5]
\end{tikzcd}.
\]
To make the notations of morphisms in $\k[\bfP]$ short, we denote each vertex $(i,j)$ of $\bfP$ simply by $ij$.
By suitable choice maps, we have the following multiplicity matrix for $I$:
$$
\bfg = \bmat{
p_{22,21} & -p_{22,12} & 0 & 0\\
0 & p_{42,12} & 0 & 0\\
0 & 0& p_{51, 11} & p_{51, 31}\\
p_{32,21} & 0& 0& -p_{32, 31}
}.
$$
By looking at the entries of $\bfg$, we define a
subposet $Z$ of $\bfP$ by
\[Z\coloneqq\ \
\begin{tikzcd}[ampersand replacement=\&]
	{(1,2)} \& {(2,2)} \& {(3,2)} \& {(4,2)} \\
	{(1,1)} \& {(2,1)} \& {(3,1)} \&\& {(5,1)}
	\arrow[from=1-1, to=1-2]
	\arrow[curve={height=-28pt}, from=1-1, to=1-4]
	\arrow[curve={height=28pt}, from=2-1, to=2-5]
	\arrow[from=2-2, to=1-2]
	\arrow[from=2-2, to=1-3]
	\arrow[from=2-3, to=1-3]
	\arrow[from=2-3, to=2-5]
\end{tikzcd},
\]
which  is a finite zigzag poset.
Then by \cref{dfn:ess-cov2}, the inclusion map $\ze\colon Z\hookrightarrow \bfP$ essentially covers $I$.
Note that $R(V_I)$ turns out to be the interval module $V_{I'}$ over $\k[Z]$,
where $I'$ is given by
	\[
	\begin{tikzcd}[ampersand replacement=\&]
	{(1,2)} \& {(2,2)} \& {(2,3)} \\
	\& {(2,1)} \& {(3,1)} \&\& {(5,1)}
	\arrow[from=1-1, to=1-2]
	\arrow[from=2-2, to=1-2]
	\arrow[from=2-2, to=1-3]
	\arrow[from=2-3, to=1-3]
	\arrow[from=2-3, to=2-5]
\end{tikzcd},
\]
or equivalently, the dimension vector of $V_{I'}$ is $\sbmat{1110\\0111}$.
Then by \cref{thm:ess-cover-equality},
we have $d_M(V_I) = d_{R(M)}(V_{I'})$.
Thus the problem is reduced to the computation of $d_{R(M)}(V_{I'})$,
the multiplicity of $V_{I'}$ in the barcodes of the zigzag persistence module
$R(M) = H_{1}(\blank;\bbZ/2\bbZ)\circ \calF \circ \ze$, where
$\calF \circ \ze =:\calF'$ is a $Z$-filtration shown in \cref{fig:Filtration_ess_cov_int_mult_zigzag_eq_1}, page~\pageref{fig:Filtration_ess_cov_int_mult_zigzag_eq_1}, in which arrows represent inclusions.
This kind of problem is already solved in the filtration level (for instance, see \cite{deyFastComputationZigzag2022,milosavljevicZigzagPersistentHomology2011,carlssonZigzagPersistentHomology2009}).

\begin{figure}[htbp]
  \centering
  \includegraphics[width=0.9\textwidth]{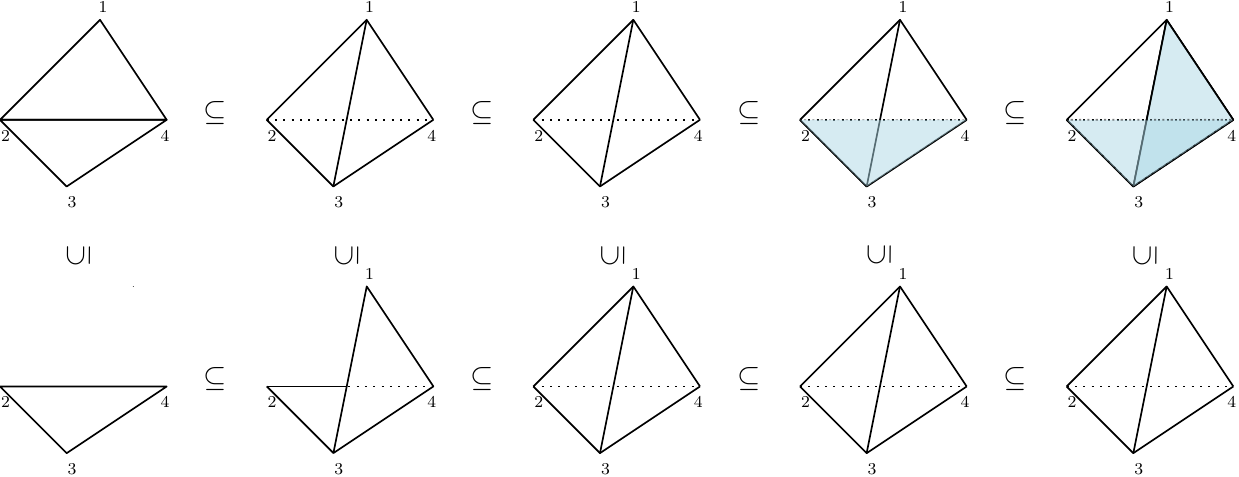}
  \caption{A $G_{5,2}$-filtration $\calF$}
  \label{fig:Filtration_ess_cov_int_mult_eq_1}
\end{figure}

\begin{figure}[htbp]
  \centering
  \includegraphics[width=0.9\textwidth]{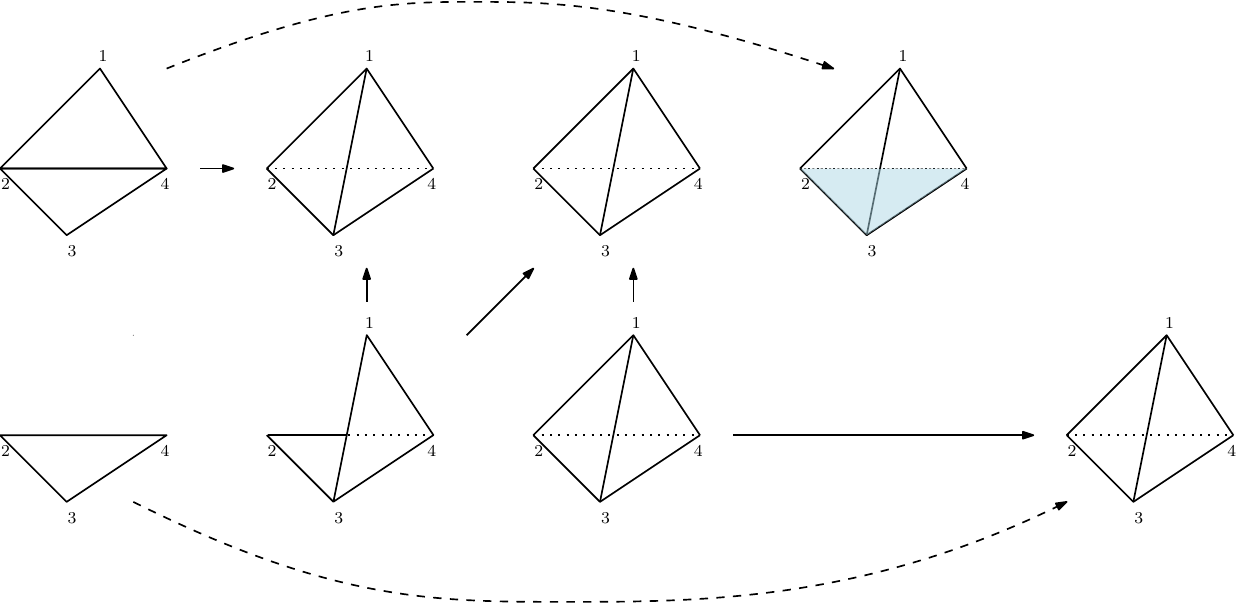}
  \caption{A $Z$-filtration $\calF'$ where $\ze \colon Z \to \bfP$ essentially covers $I$}
  \label{fig:Filtration_ess_cov_int_mult_zigzag_eq_1}
\end{figure}
\end{exm}

\begin{rmk}
In Example \ref{exm:ess-cover-filt} above,
we see that $d_{R(M)}(V_{I'}) = 0$ as follows.
Assume that it is nonzero.
Then $M'\coloneqq R(M)$ has a direct summand $X$ such that there is an isomorphism
$\al \colon V_{I'} \to X$,
say $M' = X \ds Y$.
For each vertex $(i,j)$ of $Z$,
let $\{v_{ij}\}$ be the standard basis of $V_{I'}(i,j)$,
and set $a_{ij}\coloneqq \al(v_{ij})$, and then
$\{a_{ij}\}$ becomes a basis of $X(i,j)$.
In Fig.\ \ref{fig:Filtration_ess_cov_int_mult_zigzag_eq_1},
we denote the 1-cycle $\{a,b\}+\{b,c\}+\{c,a\}$ by $(abc)$ for short.
Then $M'(1,2)$ and $M'(2,1)$ have bases $\{(124), (234)\}$ and $\{(134), (234)\}$,
respectively.
Then we can write $X(1,2) = \k a_{12}$ and $X(2,1) = \k a_{21}$ with
$a_{12} = s(124) + t(234)$, $a_{21} = u(134) + v(234)$
for some $s, t, u, v \in \k$.
Since $X(p_{22,12})(a_{12}) = a_{22} = X(p_{22,21})(a_{21})$ by construction, we have
$s(124) + t(234) = u(134) + v(234)$ in $M'(2,2)$ that has a basis $\{(124), (134), (234)\}$,
which shows that $t = v \ne 0$ and $s = u = 0$.
Hence $a_{12} = t(234)$ and $a_{21} = t(234)$, and we have
$X(5,1) = \k(234)$.
Since $X(1,1) = 0$, we have $(234) \in M'(1,1) = Y(1,1)$,
and hence $(234) = Y(p_{51,11})(234) \in Y(5,1)$.
Therefore, $X(5,1) \cap Y(5,1) \ni (234) \ne 0$, a contradiction.
In fact, the decomposition of $R(M)$ is given by
$$
R(M) = \sbmat{1101\\0000} \ds \sbmat{0010\\0011} \ds \sbmat{0001\\0000}\ds \sbmat{0110\\0111}
\ds \sbmat{1110\\1111},
$$
where all summands are interval modules presented by their dimension vectors.

However, we remark that the interval rank of $I$ under the total compression system defined in~\cite{asashiba2024interval}, or equivalently, the generalized rank of $I$ defined in~\cite{kim2021generalized} is at least $1$ (actually equal to 1) because
the restriction $R_I(M)$ of $M$ to $I$ has a direct summand $X$
isomorphic to $V_I$ with spaces $X(i)= \k(234)$ for all $i \in I$.
In summary, the ``generalized'' rank of interval $I\subseteq \bfP$ only need information inside of $I$, while its multiplicity need extra information outside of $I$, causing their distinctions.
\end{rmk}

\begin{exm}
\label{exm:ess-cover-filt-2}
	Let $\bfP = G_{6,2}$ and consider the following interval of $\bfP$:
\[
I\coloneqq\ \
\begin{tikzcd}[ampersand replacement=\&]
	{(2,2)} \& {(3,2)} \& {(4,2)} \\
	\& {(3,1)} \& {(4,1)} \& {(5,1)}
	\arrow[from=1-1, to=1-2]
	\arrow[from=1-2, to=1-3]
	\arrow[from=2-2, to=1-2]
	\arrow[from=2-2, to=2-3]
	\arrow[from=2-3, to=1-3]
	\arrow[from=2-3, to=2-4]
\end{tikzcd}.
\]
We compute the interval multiplicity of $V_{I}$. For brevity we set $a_1\coloneqq (3,1)$, $a_2\coloneqq(2,2)$, $b_1\coloneqq (5,1)$, $b_2\coloneqq(4,2)$ by adopting~\cref{ntn:notation-in-2D-grid}. Then $a_{12}=a_1\vee a_2 = (3,2)$, $b_{12}=b_1\wedge b_2 = (4,1)$, $\src(\uuset I) = \{a'_1,a'_2\} = \{(6,1),(5,2)\}$, and $\snk(\ddset I) = \{b'_1,b'_2\} = \{(2,1),(1,2)\}$.

By~\cref{thm:final-form-add-hull}, there exists a multiplicity matrix $\bfg = \bmat{\bfg_1 & \mathbf{0}\\ \bfg_3 & \bfg_2}$ for $I$.
Here $\bfg$ may be taken as the form:
\[
\bfg\coloneqq\left[
\begin{array}{c|c}
\bfg_1 & \bfzero\\
\hline
\bfg_3 & \bfg_2
\end{array}
\right] =
\left[
\begin{array}{cc|ccc}
p_{a_{12}, a_{1}} & -p_{a_{12}, a_{2}} & \bfzero & \bfzero & \bfzero\\
p_{x_{1}, a_{1}} & \bfzero & \bfzero & \bfzero & \bfzero\\
\bfzero & p_{x_{2}, a_{2}} & \bfzero & \bfzero & \bfzero\\
\hline
\bfzero & \bfzero & p_{b_{1}, y_{1}} & \bfzero & p_{b_{1}, b_{12}}\\
p_{b_{2}, a_{1}} & \bfzero & \bfzero & p_{b_{2}, y_{2}} & -p_{b_{2}, b_{12}}
\end{array}
\right].
\]
Hence if we take the (not full) subposet $Z$ of $\bfP$ given by the right-hand side of \cref{fig:Exm_G_6_2}, then $Z$ together with the usual inclusion map $\ze\colon Z\hookrightarrow \bfP$ essentially covers $I$. We remark that in this example, $Z$ itself is not the zigzag poset, but the Hasse quiver of $Z$ is a directed tree formed by connecting two zigzag posets (shown in green and blue colors in~\cref{fig:Exm_G_6_2}).

\begin{figure}[ht]
  \centering
  \includegraphics[width=\textwidth]{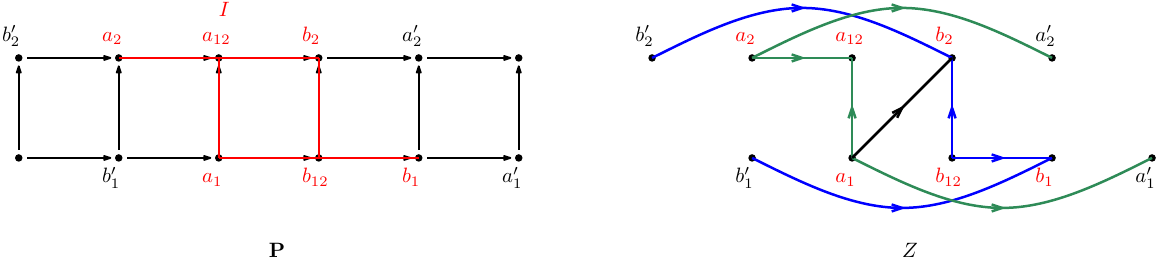}
  \caption{An illustration of essential cover of $I$}
  \label{fig:Exm_G_6_2}
\end{figure}

\end{exm}

In the subsequent examples, we demonstrate the computation in other types of posets.

\subsection{The case of Dynkin type \texorpdfstring{$\mathbb{D}$}{D}}

\begin{exm}
\label{exm:int-rk_D_4_cases}
	Consider a poset $\bfP$ having the following Hasse quiver:
\[
\begin{tikzcd}[ampersand replacement=\&]
	\& 4 \\
	1 \& 2 \& 3
	\arrow[from=1-2, to=2-2]
	\arrow[from=2-1, to=2-2]
	\arrow[from=2-2, to=2-3]
\end{tikzcd}.
\]
Let $M_1$ and $M_2$ be in $\mod \k[\bfP]$ given by
\[M_1\coloneqq
\begin{tikzcd}[ampersand replacement=\&]
	\& \k \\
	\k \& {\k^2} \& \k
	\arrow["\sbmat{1 \\ 1}", from=1-2, to=2-2]
	\arrow["\sbmat{1 \\ 1}"', from=2-1, to=2-2]
	\arrow["\sbmat{1,\, 0}"', from=2-2, to=2-3]
\end{tikzcd}
\quad \text{and}\quad
M_2\coloneqq
\begin{tikzcd}[ampersand replacement=\&]
	\& \k \\
	\k \& {\k^2} \& \k
	\arrow["\sbmat{1 \\ 0}", from=1-2, to=2-2]
	\arrow["\sbmat{0 \\ 1}"', from=2-1, to=2-2]
	\arrow["\sbmat{1,\, 1}"', from=2-2, to=2-3]
\end{tikzcd}.
\]
We compute the interval multiplicity of $V_{\bfP}$ in both two modules. In this case, $I = \bfP$.

By~\cref{thm:final-form-add-hull}, there exists a multiplicity matrix $\bfg$ for $I$ having the form
\[
\bfg\coloneqq\left[
\begin{array}{c|c}
\bfg_1 & \bfzero\\
\hline
\bfg_3 & \bfg_2
\end{array}
\right] =
\left[
\begin{array}{c|c}
p_{2, 1} & -p_{2, 4}\\
\hline
p_{3, 1} & \bfzero
\end{array}
\right].
\]
Hence, it is now clear that if we take the (not full) subposet $Z$ of $\bfP$ given by
\[Z\coloneqq
\begin{tikzcd}[ampersand replacement=\&]
	\& 4 \\
	1 \& 2 \& 3
	\arrow[from=1-2, to=2-2]
	\arrow[from=2-1, to=2-2]
	\arrow[curve={height=20pt}, from=2-1, to=2-3]
\end{tikzcd},
\]
together with the usual inclusion map $\ze\colon Z\hookrightarrow \bfP$, then $\ze$ essentially covers $\bfP$. By~\cref{thm:ess-cover-equality} it suffices to compute $\bar{d}_{R(M_{j})}(R(V_I)) = d_{R(M_{j})}(V_{Z})$ for $j\in \{1, 2\}$. Now, because
\[
R(M_1)=
\begin{tikzcd}[ampersand replacement=\&]
	\& \k \\
	\k \& \k^2 \& \k
	\arrow["\sbmat{1 \\ 1}", from=1-2, to=2-2]
	\arrow["\sbmat{1 \\ 1}", from=2-1, to=2-2]
	\arrow["1"', curve={height=20pt}, from=2-1, to=2-3]
\end{tikzcd}
\iso
\begin{tikzcd}[ampersand replacement=\&]
	\& \k \\
	\k \& \k \& \k
	\arrow["1", from=1-2, to=2-2]
	\arrow["1", from=2-1, to=2-2]
	\arrow["1"', curve={height=20pt}, from=2-1, to=2-3]
\end{tikzcd}
\ds
\begin{tikzcd}[ampersand replacement=\&]
	\& 0 \\
	0 \& \k \& 0
	\arrow[from=1-2, to=2-2]
	\arrow[from=2-1, to=2-2]
	\arrow[curve={height=20pt}, from=2-1, to=2-3]
\end{tikzcd}
\]
and
\[
R(M_2)=
\begin{tikzcd}[ampersand replacement=\&]
	\& \k \\
	\k \& \k^2 \& \k
	\arrow["\sbmat{1 \\ 0}", from=1-2, to=2-2]
	\arrow["\sbmat{0 \\ 1}", from=2-1, to=2-2]
	\arrow["1"', curve={height=20pt}, from=2-1, to=2-3]
\end{tikzcd}
\iso
\begin{tikzcd}[ampersand replacement=\&]
	\& 0 \\
	\k \& \k \& \k
	\arrow[from=1-2, to=2-2]
	\arrow["1", from=2-1, to=2-2]
	\arrow["1"', curve={height=20pt}, from=2-1, to=2-3]
\end{tikzcd}
\ds
\begin{tikzcd}[ampersand replacement=\&]
	\& \k \\
	0 \& \k \& 0
	\arrow["1", from=1-2, to=2-2]
	\arrow[from=2-1, to=2-2]
	\arrow[curve={height=20pt}, from=2-1, to=2-3]
\end{tikzcd},
\]
we conclude that $d_{M_1}(V_{\bfP}) = 1$, but $d_{M_2}(V_{\bfP}) = 0$.
\end{exm}

\begin{exm}
\label{exm:int-rk_D_4_cases_2nd}
	Consider a poset $\bfP$ having the following Hasse quiver:
\[
\begin{tikzcd}[ampersand replacement=\&]
	\& 4 \\
	1 \& 2 \& 3
	\arrow[from=2-2, to=1-2]
	\arrow[from=2-2, to=2-1]
	\arrow[from=2-2, to=2-3]
\end{tikzcd}.
\]
Let $M\in \mod \k[\bfP]$ given by
\[M\coloneqq
\begin{tikzcd}[ampersand replacement=\&]
	\& \k \\
	\k \& {\k^3} \& {\k^2}
	\arrow["\sbmat{1,\, 0,\, 0}", from=2-2, to=1-2]
	\arrow["\sbmat{1,\, 0,\, 0}", from=2-2, to=2-1]
	\arrow["\sbmat{1 & 0 & 0\\0 & 1 & 1}"', from=2-2, to=2-3]
\end{tikzcd}.
\]
We compute the interval multiplicity of $V_{\bfP}$ in $M$. In this case, $I = \bfP$.

By~\cref{thm:final-form-add-hull}, there exists a multiplicity matrix $\bfg$ for $I$
of the form
\begin{equation}
\label{eq:original-morphism}
\bfg\coloneqq\left[
\begin{array}{c|c}
\bfg_3 & \bfg_2
\end{array}
\right] =
\left[
\begin{array}{c|ccc}
\bfzero & p_{1, 2} & p_{1, 2} & \bfzero\\
p_{3, 2} & -p_{3, 2} & \bfzero & p_{3, 2}\\
\bfzero & \bfzero & -p_{4, 2} & -p_{4, 2}
\end{array}
\right].
\end{equation}
Notice that the last column of $\bfg_2$ is the linear combination of its first two columns, hence we may take another morphism $\tilde{\bfg}$ in $\Ds\k[\bfP]$ given by
\[
\tilde{\bfg}\coloneqq \left[
\begin{array}{c|c}
\bfg_3 & \tilde{\bfg}_2
\end{array}
\right] =
\left[
\begin{array}{c|cc}
\bfzero & p_{1, 2} & p_{1, 2}\\
p_{3, 2} & -p_{3, 2} & \bfzero\\
\bfzero & \bfzero & -p_{4, 2}
\end{array}
\right],
\]
such that $\rank M(\bfg) - \rank M(\bfg_2) = \rank M(\tilde{\bfg}) - \rank M(\tilde{\bfg}_2)$. This shows that the new morphism $\tilde{\bfg}$ is also a multiplicity matrix for $I$.

Now, let us take the following finite zigzag poset
\[Z\coloneqq
\begin{tikzcd}[ampersand replacement=\&]
	2 \&\& {2'} \&\& {2''} \\
	\& 3 \&\& 1 \&\& 4
	\arrow[from=1-1, to=2-2]
	\arrow[from=1-3, to=2-2]
	\arrow[from=1-3, to=2-4]
	\arrow[from=1-5, to=2-4]
	\arrow[from=1-5, to=2-6]
\end{tikzcd}
\]
and define the order-preserving map $\ze\colon Z\to \bfP$ by
\[
\ze(x)\coloneqq
\begin{cases}
	2, & \textnormal{if } x\in \{2, 2', 2''\},\\
	x, & \textnormal{if } x\in \{1, 3, 4\}.
\end{cases}
\]
Then $\ze$ essentially covers $\bfP$. Indeed, we have the following equality:
\[
\k[\ze]\left( \left[
\begin{array}{c|cc}
\bfzero & p_{1, 2'} & p_{1, 2''}\\
p_{3, 2} & -p_{3, 2'} & \bfzero\\
\bfzero & \bfzero & -p_{4, 2''}
\end{array}
\right] \right) = \left[
\begin{array}{c|cc}
\bfzero & p_{1, 2} & p_{1, 2}\\
p_{3, 2} & -p_{3, 2} & \bfzero\\
\bfzero & \bfzero & -p_{4, 2}
\end{array}
\right]
\]
Hence by~\cref{thm:ess-cover-equality} it suffices to compute $\bar{d}_{R(M)}(R(V_I)) = d_{R(M_{j})}(V_{Z})$. Now, because
\begin{align*}
	R(M) &=
\begin{tikzcd}[ampersand replacement=\&]
	{\k^3} \&\& {\k^3} \&\& {\k^3} \\
	\& {\k^2} \&\& \k \&\& \k
	\arrow["\sbmat{1 & 0 & 0\\0 & 1 & 1}"', from=1-1, to=2-2]
	\arrow["\sbmat{1 & 0 & 0\\0 & 1 & 1}"', from=1-3, to=2-2]
	\arrow["\sbmat{1,\, 0,\, 0}"', from=1-3, to=2-4]
	\arrow["\sbmat{1,\, 0,\, 0}", from=1-5, to=2-4]
	\arrow["\sbmat{1,\, 0,\, 0}", from=1-5, to=2-6]
\end{tikzcd}\\
	& \cong \sbmat{1 & \space & 1 & \space & 1 & \space \\ \space & 1 & \space & 1 & \space & 1} \ds \sbmat{1 & \space & 1 & \space & 0 & \space \\ \space & 1 & \space & 0 & \space & 0} \ds \sbmat{1 & \space & 0 & \space & 0 & \space \\ \space & 0 & \space & 0 & \space & 0} \ds \sbmat{0 & \space & 1 & \space & 0 & \space \\ \space & 0 & \space & 0 & \space & 0}\ds \sbmat{0 & \space & 0 & \space & 1 & \space \\ \space & 0 & \space & 0 & \space & 0}^{2},
\end{align*}
we conclude that $d_{M}(V_{\bfP}) = 1$.
\end{exm}

We highlight that in the example above, finding a new multiplicity matrix $\tilde{\bfg}$ for $I$ is crucial for finding the zigzag poset $Z$. Indeed, we first notice that $\ze$ does not cover the original choice of $\bfg$ given in~\eqref{eq:original-morphism}. Next, it is straightforward to verify from~\cref{dfn:map-covers-morphism} that the following order-preserving map $\ze'\colon Z'\to \bfP$ covers both $\bfg$ and $\tilde{\bfg}$:
\[
Z'\coloneqq \begin{tikzcd}[ampersand replacement=\&,sep=small]
	\&\& {2'} \&\& {2''} \\
	2 \& 3 \&\& 1 \&\& 4 \\
	\\
	\&\&\& {2'''}
	\arrow[from=1-3, to=2-2]
	\arrow[from=1-3, to=2-4]
	\arrow[from=1-5, to=2-4]
	\arrow[from=1-5, to=2-6]
	\arrow[from=2-1, to=2-2]
	\arrow[from=4-4, to=2-2]
	\arrow[from=4-4, to=2-6]
\end{tikzcd},
\ \text{and }
\ze'(x)\coloneqq
\begin{cases}
	2, & \textnormal{if } x\in \{2, 2', 2'', 2'''\},\\
	x, & \textnormal{if } x\in \{1, 3, 4\}.
\end{cases}
\]
However, $Z'$ is not the poset of Dynkin type $\bbA$.

\subsection{The case of bipath posets}

Recently, Aoki--Escolar--Tada provided a complete classification of posets whose module category only consists of interval-decomposable modules \cite{aoki2023summand}. They showed that every persistence module in $\mod \k[\bfP]$ is interval-decomposable if and only if
$\bfP$ is either a poset of Dynkin type $\bbA$ or a bipath poset.
The former poset is well studied and applied in the one-parameter persistent homology, while the latter is not commonly considered in the multi-parameter setting. They subsequently investigated the so-called bipath persistent homology in \cite{aokiBipathPersistence2024}. To obtain the visualization of the bipath persistence diagram, decomposing the bipath persistent homology in each dimension becomes the central task. In this subsection, we would apply our formula to compute the multiplicities of interval modules in the bipath poset setting, and inspired by the obtained formulas, we propose an alternate way of computing the bipath persistence diagram in the practical TDA pipeline, without obtaining the bipath persistent homology.

To begin with, we review the definition of bipath poset.
Let $n,m\in \bbZ_{\ge 1}$.
Then the \emph{bipath} poset $B_{n,m}$ is defined to be a poset having the following Hasse quiver:
\[\begin{tikzcd}[ampersand replacement=\&]
	\& 1 \& 2 \& \cdots \& n \\
	{\hat{0}} \&\&\&\&\& {\hat{1}} \\
	\& {1'} \& {2'} \& \cdots \& {m'}
	\arrow[from=1-2, to=1-3]
	\arrow[from=1-3, to=1-4]
	\arrow[from=1-4, to=1-5]
	\arrow[from=1-5, to=2-6]
	\arrow[from=2-1, to=1-2]
	\arrow[from=2-1, to=3-2]
	\arrow[from=3-2, to=3-3]
	\arrow[from=3-3, to=3-4]
	\arrow[from=3-4, to=3-5]
	\arrow[from=3-5, to=2-6]
\end{tikzcd}.
\]
We set $[m]'\coloneqq \{1',\dots, m'\}$. The full subposet $\mathbf{U}\coloneqq \{\hat{0}, \hat{1}\}\sqcup [n]$ (resp. $\mathbf{D}\coloneqq \{\hat{0}, \hat{1}\}\sqcup [m]'$) is called the \emph{upper} (resp. \emph{lower}) \emph{path} of $B_{n,m}$. In the sequel, we would let the ambient poset $\bfP$ be $B_{n,m}$.
It is easy to check the interval $I$ of $B_{n,m}$ belongs to the following five types:
\begin{itemize}
    \item[(i)] $I = B_{n,m}$.
    \item[(ii)] $I\coloneqq [s,t]\coloneqq \{x\in B_{n,m}|s\leq x\leq t\}$ for some $s,t\in [n]$. We write $\bbIu$ to denote the set of all intervals having this type.
    \item[(iii)] $I\coloneqq [s',t']\coloneqq \{x'\in B_{n,m}|s'\leq x'\leq t'\}$ for some $s',t'\in [m]'$. We write $\bbId$ to denote the set of all intervals having this type.
    \item[(iv)] $I\coloneqq [\hat{0}, t] \cup [\hat{0}, t']$ for some $t\in \mathbf{U}\setminus \{\hat{1}\}$ and $t'\in \mathbf{D}\setminus \{\hat{1}\}$. We write $\bbIl$ to denote the set of all intervals having this type.
    \item[(v)] $I\coloneqq [s, \hat{1}] \cup [s', \hat{1}]$ for some $s\in \mathbf{U}\setminus \{\hat{0}\}$ and $s'\in \mathbf{D}\setminus \{\hat{0}\}$. We write $\bbIr$ to denote the set of all intervals having this type.
\end{itemize}

From now on, we provide the formula of $d_M(V_I)$ case by case.
Before doing so, we set the following conventions, which are used below:
\begin{equation}
\label{eq:conv+-1}
\begin{aligned}
&\hat{0} + 1\coloneqq 1,\,n + 1\coloneqq \hat{1}\\
&\hat{0} + 1'\coloneqq 1',\, t' + 1'\coloneqq (t+1)' \ (t' \in [m-1]'),\, m' + 1'\coloneqq \hat{1},\\
&1 - 1\coloneqq \hat{0},\, \hat{1} - 1\coloneqq n,\\
&1' - 1'\coloneqq \hat{0},\, s' - 1'\coloneqq (s -1)'\ (s' \in [2', m']),\, \hat{1} - 1'\coloneqq m',
\end{aligned}
\end{equation}

(i)
Let $I = B_{n,m}$. Since $V_I$ is injective in $\mod \k[\bfP]$, we apply \eqref{eq:formula for injective modules} here and obtain
\begin{equation}
\label{eq:formula-for-B}
    d_M(V_I) = \rank M_{\hat{1}, \hat{0}}
\end{equation}
by noticing both $\src_1(I)$ and $\src(\uuset I)$ are empty.

(ii)
Let $I\coloneqq [s,t]\coloneqq \{x\in B_{n,m}|s\leq x\leq t\}$ for some $s,t\in [n]$.
Since $V_I$ is non-injective in $\mod \k[\bfP]$, we apply \eqref{eq:formula-d_M(V_I)-general} here.
In this case, $\src(I)=\{s\}$, $\src_1(I)=\emptyset$, $\src(\uuset I) = \{t+1\}$. On the other hand, $\snk(I)=\{t\}$, $\snk_1(I)=\emptyset$, $\snk(\ddset I) = \{s-1\}$. Then we obtain
\begin{align}
d_M(V_I) & = \rank\left[
\begin{array}{c|c}
M_{t+1,s} & \bfzero\\
\hline
M_{t,s} & M_{t,s-1}
\end{array}
\right]
- \rank M_{t+1,s} - \rank M_{t,s-1}. \label{eq:formula-for-upper-00}
\end{align}

(iii)
Let $I\coloneqq [s',t']\coloneqq \{x'\in B_{n,m}|s'\leq x'\leq t'\}$ for some $s',t'\in [m]'$. This case is similar to the above case and thus we obtain
\begin{align}
d_M(V_I) & = \rank\left[
\begin{array}{c|c}
M_{t'+1',s'} & \bfzero\\
\hline
M_{t',s'} & M_{t',s'-1'}
\end{array}
\right]
- \rank M_{t'+1',s'} - \rank M_{t',s'-1'}. \label{eq:formula-for-lower}
\end{align}

(iv)
Let $I\coloneqq [\hat{0}, t] \cup [\hat{0}, t']$ for some $t\in \mathbf{U}\setminus \{\hat{1}\}$ and $t'\in \mathbf{D}\setminus \{\hat{1}\}$.

\noindent {\bf Case 1.}
Suppose $t = \hat{0}$ or $t' = \hat{0}$. Then $V_I$ is an injective module in $\mod \k[\bfP]$ with $\max(I)=\max\{t,t'\}$. Hence we apply \eqref{eq:formula for injective modules} here. In this case, $\src(I)=\{\hat{0}\}$, $\src_1(I)=\emptyset$, $\src(\uuset I) = \{t+1, t'+1'\}$. Then we obtain
\begin{align}
d_M(V_I) & = \rank\left[
\begin{array}{c}
M_{t+1,\hat{0}}\\
M_{t'+1',\hat{0}}\\
\hline
M_{\max\{t,t'\},\hat{0}}
\end{array}
\right]
- \rank
\begin{bmatrix}
M_{t+1,\hat{0}}\\
M_{t'+1',\hat{0}}
\end{bmatrix}. \label{eq:formula-for-left-case1}
\end{align}

\noindent {\bf Case 2.} Suppose $t=n$ and $t' = m$. Then $\src(I)=\{\hat{0}\}$, $\src_1(I)=\emptyset$, $\src(\uuset I) = \{\hat{1}\}$. On the other hand, $\snk(I)=\{t,t'\}$, $\snk_1(I) = \{\hat{0}\}$, $\snk(\ddset I)=\emptyset$. Then we obtain
\begin{align}
d_M(V_I) & = \rank\left[
\begin{array}{c|c}
M_{\hat{1},\hat{0}} & \bfzero\\
\hline
M_{t,\hat{0}} & M_{t,\hat{0}}\\
\bfzero & -M_{t',\hat{0}}
\end{array}
\right]
- \rank M_{\hat{1},\hat{0}} - \rank \begin{bmatrix}
M_{t,\hat{0}}\\
-M_{t',\hat{0}}
\end{bmatrix} \label{eq:non-inj-bipath-special}
\end{align}

\noindent {\bf Case 3.} Suppose $t,t'$ are not in the above two cases. Then $\src(I)=\{\hat{0}\}$, $\src_1(I)=\emptyset$, $\src(\uuset I) = \{t+1, t'+1'\}$. On the other hand, $\snk(I)=\{t,t'\}$, $\snk_1(I) = \{\hat{0}\}$, $\snk(\ddset I)=\emptyset$. Then we obtain
\begin{align}
\label{eq:non-inj-case-bipath}
d_M(V_I) & = \rank\left[
\begin{array}{c|c}
M_{t+1,\hat{0}} & \bfzero\\
M_{t'+1',\hat{0}} & \bfzero\\
\hline
M_{t,\hat{0}} & M_{t,\hat{0}}\\
\bfzero & -M_{t',\hat{0}}
\end{array}
\right]
- \rank \begin{bmatrix}
M_{t+1,\hat{0}} \\
M_{t'+1',\hat{0}}
\end{bmatrix}
- \rank \begin{bmatrix}
M_{t,\hat{0}}\\
-M_{t',\hat{0}}
\end{bmatrix}.
\end{align}

Notice that if we let $t = n$ and $t' = m$ in \eqref{eq:non-inj-case-bipath}, then the result coincides with \eqref{eq:non-inj-bipath-special}. Therefore, we can unify {\bf Case 2} and {\bf Case 3} and summarize the final result as follows:

\noindent {\bf Case 1*.} If $t = \hat{0}$ or $t' = \hat{0}$, then we have
\begin{align}
d_M(V_I) & = \rank\left[
\begin{array}{c}
M_{t+1,\hat{0}}\\
M_{t'+1',\hat{0}}\\
\hline
M_{\max\{t,t'\},\hat{0}}
\end{array}
\right]
- \rank
\begin{bmatrix}
M_{t+1,\hat{0}}\\
M_{t'+1',\hat{0}}
\end{bmatrix}. \label{eq:formula-for-left-1}
\end{align}

\noindent {\bf Case 2*.} If $t \neq \hat{0}$ and $t' \neq \hat{0}$, then
\begin{align}
d_M(V_I) & = \rank\left[
\begin{array}{c|c}
M_{t+1,\hat{0}} & \bfzero\\
M_{t'+1',\hat{0}} & \bfzero\\
\hline
M_{t,\hat{0}} & M_{t,\hat{0}}\\
\bfzero & -M_{t',\hat{0}}
\end{array}
\right]
- \rank \begin{bmatrix}
M_{t+1,\hat{0}} \\
M_{t'+1',\hat{0}}
\end{bmatrix}
- \rank \begin{bmatrix}
M_{t,\hat{0}}\\
-M_{t',\hat{0}}
\end{bmatrix}. \label{eq:formula-for-left-2}
\end{align}

(v)
Let $I\coloneqq [s, \hat{1}] \cup [s', \hat{1}]$ for some $s\in \mathbf{U}\setminus \{\hat{0}\}$ and $s'\in \mathbf{D}\setminus \{\hat{0}\}$.
This case is just the dual of case (iv) above, and we analogously obtain the following.

\noindent {\bf Case $1'$.} If $s = \hat{1}$ or $s' = \hat{1}$, then we have
\begin{align}
d_M(V_I) & = \rank
\begin{bmatrix}
M_{\hat{1}, \min\{s,s'\}}\ \vline\ M_{\hat{1}, s-1} & M_{\hat{1}, s'-1'}
\end{bmatrix}
- \rank
\begin{bmatrix}
M_{\hat{1}, s-1} & M_{\hat{1}, s'-1'}
\end{bmatrix}. \label{eq:formula-for-right-1}
\end{align}

\noindent {\bf Case $2'$.} If $s \neq \hat{1}$ and $s' \neq \hat{1}$, then we have
\begin{align}
d_M(V_I) & = \rank\left[
\begin{array}{cc|cc}
M_{\hat{1},s} & M_{\hat{1},s'} & \bfzero & \bfzero\\
\hline
M_{\hat{1},s} & \bfzero & M_{\hat{1},s-1} & M_{\hat{1},s'-1'}
\end{array}
\right]
- \rank \begin{bmatrix}
M_{\hat{1},s} & M_{\hat{1},s'}
\end{bmatrix}\nonumber \\
&\quad - \rank \begin{bmatrix}
M_{\hat{1},s-1} & M_{\hat{1},s'-1'}
\end{bmatrix}. \label{eq:formula-for-right-2}
\end{align}

The obtained formulas \eqref{eq:formula-for-B}, \eqref{eq:formula-for-upper-00}, \eqref{eq:formula-for-lower}, \eqref{eq:formula-for-left-1}, \eqref{eq:formula-for-left-2}, \eqref{eq:formula-for-right-1}, \eqref{eq:formula-for-right-2} suggest us to consider the essential covering of the bipath poset. As a rough description, it suffices to consider two special subposets of $B_{n,m}$ that are of Dynkin type $\bbA$, and decompose the restricted module in each respective module category. This strategy of decomposing the bipath persistence module can utilize the fast computation of zigzag persistence. Another remarkable advantage is that our strategy does not consider the basis changes at the global minimum and maximum elements, which is the key difference compared with the original decomposition method proposed by Aoki--Escolar--Tada in \cite{aokiBipathPersistence2024}.

Let $Z$ be a poset having the Hasse quiver
\begin{equation}
\label{eq:right-cover}
\begin{tikzcd}[row sep=0pt]
	&\Nname{1} 1 &\Nname{2} 2 &\Nname{3} \cdots &\Nname{4} n \\[8pt]
	 \Nname{ovl0}{\ovl{0}} &&&&&\\
     &&&&& \Nname{h1}{\hat{1}} \\
    \Nname{h0}{\hat{0}}\\[8pt]
	& \Nname{1'}{1'} &\Nname{2'} {2'} &\Nname{3'} \cdots &\Nname{4'} {m'}
	\arrow[from=1, to=2]
	\arrow[from=2, to=3]
	\arrow[from=3, to=4]
	\arrow[from=4, to=h1]
	\arrow[from=ovl0, to=1]
	\arrow[from=h0, to=1']
	\arrow[from=1', to=2']
	\arrow[from=2', to=3']
	\arrow[from=3', to=4']
	\arrow[from=4', to=h1]
\end{tikzcd},
\end{equation}
and define the order-preserving map $\ze\colon Z \to \bfP$ by
$\ze(x)\coloneqq x$ if $x \in Z \setminus \{ \ovl{0}\}$,
and $\ze(\ovl{0})\coloneqq\hat{0}$. Then we have the following.

\begin{prp}
\label{prp:corresponding-Q_1}
Let $\ze$ be an order-preserving map defined above, and let $R$
be the restriction functor induced by $\ze$. Then for every interval $I\in \bbId\sqcup \bbIu\sqcup \bbIr\sqcup \{B_{n, m}\}$ and every $M\in \mod\k[\bfP]$, we have
\begin{equation}
\label{eq:main-eq-1}
d_M(V_I) = \bar{d}_{R(M)}(R(V_I)) = d_{R(M)}(R(V_I)).
\end{equation}
\end{prp}

\begin{proof}
We recall the $\bar{d}$ notation given in Notation~\ref{def:ds-counting}. The second equality of \eqref{eq:main-eq-1} holds since $R(V_I)$ is the indecomposable module in $\mod \k[Z]$ for every interval $I\in \bbId\sqcup \bbIu\sqcup \bbIr\sqcup \{B_{n, m}\}$. It suffices to show that $\ze$ essentially covers every $I \in \bbId\sqcup \bbIu\sqcup \bbIr\sqcup \{B_{n, m}\}$ by Theorem~\ref{thm:ess-cover-equality}.

(i) Let $I = B_{n, m}$. Then it is obvious that there exists a morphism $\bfg'\coloneqq\bmat{p_{\hat{1}, \hat{0}}}$ in $\Ds \k[Z]$ such that $\ze(\bfg') = \bmat{p_{\hat{1}, \hat{0}}}$ in $\Ds \k[\bfP]$, and hence $\ze$ essentially covers $B_{n, m}$ by~\eqref{eq:formula-for-B}.

(ii) Let $I\in \bbId$. This case is trivial by observing \eqref{eq:formula-for-lower} and the definition of $\ze$.

(iii) Let $I\in \bbIu$. Write $I \coloneqq [s,t]$ for some $s, t\in [n]$. All cases are trivial except $s = 1$. If $I = [1, t]$, then by \eqref{eq:formula-for-upper-00} the morphism in $\Ds \k[\bfP]$ can be taken as
$$
\bfg \coloneqq \left[
\begin{array}{c|c}
p_{t+1,1} & \bfzero\\
\hline
p_{t,1} & p_{t, \hat{0}}
\end{array}
\right].
$$
Let
$$
\bfg' \coloneqq \left[
\begin{array}{c|c}
p_{t+1,1} & \bfzero\\
\hline
p_{t,1} & p_{t, \ovl{0}}
\end{array}
\right].
$$
Then $\bfg'$ is the morphism in $\Ds \k[Z]$ satisfying $\ze(\bfg') = \bfg$.

(iv) Let $I\in \bbIr$. This case is also trivial by observing \eqref{eq:formula-for-right-1}, \eqref{eq:formula-for-right-2}, and the definition of $\ze$.

Therefore, the assertion follows.
\end{proof}

\begin{rmk}
    From this proposition, one can easily see that to compute four sub-diagrams $\mathcal{D}(\bbId)$, $\mathcal{D}(\bbIu)$, $\mathcal{D}(\bbIr)$ and $\mathcal{D}(B_{n,m})$ of the bipath persistence diagram (see the precise definition in \cite{aokiBipathPersistence2024}), it suffices to compute the persistence diagram of zigzag poset $Z$ and retrieve the corresponding intervals.
\end{rmk}

\begin{rmk}
The above $\ze$ does not essentially cover the interval $I\in \bbIl$. To interpret this, let us take $I = [\hat{0}, 1]$ as an example. By \eqref{eq:formula-for-left-1} we may take
$$
\bfg \coloneqq \left[
\begin{array}{c}
p_{2,\hat{0}}\\
p_{1',\hat{0}}\\
\hline
p_{1,\hat{0}}
\end{array}
\right],
$$
and by the definition of $\ze$, $\ze(p_{2,\ovl{0}}) = p_{2,\hat{0}}$, $\ze(p_{1',\hat{0}}) = p_{1',\hat{0}}$, $\ze(p_{1,\ovl{0}}) = p_{1,\hat{0}}$. However, we notice the reader that the family
$$
\left(
\begin{array}{c}
p_{2,\ovl{0}}\\
p_{1',\hat{0}}\\
\hline
p_{1,\ovl{0}}
\end{array}
\right)
$$
does not satisfy the matrix condition, and cannot be a morphism in $\Ds \k[Z]$,
and hence $\ze$ does not essentially cover $I = [\hat{0}, 1]$.
\end{rmk}

For those intervals in $\bbIl$, we shall consider another $Z'$ and $\ze'$. Let $Z'$ be a poset having the Hasse quiver
\[\begin{tikzcd}[row sep=0pt]
	&\Nname{1} 1 &\Nname{2} 2 &\Nname{3} \cdots &\Nname{4} n \\[8pt]
	 &&&&& \Nname{h1}{\hat{1}} \\
    \Nname{0}{\hat{0}}\\
    &&&&& \Nname{c1}{\ovl{1}}\\[8pt]
	& \Nname{1'}{1'} &\Nname{2'} {2'} &\Nname{3'} \cdots &\Nname{4'} {m'}
	\arrow[from=1, to=2]
	\arrow[from=2, to=3]
	\arrow[from=3, to=4]
	\arrow[from=4, to=h1]
	\arrow[from=0, to=1]
	\arrow[from=0, to=1']
	\arrow[from=1', to=2']
	\arrow[from=2', to=3']
	\arrow[from=3', to=4']
	\arrow[from=4', to=c1]
\end{tikzcd},\]
and define the order-preserving map $\ze'\colon Z' \to \bfP$ by
$\ze'(x)\coloneqq x$ if $x \in Z \setminus \{ \ovl{1}\}$,
and $\ze'(\ovl{1})\coloneqq\hat{1}$. Then we have the following.

\begin{prp}
\label{prp:corresponding-Q_2}
Let $\ze'$ be an order-preserving map defined above, and let $R'$ be the restriction functor induced by $\ze'$. Then for every interval $I\in \bbId\sqcup \bbIu\sqcup \bbIl\sqcup \{B_{n, m}\}$ and every $M\in \mod\k[\bfP]$, we have
    \begin{equation}
        d_M(V_I) = \bar{d}_{R'(M)}(R'(V_I)) = d_{R'(M)}(R'(V_I)).
    \end{equation}
\end{prp}

\begin{proof}
    It is similar to the proof of Proposition~\ref{prp:corresponding-Q_1} and thus we leave this proof to the reader.
\end{proof}

\begin{rmk}
    From this proposition, one can easily see that to get the remaining sub-diagram $\mathcal{D}(\bbIl)$ of the bipath persistence diagram, it suffices to compute the persistence diagram of zigzag poset $Z'$ and retrieve the corresponding intervals.
\end{rmk}

\subsection*{Summary of the procedure}
To summarize, we propose our decomposition strategy in the following practical data analysis setting. Given a bipath filtration $\mathcal{F}_{\mathrm{bipath}}$ (that is, a functor $\mathcal{F}_{\mathrm{bipath}} \colon B_{n,m}\to \mathrm{Top}$):
\[\begin{tikzcd}[ampersand replacement=\&]
	\&\& {\calK_{1}} \& {\calK_{2}} \& \cdots \& {\calK_{n}} \\
	{\mathcal{F}_{\mathrm{bipath}}:} \& {\calK_{\hat{0}}} \&\&\&\&\& {\calK_{\hat{1}}} \\
	\&\& {\calK_{1'}} \& {\calK_{2'}} \& \cdots \& {\calK_{m'}}
	\arrow[from=1-3, to=1-4]
	\arrow[from=1-4, to=1-5]
	\arrow[from=1-5, to=1-6]
	\arrow[from=1-6, to=2-7]
	\arrow[from=2-2, to=1-3]
	\arrow[from=2-2, to=3-3]
	\arrow[from=3-3, to=3-4]
	\arrow[from=3-4, to=3-5]
	\arrow[from=3-5, to=3-6]
	\arrow[from=3-6, to=2-7]
\end{tikzcd}.\]
We would like to obtain the $q$-th bipath persistence diagram $\mathcal{D}_{q}(\mathcal{F}_{\mathrm{bipath}})\coloneqq \mathcal{D}_{q}(H_q(\mathcal{F}_{\mathrm{bipath}}; \k))$ of $\mathcal{F}_{\mathrm{bipath}}$.

\noindent {\bf Step 1.} Take the information of spaces and maps from $\mathcal{F}_{\mathrm{bipath}}$ and build the following zigzag filtration.
$$
\mathcal{F}_{\mathrm{r}}\coloneqq\hspace{10pt}
\begin{tikzcd}[row sep=0pt]
	&\Nname{1} \calK_1 &\Nname{2} \calK_2 &\Nname{3} \cdots &\Nname{4} \calK_n \\[8pt]
	 \Nname{ovl0}{\calK_{\hat{0}}} &&&&&\\
     &&&&& \Nname{h1}{\calK_{\hat{1}}} \\
    \Nname{h0}{\calK_{\hat{0}}}\\[8pt]
	& \Nname{1'}{\calK_{1'}} &\Nname{2'} {\calK_{2'}} &\Nname{3'} \cdots &\Nname{4'} {\calK_{m'}}
	\arrow[from=1, to=2]
	\arrow[from=2, to=3]
	\arrow[from=3, to=4]
	\arrow[from=4, to=h1]
	\arrow[from=ovl0, to=1]
	\arrow[from=h0, to=1']
	\arrow[from=1', to=2']
	\arrow[from=2', to=3']
	\arrow[from=3', to=4']
	\arrow[from=4', to=h1]
\end{tikzcd}.
$$
We can compute the zigzag persistence diagram $\mathcal{D}_q(\mathcal{F}_{\mathrm{r}})$ by utilizing the fast computation algorithm provided in \cite{deyFastComputationZigzag2022}, and obtain the sub-diagrams $\mathcal{D}_q(\bbId)$, $\mathcal{D}_q(\bbIu)$, $\mathcal{D}_q(\bbIr)$, and $\mathcal{D}_q(B_{n,m})$ based on Proposition~\ref{prp:corresponding-Q_1}.

\noindent {\bf Step 2.} Take the information of spaces and maps from $\mathcal{F}_{\mathrm{bipath}}$ and build the following zigzag filtration.
\[
\mathcal{F}_{l}\coloneqq\hspace{10pt}
\begin{tikzcd}[row sep=0pt]
	&\Nname{1} \calK_1 &\Nname{2} \calK_2 &\Nname{3} \cdots &\Nname{4} \calK_n \\[8pt]
	 &&&&& \Nname{h1}{\calK_{\hat{1}}} \\
    \Nname{0}{\calK_{\hat{0}}}\\
    &&&&& \Nname{c1}{\calK_{\hat{1}}}\\[8pt]
	& \Nname{1'}{\calK_{1'}} &\Nname{2'} {\calK_{2'}} &\Nname{3'} \cdots &\Nname{4'} {\calK_{m'}}
	\arrow[from=1, to=2]
	\arrow[from=2, to=3]
	\arrow[from=3, to=4]
	\arrow[from=4, to=h1]
	\arrow[from=0, to=1]
	\arrow[from=0, to=1']
	\arrow[from=1', to=2']
	\arrow[from=2', to=3']
	\arrow[from=3', to=4']
	\arrow[from=4', to=c1]
\end{tikzcd}.\]
Again, we compute the zigzag persistence diagram $\mathcal{D}_q(\mathcal{F}_{l})$ by utilizing the fast computation provided in \cite{deyFastComputationZigzag2022}, and obtain the remaining sub-diagram $\mathcal{D}_q(\bbIl)$ based on Proposition~\ref{prp:corresponding-Q_2}.

    \begin{rmk}
    \label{rmk:snake}
    It is possible to compute the bipath persistence diagram by computing the persistence diagram of a single zigzag filtration. Without loss of generality, we now assume the number of vertices in the lower path (i.e., $m$) is not greater than that in the upper path (i.e., $n$) of $B_{n,m}$, in terms of the usual total order of $\bbZ$.

Let $S$ be a poset having the Hasse quiver:
\begin{equation}
\label{eq:uni-cover}
\begin{tikzcd}[ampersand replacement=\&]
	\& {1''} \& {2''} \& \cdots \& {m''} \\
	{\tilde{0}} \\
	\& 1 \& 2 \& \cdots \& n \\
	{\hat{0}} \&\&\&\&\& {\hat{1}} \\
	\& {1'} \& {2'} \& \cdots \& {m'}
	\arrow[from=1-2, to=1-3]
	\arrow[from=1-3, to=1-4]
	\arrow[from=1-4, to=1-5]
	\arrow[from=2-1, to=1-2]
	\arrow[from=2-1, to=3-2]
	\arrow[from=3-2, to=3-3]
	\arrow[from=3-3, to=3-4]
	\arrow[from=3-4, to=3-5]
	\arrow[from=3-5, to=4-6]
	\arrow[from=4-1, to=5-2]
	\arrow[from=5-2, to=5-3]
	\arrow[from=5-3, to=5-4]
	\arrow[from=5-4, to=5-5]
	\arrow[from=5-5, to=4-6]
\end{tikzcd},
\end{equation}
where we set $[m]''\coloneqq \{1'',\dots, m''\}$. Now, we define the order-preserving map $\ze\colon S \to \bfP$ by setting
\begin{align}
\ze (x)\coloneqq\begin{cases}
x, & \textnormal{if } x\in [n]\sqcup [m]' \sqcup \{\hat{0}, \hat{1}\},\\
\hat{0}, & \textnormal{if } x=\tilde{0},\\
i', & \textnormal{if } x\coloneqq i''\in [m]''.
\end{cases}
\end{align}
Then it is straightforward to check that $\ze$ essentially covers all intervals of $B_{n, m}$, and hence we have
\begin{equation}
\label{eq:main-eq-uni}
d_M(V_{I}) = \bar{d}_{R(M)}(R(V_{I})).
\end{equation}
We remark here that in this case, $\bar{d}$ appearing in~\eqref{eq:main-eq-uni} can not be changed to the usual symbol $d$ for the multiplicity since $R(V_{I})$ may not be indecomposable in $\mod \k[S]$. For instance, considering an interval $I = [s', t']\in \bbId$ of $B_{n, m}$.
\end{rmk}

\section*{Declarations}
\begin{itemize}
\item Funding: H.A.\ is partially supported by JSPS Grant-in-Aid for Scientific Research (C) 18K03207, 25K06922, JSPS Grant-in-Aid for Transformative Research Areas (A) (22A201), and by Osaka Central Advanced Mathematical Institute
(MEXT Promotion of Distinctive Joint Research Center Program JPMXP0723833165). E.L.\ was supported by JST SPRING, Grant Number JPMJSP2110.
\item Conflict of interest: the authors declare that they have no relevant financial or non-financial conflicts of interest with regard to the content of this article.
\end{itemize}

\begin{appendices}
\section{The salamander lemma}
\label{sect:app}
We apply the salamander lemma in this paper,
for which we refer the reader to papers \cite{bergmanDiagramchasingDoubleComplexes2012} and \cite{geraschenkoAntonGeraschenkoSalamander2007} by Bergman and Geraschenko, respectively.
In particular, we use the notations introduced by Geraschenko.
Here we recall some necessary definitions and statements.

\begin{dfn}
\label{dfn:double-cpx}
A \emph{double complex} in an abelian category $\calA$ is a complex of complexes,
i.e., a family $X = (X_{i,j}, d^H_{i,j}, d^V_{i,j})_{(i,j) \in \bbZ^2}$ of objects $X_{i,j}$ and morphisms $d^H_{i,j} \colon X_{i,j} \to X_{i,j+1}$,
$d^V_{i,j} \colon X_{i,j} \to X_{i+1,j}$,
which satisfy the zero relations $d^H_{i,j+1}d^H_{i,j} = 0$,
$d^V_{i+1,j}d^V_{i,j} = 0$,
and the full commutativity relations
$(d^D_{i,j}\coloneqq )\, d^V_{i,j+1}d^H_{i,j} = d^H_{i+1,j}d^V_{i,j}$  for all $i,j \in \bbZ$.
We usually draw $d^H_{i,j}$ from the left to the right, and
$d^V_{i,j}$ downward in the diagram as in
$$
\begin{tikzcd}
{X_{i-1,j-1}} & {X_{i-1,j}} \\
{X_{i,j-1}} & {X_{i,j}} & {X_{i,j+1}} \\
& {X_{i+1,j}} & {X_{i+1,j+1}}
\arrow["{d^D_{i-1,j-1}}" near start, from=1-1, to=2-2]
\arrow["{d^V_{i-1,j}}", from=1-2, to=2-2]
\arrow["{d^H_{i,j-1}}", from=2-1, to=2-2]
\arrow["{d^H_{i,j}}", from=2-2, to=2-3]
\arrow["{d^V_{i,j}}", from=2-2, to=3-2]
\arrow["{d^D_{i,j}}" near end, from=2-2, to=3-3]
\end{tikzcd}.
$$
When we have a finite double complex, then we always extend it by adding zeros.

Here we define four homologies at $A\coloneqq X_{i,j}$ for each $(i,j) \in \bbZ^2$:
$$
\begin{gathered}
{}_=A\coloneqq \Ker d^H_{i,j}/\Im d^H_{i,j-1},\ A^\parallel\coloneqq \Ker d^V_{i,j}/\Im d^V_{i-1,j},\\
{}^\square A\coloneqq (\Ker d^H_{i,j} \cap \Ker d^V_{i,j})/\Im d^D_{i-1, d-1},\
A_\square\coloneqq \Ker d^D_{i,j}/(\Im d^H_{i-1,j} + \Im d^V_{i,j-1}),
\end{gathered}
$$
which are called the {\em horizontal homology}, the {\em vertical homology},
the {\em receptor} and the {\em donor}, respectively.

Inclusion morphisms induce canonical morphisms
$$
\begin{tikzcd}
{}^\square A & A^\parallel\\
{}_=A & A_\square
\Ar{1-1}{1-2}{}
\Ar{1-2}{2-2}{}
\Ar{1-1}{2-1}{}
\Ar{2-1}{2-2}{}
\end{tikzcd},
$$
which are called {\em intramural} morphisms, and a horizontal arrow (or a vertical arrow) $A \to B$ in the double complex induces a canonical morphism $A_\square \to {}^\square B$, called an {\em extramural} morphism.
\end{dfn}

\begin{prp}[The salamander lemma]
\label{prp:salamander}
Let $C \ya{f} A \ya{g} B \ya{h} D$ be a path in a double complex, where
both $f$ and $h$ are horizontal (resp.\ vertical) and $g$ is a vertical
(resp.\ horizontal) arrow.
Then there exists an exact sequence
$$
\begin{gathered}
C_\square \ya{a} A^\parallel \ya{b} A_\square \ya{c} {}^\square B
\ya{d} B^\parallel \ya{e} {}^\square D\\
(\text{resp.\ }
C_\square \ya{a} {}_=A \ya{b} A_\square \ya{c} {}^\square B
\ya{d} {}_=B \ya{e} {}^\square D),
\end{gathered}
$$
where $a$ is the composite of an extramural and an intramural,
$b, d$ are intramurals, $c$ is an extramural and $e$ is the composite of
an intramural and an extramural.
\end{prp}

The following two corollaries are immediate from the salamander lemma.

\begin{cor}
\label{cor:salamander-extra}
Assume the setting above.
If $A^\parallel = 0 = B^\parallel$
(resp.\ ${}_=A = 0 = {}_=B$), then the extramural $A_\square \ya{c} {}^\square B$
is an isomorphism. \qed
\end{cor}

\begin{cor}
\label{cor:salamander-corner}
Let $X$ be a complex in Definition \ref{dfn:double-cpx}, and
$A\coloneqq X_{i,j}$ for some $(i,j) \in \bbZ^2$.
Set $(C,B)\coloneqq (X_{i-1,j}, X_{i,j+1})$ $($resp.\ $(X_{i,j-1}, X_{i+1,j})$$)$
so that we have the following situation $($at first ignore homology signs$)$:
$$
\begin{tikzcd}
\Nname{C}C_\square\\
\Nname{A}{}^\square _=A_\square^\parallel &\Nname{B} {}^\square B
\Ar{C}{A}{}
\Ar{A}{B}{}
\Ar{A}{A}{dashed, no head,rounded corners, to path={([xshift=10pt,yshift=-5pt]A.north west)--
([yshift=-2pt]A.north)--([xshift=-10pt,yshift=-5pt]A.north east)}}
\Ar{A}{A}{dashed, no head,rounded corners, to path={([xshift=10pt,yshift=5pt]A.south west)--
([yshift=-2pt]A.south)--([xshift=-10pt,yshift=5pt]A.south east)}}
\end{tikzcd}
\quad
\left(\text{resp.\ }
\begin{tikzcd}
C_\square & \Nname{A}{}^\square_=A_\square^\parallel\\
 & {}^\square B
\Ar{1-1}{1-2}{}
\Ar{1-2}{2-2}{}
\Ar{A}{A}{dashed, no head,rounded corners, to path={([xshift=5pt,yshift=-6pt]A.north west)--
([xshift=2pt]A.west)--([xshift=5pt,yshift=5pt]A.south west)}}
\Ar{A}{A}{dashed, no head,rounded corners, to path={([xshift=-5pt,yshift=-6pt]A.north east)--
([xshift=-2pt]A.east)--([xshift=-5pt,yshift=6pt]A.south east)}}
\end{tikzcd}.
\right)
$$
If $C_\square = 0 = {}^\square B$, then we have
${}^\square A \iso A^\parallel$ and ${}_=A \iso A_\square$
$($i.e., the horizontal dashed lines become isomorphisms$)$
$($resp. ${}^\square A \iso {}_=A$ and $A^\parallel \iso A_\square\ ($i.e. the vertical dashed lines become isomorphisms$))$.
\qed
\end{cor}
\end{appendices}

\bibliography{sn-bibliography}
\end{document}